\documentclass[12pt, twoside]{amsart}
\usepackage[a4paper, width=150mm, top=25mm, bottom=25mm, hmarginratio=1:1]{geometry}
\usepackage{fancyhdr}
\usepackage{setspace}

\usepackage[colorinlistoftodos, bordercolor=xkcdForestGreen, backgroundcolor=xkcdForestGreen!20, linecolor=xkcdForestGreen, textsize=scriptsize]{todonotes}
\usepackage[latin1, utf8]{inputenc}
\usepackage{amsbsy}
\usepackage{amscd} 
\usepackage{amsfonts}
\usepackage{amsgen} 
\usepackage{amsmath}
\usepackage{amsopn} 
\usepackage{amssymb}
\usepackage{amstext}
\usepackage{amsthm} 
\usepackage{amsxtra}
\usepackage{breakurl}
\usepackage{hyperref}
 \hypersetup{
           breaklinks=true,  
           colorlinks=true,
           pdfusetitle=true, 
        }

\usepackage[titlenumbered, linesnumbered, ruled, vlined]{algorithm2e}
\usepackage{booktabs}
\usepackage[english]{babel}
\usepackage[sorting=anyt, style=nature, maxbibnames=99, backend=biber]{biblatex}
\usepackage{csquotes}
\usepackage{setspace}
\usepackage{physics}
\usepackage{enumitem}
\usepackage{xkcdcolors}
\usepackage[page, header, toc]{appendix}
\usepackage{float} 

\usepackage{tikz}
\usetikzlibrary{positioning, calc}
\usepackage{ifthen}

\usepackage{tabularx}
\usepackage{makecell}

\swapnumbers
\theoremstyle{plain} 
\newtheorem{thm}{Theorem}[section] 

\newtheorem{theorem*}{Main Theorem}
\newtheorem{obs*}{Observation}

\newtheorem*{thm*}{Main Theorem}
\newtheorem{prop}[thm]{Proposition}
\newtheorem{lem}[thm]{Lemma}
\newtheorem{propr}[thm]{Properties}
\newtheorem{cor}[thm]{Corollary}
\theoremstyle{definition}
\newtheorem{defn}[thm]{Definition}
\newtheorem{definition}[thm]{Definition}

\newtheorem{ex}[thm]{Example}

\newtheorem{notation}[thm]{Notation}

\numberwithin{equation}{section}

\renewcommand{\theta}{\vartheta}
\renewcommand{\phi}{\varphi}
\renewcommand{\epsilon}{\varepsilon}

\renewcommand{\labelenumi}{(\roman{enumi})}
\renewcommand{\theenumi}{(\roman{enumi})}

\newcommand{\N}{\mathbb N}
\newcommand{\Z}{\mathbb Z}

\newcommand{\R}{\mathbb R}
\newcommand{\C}{\mathbb C}

\makeatletter
\def\moverlay{\mathpalette\mov@rlay}
\def\mov@rlay#1#2{\leavevmode\vtop{%
   \baselineskip\z@skip \lineskiplimit-\maxdimen
   \ialign{\hfil$\m@th#1##$\hfil\cr#2\crcr}}}
\newcommand{\charfusion}[3][\mathord]{
    #1{\ifx#1\mathop\vphantom{#2}\fi
        \mathpalette\mov@rlay{#2\cr#3}
      }
    \ifx#1\mathop\expandafter\displaylimits\fi}
\makeatother

\DeclareMathOperator{\Aut}{Aut}

\DeclareMathOperator{\CN}{CN}

\DeclareMathOperator{\Line}{L}
\DeclareMathOperator{\Trunc}{Trunc}
\DeclareMathOperator{\Antip}{Antip}
\newcommand{\QBan}{G_{aut}^+}
\newcommand{\QBic}{G_{aut}^*}

\newlist{pfsteps}{enumerate}{3}
\setlist[pfsteps,1]{%
  label=\sffamily{Step {\arabic*}:},
  ref=\normalfont{\arabic*},
  wide,itemsep=0pt,topsep=0pt
}

\author{Julien Schanz}
\title{Quantum Symmetries of Vertex-Transitive Graphs on 12 Vertices}
\begin{document}
\begin{abstract}
    Recently, the work on quantum automorphism groups of graphs has seen renewed progress, which we expand in this paper.
    Quantum symmetry is a richer notion of symmetry than the classical symmetries of a graph. In general, it is non-trivial to decide
    whether a given graph does have quantum symmetries or not. For vertex-transitive graphs, the quantum symmetries have already
    been determined in earlier work 
    on up to $11$ and on  $13$ vertices. This paper fills the gap by determining for all vertex-transitive graphs on  $12$ vertices, whether they have
    quantum symmetries and for most of these graphs we also give their quantum automorphism group explicitly.
\end{abstract}
\maketitle
\setcounter{tocdepth}{1}
\tableofcontents

\section*{Introduction}
Quantum automorphism groups of graphs were first studied by Bichon in~\cite{bichon_2003} and by 
Banica in~\cite{banica_2005}.
They are compact matrix quantum groups, as introduced by Woronowicz in~\cite{woronowicz_1987, woronowicz_1991},
in particular they are quantum subgroups of the quantum symmetric group $S_n^+$, which was first described by Wang in~\cite{wang_1998}.
If $n \ge 4$, the quantum symmetric group $S_n^+$ is bigger than the classical symmetric group  $S_n$, and thus
for graphs on more than  $4$ vertices, also the quantum automorphism group is potentially bigger than the 
classical automorphism group, i.e. we have
\[
    G_{aut}(\Gamma) \subseteq \QBan(\Gamma)
\] 
for any graph $\Gamma$.

Previous work on quantum groups of graphs includes a computer based approach by Eder et al. in~\cite{eder_et_al_2022},
where the authors compute for all connected, simple graphs on up to $6$ vertices and all connected, simple graphs 
on  $7$ vertices with classical automorphism group of order less than  $2$, whether they have quantum symmetries.
Moreover, in~\cite{vanDobbendeBruyn_et_al_2023_1}, van Dobben de Bruyn et al. give a complete characterisation of 
quantum automorphism groups of trees, including a polynomial-time algorithm to explicitly compute the 
quantum automorphism group of a given tree. Some of the same authors also give the first construction of 
a graph that has trivial automorphism group but non-trivial quantum automorphism group 
in~\cite{vanDobbendeBruyn_et_al_2023_2}.
In~\cite{gromada_2022}, Gromada defines and investigates quantum automorphisms of Hadamard matrices and shows 
that these results also pass to the corresponding Hadamard graphs. The study of quantum symmetries of graphs 
has also been extended to the study of quantum symmetries of quantum graphs in~\cite{brannan_et_al_2019, chirvasitu_wasilewski_2022}.

A particular class of graphs that is of interest are the vertex-transitive graphs. They are interesting, because the
vertex-transitivity ensures that the graphs have at least a certain amount of symmetry.
In~\cite{banica_bichon_2006_1}, Banica and Bichon described the quantum automorphism goups of all vertex-transitive graphs on
up to $11$ vertices except for the Petersen graph, for which Schmidt showed in~\cite{schmidt_2018} that it does not have quantum
symmetries. Moreover, Chassaniol described the quantum automorphism groups of all vertex-transitive graphs on $13$ vertices
in~\cite{chassaniol_2016, chassaniol_2019}.
This paper fills the gap by investigating all vertex-transitive graphs on  $12$ vertices to see which of these have quantum
symmetries and which do not, partially answering a question posed in the PhD thesis of Schmidt~\cite{schmidt_2020_1}.

We sort the vertex-transitive graphs on $12$ vertices in  $5$ subclasses: \hyperref[section:non_connected]{disconnected graphs},
\hyperref[section:products]{products of smaller graphs},
\hyperref[section:circulant_graphs]{circulant graphs}, \hyperref[section:semicirculant_graphs]{semicirculant graphs}
and \hyperref[section:special_cases]{special cases} that do not fit into any of the other subclasses.

In total, there are $74$ vertex-transitive graphs on  $12$ vertices. However, since the quantum automorphism group of 
a graph is the same as the one of its complement, one only needs to consider one of the two. Due to this, 
we study only  $37$ graphs. 

Our main result is the continuation of the classification programme of quantum automorphism groups of graphs.
We also observe that the existence of disjoint automorphisms is equivalent to a graph having quantum symmetries for 
vertex-transitive graphs on up to $13$ vertices, which is not the case in general.

In Section~\ref{section:preliminaries}, we give the basic definitions around quantum automorphism groups of graphs.
In Section~\ref{section:useful_lemmas}, we present some lemmas, that are useful when computing the quantum 
automorphism group of a graph. In Section~\ref{section:overview_and_notation}, we give an overview of the graphs that are
considered in this paper and introduce notation to help with some of the proofs.
In Sections~\ref{section:constructions_from_smaller_graphs}, \ref{section:circulant_and_semi_circulant_graphs} and~\ref{section:special_cases}, we give the results on the
existence of quantum symmetries for all the different groups of vertex transitive graphs and then compute the quantum
automorphism group of two such graphs explicitly in Section~\ref{section:computation_of_quantum_automorphism_groups}. In Appendix~\ref{appendix:omitted-proofs}, we collect the
proofs omitted in Section~\ref{section:circulant_and_semi_circulant_graphs} about computing the quantum symmetries of 
the graphs. In Appendix~\ref{appendix:computing_non_commutative_groebner_bases_in_oscar}, we give a short overview
of the Groebner basis computation in OSCAR and its limitations, that was used for some results in this paper.

\section*{Main Results}
The main result of this paper is that we have decided the existence of quantum symmetries of vertex-transitive graphs on $12$ vertices.
Combining this with previous results of Banica and Bichon in~\cite{banica_bichon_2006_1}, Schmidt in~\cite{schmidt_2018} and Chassaniol in~\cite{chassaniol_2016, chassaniol_2019}, we 
get the following:
 \begin{theorem*}
    For vertex transitive graphs on up to $13$ vertices, the existence of quantum symmetries has been completely determined.
\end{theorem*}

Of the $37$ graphs we studied,  $20$ have quantum symmetries. Among those with quantum symmetries there are all graphs
that are disjoint copies of smaller graphs, some graph products, three circulant and three semi-circulant graphs.
\begin{table}[H]
    \centering
    \begin{tabular}{l|l|l}
        subclass & total graphs & graphs with quantum symmetries\\
        \hline
        disconnected graphs & $9$ &  $9$\\
        products of smaller graphs &  $6$ &  $5$ \\
        circulant graphs & $12$ &  $3$\\
        semi-circulant graphs &  $5$ &  $3$\\ 
        special cases &  $5$ &  $0$
    \end{tabular}
\end{table}
In the following table, we recollect the previous results on quantum symmetries of vertex-transitive graphs by Banica and 
Bichon in~\cite{banica_bichon_2006_1} and by Chassaniol in~\cite{chassaniol_2016, chassaniol_2019}. Also included is the result that the Petersen
graph has no quantum symmetries by Schmidt~\cite{schmidt_2018}. These are all vertex transitive graphs of the 
given order up to complements.
We exclude the trivial cases of the full graphs and of disjoint copies of smaller graphs, due to the 
fact that the full graph on $n$ vertices always has quantum automorphism group  $S_n^+$, which can be seen easily,
and due to the result of Banica and Bichon in~\cite{banica_bichon_2006_2}, that a graph consisting of  $n$ disjoint
copies of a connected graph  $\Gamma$ has quantum automorphism group
     \[
        \QBan(n\Gamma) = \QBan(\Gamma) \wr_{*} S_{n}^+
    .\] 

\begin{table}[H]
    \centering
    \caption{Vertex-transitive graphs on $\le 11$ and on $13$ vertices}
    \begin{tabular}{l|l|l|l}
        Order & Graph& Automorphism Group& Quantum Automorphism Group\\
        \hline
        $n\neq 4$ &  $C_n$ & $D_n$ & $D_n$\\
        $8$ &  $C_8(4)$ & $D_8$ & $D_8$\\
        $8$ &  $K_2 \square C_4$ & $S_4 \times \Z_2$ & $S_4^+ \times \Z_2$\\
        $9$ &  $C_9(3)$ & $D_9$ & $D_9$\\
        $9$ &  $K_3 \square K_3$ & $S_3 \wr \Z_2$ & $S_3 \wr \Z_2$\\
        $10$ &  $C_{10}(2)$ &$D_{10}$ & $D_{10}$\\
        $10$ &  $C_{10}(5)$ &$D_{10}$ & $D_{10}$\\
        $10$ &  $K_2 \square C_{5}$ &$D_{10}$ & $D_{10}$\\
        $10$ &  Petersen &  $S_5$ & $S_5$\\
        $10$ &  $K_2 \square K_5 $ & $S_5 \times \Z_2$ & $S_5^+ \times \Z_2$\\
        $10$ &  $C_{10}(4)$ & $\Z_2 \wr D_5$ & $\Z_2 \wr_{*} D_5$ \\
        $11$ &  $C_{11}(2)$ & $D_{11}$ & $D_{11}$\\
        $11$ &  $C_{11}(3)$ & $D_{11}$ & $D_{11}$\\
        $13$ &  $C_{13}$ & $D_{13}$ & $D_{13}$\\
        $13$ &  $C_{13}(2)$ & $D_{13}$ & $D_{13}$\\
        $13$ &  $C_{13}(2, 5)$ & $D_{13}$ & $D_{13}$\\
        $13$ &  $C_{13}(2, 6)$ & $D_{13}$ & $D_{13}$\\
        $13$ &  $C_{13}(3)$ & $D_{13}$ & $D_{13}$\\
        $13$ &  $C_{13}(5)$ & $\Z_{13} \rtimes \Z_4$ & $\Z_{13} \rtimes \Z_4$\\
        $13$ &  $C_{13}(3, 4)$ & $\Z_{13} \rtimes \Z_6$ & $\Z_{13} \rtimes \Z_6$
    \end{tabular}
\end{table}
Next, we give a table of all vertex-transitive graphs on $12$ vertices up to complements with their automorphism 
groups and their quantum automorphism groups, if it is known. For completeness sake, we include the trivial 
cases in this table. 
In particular, we have computed for all of the following graphs, whether they have quantum symmetries. Therefore,
whenever the column with the quantum automorphism group is a question mark, the graph in question does have quantum 
symmetries, but we did not manage to find the actual quantum automorphism group.

\begin{table}[H]
    \centering
    \caption{Vertex-transitive graphs on $12$ vertices}
    \begin{tabular}{l|l|l}
        Graph& Automorphism Group& Quantum Automorphism Group\\
        \hline
        $6 K_2$ & $\Z_2 \wr S_6$ & $\Z_2 \wr_* S_6^+$\\
        $4 K_3$ & $S_3 \wr S_4$ & $S_3 \wr_* S_4^+$\\
        $3 K_4$ & $S_4 \wr S_3$ & $S_4^+ \wr_* S_3$\\
        $3C_4$ & $H_2 \wr S_3$ & $H_2^+ \wr_* S_3$ \\
        $2 K_6$ & $S_6 \wr \Z_2$ & $S_6^+ \wr_* \Z_2$\\
        $2 C_6$ & $D_6 \wr \Z_2$ & $D_6 \wr_* \Z_2$\\
        $2(K_2 \square K_3)$ & $D_{6} \wr \Z_2$ & $D_{6} \wr_* \Z_2$\\
         $2 C_6(2)$ &  $(\Z_2 \wr S_3) \wr \Z_2$ & $(\Z_2 \wr_* S_3) \wr_* \Z_2$\\
        $2 C_6(3)$ & $(S_3 \wr \Z_2) \wr \Z_2$ & $(S_3 \wr_* \Z_2) \wr_* \Z_2$\\
        $K_6 \times K_2$& $S_6 \times \Z_2$ & $S_6^+ \times \Z_2$\\
        $K_3 \times K_4$ & $S_3 \times S_4$ & $S_3 \times S_4^+$\\
        $C_4 \square C_3$ & $H_2 \times S_3$ & $H_2^+ \times S_3$\\
        $K_2 \square C_6(3) $ & $\Z_2 \times (S_3 \wr \Z_2)$ & $\Z_2 \times (S_3 \wr_* \Z_2)$\\
        $K_2 \square C_6 $ & $\Z_2 \times D_6$ & $\Z_2 \times D_6$\\ 
        $K_2 \square C_6(2) $ & $\Z_2 \times (\Z_2 \wr S_3)$& ?\\
        $C_{12}$ & $D_{12}$ & $D_{12}$\\
         $C_{12}(3)$ & $D_{12}$ & $D_{12}$\\
         $C_{12}(6)$ & $D_{12}$ & $D_{12}$\\
         $K_{12}$ & $S_{12}$ & $S_{12}^+$\\
         $C_{12}(5)$ & $\mathcal{A}$&?\\
         $C_{12}(4, 5)$ & $H_2 \times S_3$ & $H_2^+ \times S_3$\\
         $C_{12}(5, 6)$ & $\mathcal{A}$ & ?\\
         $C_{12}(2)$& $D_{12}$ & $D_{12}$\\
         $C_{12}(4)$& $D_{12}$ & $D_{12}$\\
         $C_{12}(2, 6)$&$D_{12}$&$D_{12}$\\
         $C_{12}(3, 6)$&$D_{12}$&$D_{12}$\\
         $C_{12}(4, 6)$&$D_{12}$&$D_{12}$\\
         $C_{12}(5^+)$ & $H_3$ & ?\\
         $C_{12}(3^+, 6)$ & $H_2 \times S_3$ & $H_2^+ \times S_3$\\
         $C_{12}(5^+, 6)$ & $H_3$ & ?\\
         $C_{12}(2, 5^+)$&$D_6$& $D_6$\\
         $C_{12}(4, 5^+)$&$D_6$& $D_6$\\
          Icosahedral Graph & $\Z_2 \times A_5$ & $\Z_2 \times A_5$\\ 
         $L(C_6(2))$ & $H_3$ & $H_3$\\
         $Trunc(K_4)$ &$S_4$ & $S_4$\\
         $Antip(Trunc(K_4))$ &$S_4$ & $S_4$\\ 
         Cuboctahedral Graph & $H_3$ & $H_3$
    \end{tabular}
\end{table}
Here, we call
 \[
    \mathcal{A} := (\Z_2 \times ((((\Z_2 \times \Z_2 \times \Z_2 \times \Z_2)\rtimes \Z_2)\rtimes \Z_3)\rtimes \Z_2)) \rtimes \Z_2
.\] 
Looking at the data in the above tables, one can make some observations.
\begin{obs*}
    \label{obs:observation1}
    For vertex transitive graphs on up to $13$ vertices, the existence of quantum symmetries is equivalent to the 
    existence of disjoint automorphisms.
\end{obs*}
It is already known that the statement from Observation~\ref{obs:observation1} does not hold for all 
graphs, as has been shown in~\cite{schmidt_2020_1}, but it is an open question 
for what families of graphs this is an alternative characterisation of the existence of quantum symmetries.

We moreover observe the following fact.
\begin{obs*}
    \label{obs:observation2}
    Let $\Gamma$ be a vertex-transitive graph on up to  $13$ vertices. If the automorphism
    group of $\Gamma$ is  $D_n$ with  $n \neq 4$, then $\Gamma$ does not have quantum symmetries.
\end{obs*}
Due to the result of van Dobben de Bruyn et al in~\cite{vanDobbendeBruyn_et_al_2023_2} that there exists a graph 
with quantum symmetry but with trivial automorphism group, we know that a question of the form from the above theorem
does not make sense for general graphs: calling the graph constructed in~\cite{vanDobbendeBruyn_et_al_2023_2} $\Gamma_{0}$ and
given a graph $\Gamma$ with any automorphism group, one could simply take the disjoint union 
of $\Gamma$ and would get a graph with the same automorphism group
as  $\Gamma$ (as long as $\Gamma$ is not the same as $\Gamma_{0}$) that now has quantum symmetries.
In fact, in their paper they even state a more general result, namely that for any given group, there exist both graphs with
that group as an automorphism group that do have quantum symmetries and graphs that do not have quantum symmetries.
However, one could investigate this question only for certain classes of graphs, such as vertex-transitive graphs.
Since $\Gamma_{0}$ is asymmetric, it is necessarily also not vertex-transitive and thus any construction with $\Gamma_{0}$ 
is also not vertex-transitive. Inspired by Observation~\ref{obs:observation2}, one could therefore ask the question:
Does any group $G$ exist such that whenever a vertex-transitive graph $\Gamma$ has automorphism group  $G$, then  $\Gamma$
has no quantum symmetries?

In addition to computing whether the graphs have quantum symmetries and giving the explicit quantum automorphism
group wherever it was known by some previous result, we also compute the quantum automorphism group of two graphs for
which it was not yet known in Section~\ref{section:computation_qaut_groups}.

\section*{Acknowledgements}
I want to thank Nicolas Faroß, Luca Junk, Simon Schmidt and Moritz Weber for numerous fruitful discussions.

This article is part of the author’s PhD thesis, under the supervision of Moritz Weber and Simon Schmidt. The author was supported by the SFB-TRR 195.

\section{Preliminaries}
\label{section:preliminaries}
\begin{definition}
    A \emph{finite graph without multiple edges} $\Gamma$ consists of a finite vertex set
    $V(\Gamma)$ and an edge set  $E(\Gamma) \subseteq V(\Gamma) \times V(\Gamma)$.
    It is called \emph{undirected}, if for any edge $(u, v) \in E(\Gamma)$ it also holds that
    $(v, u) \in E(\Gamma)$.
    A \emph{loop} is an edge of the form $(i, i) \in E(\Gamma)$.

    The \emph{adjacency matrix} $A_{\Gamma}$ of a graph on $n$ vertices is the $n \times n$-matrix 
    with the $(i, j)$-entry being the number of edges form vertex $i$ to vertex  $j$.
    If $\Gamma$ is undirected and without multiple edges,  $A_{\Gamma}$ is thus a symmetric matrix
    with $\{0, 1\}$-entries.

    In this paper, we only consider undirected finite graphs without multiple edges and without loops.
    Since all graphs will be undirected, we will write $u \sim v$ to mean that  $(u, v) \in E(\Gamma)$
    and thus also $(v, u) \in E(\Gamma)$.
\end{definition}

\begin{definition}
    Given a graph $\Gamma$ on  $n$ vertices with adjacency matrix $A_{\Gamma}$, its automorphism group  $G_{aut}(\Gamma)$ is the group of all 
    automorphisms of $\Gamma$ and can be expressed as a group of permutation matrices in the following sense:
     \[
        G_{aut}(\Gamma) = \{P \in S_{n} | P A_{\Gamma} = A_{\Gamma} P\} 
    .\] 
\end{definition}

\begin{definition}
    A graph $\Gamma$ is \emph{vertex-transitive} if for any two vertices  $i, j \in V(\Gamma)$
    there exists an automorphism $\phi \in G_{aut}(\Gamma)$ such that
    $\phi(i) = j$.
\end{definition}

\begin{definition}
    Given a graph $\Gamma$ and two automorphisms  $\phi$ and  $\psi$ of  $\Gamma$, we say that 
    $\phi$ and  $\psi$ are \emph{disjoint} if we have
     \[
        \phi(i) \neq i \implies \psi(i) = i \text{ and } \psi(i) \neq i \implies \phi(i) = i
    \] 
    for all vertices $i \in V(\Gamma)$.

    We say that $\Gamma$ \emph{has disjoint automorphisms}, if there are two non-trivial disjoint automorphisms
    in  $G_{aut}(\Gamma)$.
\end{definition}

\begin{definition}
    \label{def:cmqg}
    A \emph{compact matrix quantum group} $G$ as defined by 
    Woronowicz~\cite{woronowicz_1987, woronowicz_1991} in 1987
    is a unital $C^*$-algebra $C(G)$ equipped with a $*$-homomorphism 
    $\Delta: C(G) \rightarrow C(G) \otimes C(G)$ and a unitary $u \in M_n(C(G))$, $n\in \N$, such that
    \begin{enumerate}
            \item $\Delta(u_{ij}) = \sum_k u_{ik} \otimes u_{kj}$ for all $i, j$\\
            \item $\bar{u}$ is an invertible matrix\\
            \item the elements $u_{ij}$ $(1 \leq i, j \leq n)$ generate $C(G)$ (as a $C^*$-algebra).\\
    \end{enumerate}
    The unitary $u$ is called the \emph{fundamental corepresentation (matrix)} of $(C(G), \Delta, u)$.\\
    Since (i) and (iii) uniquely determine $\Delta$, one can also refer to the pair $(C(G), u)$ as a compact matrix quantum group.\\
    If $G = \left( C(G), u \right)$ and $H = \left( C(H), v \right)$ are compact matrix quantum groups with $u~\in~M_n(C(G))$ and $v \in M_n(C(H))$, we say that $G$ is a 
    \emph{compact matrix quantum subgroup} of $H$, if there is a surjective $^{*}$-isomorphism from $C\left( H \right)$ to $C\left( G \right)$ mapping generators to generators.
    We then write $G \subseteq H$. If we have $G \subseteq H$ and $H \subseteq G$, they are said to be equal as compact matrix quantum groups.
\end{definition}

\begin{definition}
    \label{def:snplus}
    The \emph{quantum symmetric group} $S_n^+ = (C(S_n^+), u)$, which was first defined by 
    Wang~\cite{wang_1998} in 1998, is the compact matrix quantum group given by
    \begin{displaymath}
            C\left(S_n^{+}\right) := C^{*}\left(u_{ij} | u_{ij} = u_{ij}^{*} = u_{ij}^2, \sum_{k=1}^{n} u_{ik} = \sum_{k=1}^n u_{ki} = 1 \forall i, j = 1,\cdots, n\right)
    \end{displaymath}
    and the $*$-homomorphism $\Delta$ is given by 
    \begin{displaymath}
            \Delta(u_{ij}) := u_{ij}^{'} := \sum_{k}u_{ik} \otimes u_{kj}.
    \end{displaymath}
    It can be shown that the quotient of $C\left(S_n^{+}\right)$ by the relation that all $u_{ij}$ commute is exactly $C\left(S_n\right)$.
    Moreover, $S_n$ can be seen as a compact matrix quantum group $S_n = \left( C\left( S_n \right), u \right)$, where $u_{ij}: S_n \rightarrow \C$ are the evaluation maps 
    of the matrix entries. We then have $S_n \subseteq S_n^{+}$ as compact matrix quantum groups, 
    which justifies the name ``quantum symmetric group''.
    
    For $n < 4$, one can see that the relations of $S_n^+$ already imply the commutation of all
    the generators, which means that for  $n < 4$, we have $S_n = S_n^+$ as compact matrix quantum 
    groups. However, for  $n \ge 4$, one can show, that $S_n^+$ is non-commutative and thus strictly 
    bigger than  $S_n$.
\end{definition}

\begin{definition}
    \label{def:hnplus}
    The hyperoctahedral quantum group $H_{n}^+$ is a quantization of the hyperoctahedral group $H_{n}$ and was first defined by Bichon in~\cite{bichon_2004} and
    further studied by Banica, Bichon and Collins in~\cite{banica_bichon_collins_2007}.
    It is defined as the free wreath product of $S_2$ with $S_{n}^+$ :
    \[
        C(H_{n}^+) = C(S_2) \wr_{*} C(S_{n}^+)
    .\] 
\end{definition}
\begin{ex}
    In~\cite{schmidt_weber_2018}, Schmidt and Weber showed that the hyperoctahedral quantum group $H_2^+$ is the quantum automorphism group of the $4$-cycle:
     \[
        \QBan(C_4) = H_2^+
    .\] 
\end{ex}

\begin{definition}
    Given a graph $\Gamma$ on $n$ vertices, its quantum automorphism group  $\QBan(\Gamma)$ 
    is the compact matrix quantum group $\left(C(\QBan(\Gamma)), u\right)$, where 
    $C(\QBan(\Gamma))$ is the universal  $C^*$-algebra with generators  
    $u_{ij}$, $1 \le i, j \le j$ subject to the following 
    relations:
    \begin{align}
        \label{rel:symmetric_1}&u_{ij} = u_{ij}^* = u_{ij}^{2}\quad\quad & 1 \le i, j \le n&\\
        \label{rel:symmetric_2}&\sum_{k = 1}^{n} u_{ik} = 1 = \sum_{k=1}^n u_{ki}\quad \quad& 1\le i\le n&\\
        \label{rel:graph_1} &u_{ij} u_{kl} = u_{kl} u_{ij} = 0\quad\quad & i \sim k, j \not\sim l&\\
        \label{rel:graph_2} &u_{ij} u_{kl} = u_{kl} u_{ij} = 0 \quad \quad & i\not\sim k, j \sim l&
    .\end{align}
    To see the connection to the classical automorphism group of $\Gamma$, 
    the relations above can equivalently be written as
    \[
        C(\QBan(\Gamma)) = C^* ( u_{ij} = u_{ij}^* = u_{ij}^2,  \sum_{k = 1}^{n} u_{ik} = 1 = \sum_{k=1}^n u_{ki}, 1\le i\le n, A_{\Gamma} u = u A_{\Gamma})
    .\] 
\end{definition}
The quantum automorphism group of a graph is a quantum subgroup of the quantum symmetric group
$S_n^+$. Therefore, as soon as a graph has at least  $4$ vertices, the quantum automorphism group
might not be commutative. Sometimes however, the additional relations coming from the graph, i.e.
relations~(\ref{rel:graph_1}) and (\ref{rel:graph_2}) from the above definition, already are 
enough to get commutativity of the entire quantum group. In that case, we obtain that
\[
    C(\QBan(\Gamma)) = C(G_{aut}(\Gamma))
,\] 
i.e. the quantum automorphism group is equal to the classical automorphism group as a compact matrix quantum group.

\begin{definition}
    We say a graph $\Gamma$ has quantum symmetries, whenever  $\QBan(\Gamma) \neq G_{aut}(\Gamma)$, i.e. when 
    $\QBan(\Gamma)$ is not commutative. If  $\QBan(\Gamma)$ is commutative, we say  that  $\Gamma$ has no quantum symmetries.
\end{definition}

\section{Useful Lemmas} 
\counterwithin{thm}{subsection}
\label{section:useful_lemmas}
In this section, we collect some lemmas that are useful when computing whether a given graph 
has quantum symmetries or not.
Some of these lemmas come from the PhD thesis of Schmidt~\cite{schmidt_2020_1} or are 
generalizations of lemmas from there, while some are new.

\subsection{General Lemmas}
First, we give a few lemmas that hold in any graph.
The following one makes use of the adjoint in the $C^*$-algebra defining the quantum automorphism group and is
an essential part in most of the proofs of commutativity of the generators.
\begin{lem}[\cite{schmidt_2020_1}]
	\label{lem:adj_commute}
        Let $(u_{ij})_{1 \leq i, j \leq n}$ be the generators of $C(\QBan(\Gamma))$. If we have 
        \[
                u_{ij} u_{kl} = u_{ij} u_{kl} u_{ij}
        \]
        then $u_{ij}$ and $u_{kl}$ commute.
\end{lem}
The next lemma has been shown by Schmidt in his PhD thesis~\cite{schmidt_2020_1} and is an essential tool
to check whether a graph has quantum symmetries. For a lot of graphs of which we know that they have quantum
symmetries, this lemma is the reason we know it. However, we also know that not all quantum symmetries in graphs
come from disjoint automorphisms, as has been shown by Schmidt in~\cite{schmidt_2020_1}. 
\begin{lem}[Theorem 3.1.2 in~\cite{schmidt_2020_1}]
    \label{lem:disjoint_automorphisms}
    Let $\Gamma$ be a graph. If there exist two non-trivial, disjoint automorphisms in  $Aut(\Gamma)$, then
     $\Gamma$ has quantum symmetries and its quantum automorphism group is infinite dimensional. 
\end{lem}

The next lemma is a quantum version of the fact that any automorphism of a graph has to preserve the distance of any pair of 
vertices.
\begin{lem}[\cite{schmidt_2020_1}]
        \label{lem:distances}
        Let $\Gamma$ be a finite, undirected graph and let $(u_{ij})_{1 \leq i, j \leq n}$ be the generators of $C(\QBan(\Gamma))$.
        If we have $d(i, k) \neq d(j, l)$, then $u_{ij} u_{kl} = 0$.
\end{lem}

The following two lemmas are quite technical. However they allow to drastically simplify some computations by making 
the question whether certain products of generators are zero a question of simple properties of the graph in terms of the 
distance of certain sets of vertices to one another. This makes it possible to check e.g. with a computer in an automated 
fashion whether these products are zero.
The two lemmas are used extensively in the proofs of this paper, which is why we introduce a special notation for certain 
applications of the lemmas in~\ref{notation:compact_notation_lemma_choose_q_right} and~\ref{notation:compact_notation_lemma_choose_q_middle}.
\begin{lem}[\cite{schmidt_2020_1}]
	\label{lem:choose_q_right}
        Let $\Gamma$ be a finite, undirected graph and let $(u_{ij})_{1 \leq i, j \leq n}$ be the generators of $C(\QBan(\Gamma))$.
        Let $d(i, k) = d(j, l) = m$. Let $q$ be a vertex with $d(j, q) = s$, $d(q, l) = t$ and $u_{kl}u_{aq} = u_{aq}u_{kl}$ for all 
        $a$ with $d(a, k) = t$. Then
        \[
                u_{ij}u_{kl} = u_{ij} u_{kl} \sum_{\substack{p;\, d(l, p) = m, \\ d(p, q) = s}} u_{ip}.
        \]
        In particular, if we have $m = 2$ and if $\QBan(\Gamma) = \QBic(\Gamma)$ holds, then choosing $s = t = 1$ implies
        \[
                u_{ij} u_{kl} = u_{ij} u_{kl} \sum_{\substack{p; \, d(p, l) = 2\\ (p, q) \in E}} u_{ip}.
        \]
\end{lem}
\begin{lem}[\cite{schmidt_2020_1}]
        \label{lem:choose_q_middle}
        Let $\Gamma$ be a finite, undirected graph and let $(u_{ij})_{1 \leq i, j\leq n}$ be the generators of $C(\QBan(\Gamma))$.
        Let $d(i, k) = d(j, l) = m$ and let $p \neq j$ be a vertex with $d(p, l) = m$. Let $q$ be a vertex with $d(q, l) = s$ and $d(j, q) \neq d(q, p)$.
        Then
        \[
                u_{ij} \left( \sum_{\substack{t; \, d(t, j) = d(t, p) = m \\ d(t, q) = s}} u_{kt} \right) u_{ip} = 0.
        \]
        Especially, if $l$ is the only vertex satisfying $d(l, q) = s$, $d(l, j) = m$ and $d(l, p) = m$, we obtain $u_{ij} u_{kl} u_{ip} = 0$.
\end{lem}

The next lemma is a simple fact about a monomial that vanishes, but it appears a lot in the computations
of this paper, and therefore we include it as a lemma here.
\begin{lem}
    \label{lem:monomial_zero}
    Let $\Gamma$ be a finite, undirected graph and let  $(u_{ij})$ be the generators of $\QBan(\Gamma)$. Let  $i, j, k, l, p, q \in V(\Gamma)$
    be vertices such that $d(p, q) \neq d(j, q)$ and $u_{kl} u_{rq} = u_{rq} u_{kl}$ for any vertex $r \in V(\Gamma)$ with 
    $d(i, r) = d(p, q)$.
    Then it holds that
    \[
        u_{ij} u_{kl} u_{ip} = 0
    .\] 
\end{lem}
\begin{proof}
    We observe that
    \begin{align*}
        u_{ij} u_{kl} u_{ip} &= u_{ij} u_{kl} \sum_{\substack{r \in V(\Gamma)\\d(i, r) = d(p, q)}}u_{rq} \;u_{ip} \\
                             &=  \underbrace{u_{ij} \sum_{\substack{r \in V(\Gamma)\\d(i, r) = d(p, q)}}u_{rq}}_{=0} u_{kl} u_{ip}\\
                             &= 0
    .\end{align*}
    Here, the commutation of $u_{kl}$ and $u_{rq}$ is by assumption and the fact that $u_{ij} u_{rq} = 0$ is due to the fact that
    $d(i, r) = d(p, q) \neq d(j, q)$ by assumption.

\end{proof}

\subsection{Lemmas for Vertex-Transitive Graphs}
The following lemma is similar to Lemma~3.2.4 in \cite{schmidt_2020_1}, however there it is a stronger version that is only 
applicable for distance-transitive graphs.
\begin{lem}
    \label{lem:get_commutation_by_automorphism}
    Let $\Gamma$ be a graph and  $(u_{ij})_{1 \le i, j \le n}$ be the generators of $C(\QBan(\Gamma))$. Let 
    $j_1, l_1, j_2, l_2$ be vertices in $\Gamma$ such that  $d(j_1, l_1) = d(j_2, l_2)$. If we have 
    $u_{a j_1} u_{b l_1} = u_{b l_1} u_{a j_1}$ for all $a, b \in V(\Gamma)$ and there exists an automorphism 
    $\phi \in G_{aut}(\Gamma)$ such that $\phi(j_1) = j_2$ and $\phi(l_1) = l_2$, then we get
    \[
    u_{a j_2} u_{b l_2} = u_{b l_2} u_{a j_2} \text{ for all } a, b \in V(\Gamma)
    .\] 
\end{lem}

\begin{cor}
    \label{cor:commutation_with_one_implies_all}
    Let $\Gamma$ be a vertex-transitive graph  and  $(u_{ij})_{1 \le i, j \le n}$ be the generators of $C(\QBan(\Gamma))$. 
    If there exists a vertex $j_0 \in V(\Gamma)$ and a distance $m$ such that for all vertices $l \in V(\Gamma)$
    with $d(j_0, l) = m$ we have commutation of $u_{i j_0}$ and $u_{k l}$ for any vertices $i, k \in V(\Gamma)$, 
    then it already holds that for all vertices $i, j, k, l \in V(\Gamma)$ we have 
    \[
        d(j, l) = m \implies u_{ij} u_{kl} = u_{kl} u_{ij}
    .\] 

    In particular, if $u_{i j_0}$ commutes with all other generators $u_{kl}$ of $C(\QBan(\Gamma))$, then
     $\Gamma$ does not have quantum symmetries.
\end{cor}
\begin{proof}
    Let $j', l' \in V(\Gamma)$ be vertices of $\Gamma$ such that $d(j', l') = m$. 
    Since  $\Gamma$ is vertex-transitive, there exists an automorphism $\phi$ of $\Gamma$ that maps 
    $j_0$ to $j'$. Since $\phi$ is an automorphism, the preimage of $l_0 := \phi^{-1}(l')$ of $l'$ must 
    satisfy
    \[
        d(j_0, l_0) = d(j', l') = m.
    \]
    By assumption, we know that $u_{i j_0} u_{k l_0} = u_{k l_0} u_{i j_0}$ for any $i, k \in V(\Gamma)$.
    Applying Lemma~\ref{lem:get_commutation_by_automorphism}, we get that also 
    $u_{\phi(i) j'}$ and $u_{\phi(k) l'}$ commute. Since $i$ and  $k$ were arbitrary, we also get
     $u_{i j'} u_{k l'} = u_{k l'} u_{i j'}$ for arbitrary $i, k \in V(\Gamma)$.
\end{proof}

\subsection{Lemmas About Common Neighbours}
As seen in Lemma~\ref{lem:distances}, the distance between vertices matters when looking at the
generators of the quantum automorphism group of a graph. The simplest case to consider when looking at
distances is the neighbour. Because of this, a few lemmas were found concerning common neighbours of vertices.

The following two lemmas provide an easy way to see the commutation of all generators for adjacent vertices by looking 
at the structure of the graph, namely the existence of quadrangles and triangles.
\begin{lem}[\cite{schmidt_2020_1}]
        \label{lem:quadrangle}
        Let $\Gamma$ be an undirected graph that does not contain any quadrangle. Then for vertices $i\sim k$ and
        $j \sim l$ the generators  $u_{ij}$ and $u_{kl}$ commute.
\end{lem}
\begin{lem}[\cite{schmidt_2020_1}]
        \label{lem:one_common_neighbour}
        Let $\Gamma$ be an undirected graph such that adjacent vertices have exactly one common neighbour. 
        Then for vertices $i\sim k$ and $j \sim l$ the generators  $u_{ij}$ and $u_{kl}$ commute.
\end{lem}
The above lemma can be generalized a bit.
\begin{lem}
        \label{lem:one_common_neighbour_gen}
        Let $\Gamma = (V, E)$ be an undirected graph and let $(u_{ij})_{1 \leq i, j \leq n}$ be the generators of $C(\QBan(\Gamma))$.
        Let $i, j, k, l \in V$ be vertices of $\Gamma$ such that $(i, k) \in E$ and $(j, l) \in E$.
        Let moreover $i$ and $k$ have exactly one common neighbour $p$ and let it hold that for any two vertices from the set $\left\{ i, k, p \right\}$ the third
        vertex from the set is the only common neighbour of the two, i.e. the only common neighbour of $i$ and $p$ is $k$ and the only common neighbour of $k$ and $p$ is $i$.
        Let $j$ and $l$ also have exactly one common neighbour $q$ and let the same as above hold for the set of vertices $\left\{ j, l, q \right\}$.

        Then $u_{ij} u_{kl} = u_{kl} u_{ij}$.
\end{lem}
\begin{proof}
    The proof is the same as the one of Lemma~\ref{lem:one_common_neighbour} in the PhD thesis of Schmidt~\cite{schmidt_2020_1}:

        Using the relations of $C(\QBan(\Gamma))$, we get
        \[
            u_{ij} u_{kl} = u_{ij} u_{kl} \sum_{a; (l, a) \in E} u_{ia}
        .\] 
        We want to show that $u_{ij} u_{kl} u_{ia} = 0$ for $a \neq j$, $(a, l) \in E$. First, we consider $q$, the unique common neighbour of $j$ and $l$.
        We note, that for any vertex $b \neq l$ we have $(j, b) \notin  E$ or $(q, b) \notin E$, since $l$ is also the only common neighbour of $j$ and $q$.
        We deduce
        \[
            u_{ij} u_{kl} u_{iq} = u_{ij} \left( \sum_b u_{kb} \right) u_{iq} = u_{ij} 1 u_{iq} = u_{ij} u_{iq} = 0
        .\] 
        Let now $a \notin {j, q}$ and $(a, l) \in E$. 
        Then it holds that
        \[
            u_{ij} u_{kl} = u_{ij} \left( \sum_{b} u_{pb} \right) u_{kl} = u_{ij} u_{p q} u_{kl}
        \] 
        since $p$ is the only common neighbour of $i$ and $k$ while $q$ is the only common neighbour of $j$ and $l$.
        As $j$ is the only common neighbour of $l$ and $q$ and we have that $(a, l) \in E$, we deduce that $(a, q) \notin E$.
        We thus get
        \[
            0 = u_{p q} u_{i a} = u_{p q} \left( \sum_{c} u_{cl} \right) u_{ia} = u_{pq} u_{kl} u_{ia}
        \] 
        since $(i, p) \in E$ but $(a, q) \notin E$ and since $k$ is the only common neighbour of $p$ and $i$.

        Using the two equations we deduced above, we get
        \[
            u_{ij} u_{kl} u_{ia} = u_{ij} u_{pq} u_{kl} u_{ia} = 0
        .\] 
        We thus get for all $a \neq j$ that
        \[
            u_{ij} u_{kl} u_{ia} = 0
        .\] 
        Finally, we get
        \[
            u_{ij} u_{kl} = u_{ij} u_{kl} \left( \sum_a u_{ia} \right) = u_{ij} u_{kl} u_{ij}
        \] 
        and thus we get by Lemma~\ref{lem:adj_commute} that $u_{ij} u_{kl} = u_{kl} u_{ij}$.
\end{proof}

In order to make some statements and computations more compact, we introduce the following definition:
\begin{definition}
    \label{def:cn}
    Given a graph $\Gamma$ and vertices  $i, j$ of  $\Gamma$, we denote by  $\CN(i, j)$ the set of all common neighbours
    of  $i$ and  $j$, i.e. all vertices  $k$ of  $\Gamma$ such that both $i \sim k$ and  $j\sim k$ hold.
\end{definition}
The following lemma again yields an easy to check criterion for a certain product of generators to be zero. The way we used it
in this paper was mainly by using the two corollaries that follow.
\begin{lem}
        \label{lem:dist_two_one_vs_two_common_neighbours}
        \label{lem:different_numbers_common_neighbours}
        Let $\Gamma$ be a finite, undirected graph and let $(u_{ij})_{1 \leq i, j\leq n}$ be the generators of $C(\QBan(\Gamma))$.
        Let $i, j, k, l \in V$ be vertices in $\Gamma$. 

        If $ \lvert \CN(i, k) \rvert \neq  \lvert \CN(j, l) \rvert $ then we already have
        \[
            u_{ij} u_{kl} = 0
        .\] 
\end{lem}
\begin{proof}
        We set $a =  \lvert \CN(i, k) \rvert $ and $b =  \lvert \CN(j, l) \rvert $ and observe:
        \begin{align*}
            a \cdot u_{ij} u_{kl} &= \sum_{p \in \CN(i, k)} \underbrace{u_{ij} \sum_{q \in \CN(j, l)} u_{pq} u_{kl}}_{= u_{ij} u_{kl}} \\
                                  &= u_{ij} \sum_{p \in \CN(i, k)} \sum_{q \in \CN(j, l)} u_{pq} u_{kl}\\
                                  &= \sum_{q \in \CN(j, l)} \underbrace{u_{ij} \sum_{p \in \CN(i, k)} u_{pq} u_{kl}}_{= u_{ij} u_{kl}}\\
                                  &= b \cdot u_{ij} u_{kl}
        .\end{align*}
        Since by assumption we had $a \neq b$, we get that $u_{ij} u_{kl} = 0$.
\end{proof}

This lemma leads to the following corollaries.
\begin{cor}
    \label{cor:different_numbers_common_neighbours}
    If $i, j, k, l, p$ are vertices in  $\Gamma$ and  $ \lvert \CN(j, l) \rvert \neq  \lvert \CN(l, p)  \rvert $ then
    \[
        u_{ij} u_{kl} u_{ip} = 0
    .\] 
\end{cor}
\begin{proof}
    If $ \lvert\CN(i, k) \rvert \neq  \lvert \CN(j, l)\rvert $ then we have $u_{ij} u_{kl} = 0$ and thus also  $u_{ij} u_{kl} u_{ip} = 0$. Otherwise we have
    $ \lvert \CN(i, k)  \rvert =  \lvert \CN(j, l) \rvert  \neq  \lvert \CN(l, p)  \rvert $ and thus $u_{kl} u_{ip} = 0$ and thus  $u_{ij}u_{kl} u_{ip} = 0$.
\end{proof}
\begin{cor}
        \label{cor:one_triangle_one_not}
        Let $\Gamma$ be a finite, undirected graph and let $(u_{ij})_{1 \leq i, j\leq n}$ be the generators of $C(\QBan(\Gamma))$.
        If $i, j, k, l \in V$ are vertices of $\Gamma$ with $(i, k) \in E$ and $(j, l) \in E$ such that $i$ and $k$ 
        have at least one common neighbour, but $j$ and $l$ do not have a common neighbour, then
        \[
            u_{ij} u_{kl} = u_{kl} u_{ij} = 0
        .\] 
        In other words, if $i$ and $k$ are part of the same triangle but $j$ and $l$ are not, then the product of $u_{ij}$ and $u_{kl}$ is zero.
\end{cor}
\begin{proof}
        If $i$ and  $k$ are part of the same triangle, then in particular  $ \lvert \CN(i, k) \rvert > 0$, while
        $\lvert \CN(j, l) \rvert = 0$, since they are not in the same triangle.
        Applying Lemma~\ref{lem:different_numbers_common_neighbours}, we get $u_{ij} u_{kl} = 0 = u_{kl} u_{ij}$.
\end{proof}

\subsection{A strategy for deciding quantum symmetries}
In this section we want to give a short checklist of steps to take when one has a given graph $\Gamma$ and wants to decide, whether it has quantum symmetries.

\renewcommand{\labelenumii}{\arabic{enumi}.\arabic{enumii}.}
\renewcommand{\labelenumi}{\arabic{enumi}.}
\begin{enumerate}
    \item First, one should check, whether $\Gamma$ has disjoint automorphisms. If yes, then by Lemma~\ref{lem:disjoint_automorphisms} it has quantum symmetries.
    \item If $\Gamma$ does not have disjoint automorphisms, one should first try to show that it does not have quantum symmetries, i.e. that
        $C(\QBan(\Gamma))$ is commutative:
        \begin{enumerate}
            \item By Lemma~\ref{lem:distances}, it suffices to show $u_{ij} u_{kl} = u_{kl} u_{ij}$ only for vertices that fulfill $d(i, k) = d(j, l)$.
                One can thus fix a distance $m = d(i, k) = d(j, l)$ and try to show commutativity for vertices in this distance.
            \item For $m = 1$, the first thing to check is whether Lemma~\ref{lem:quadrangle} or Lemma~\ref{lem:one_common_neighbour} are 
                applicable, because these already yield commutativity for this case.
            \item In the next step, one can check whether using Lemmas~\ref{lem:choose_q_right} or \ref{lem:choose_q_middle} one can get 
                $u_{ij} u_{kl} = u_{ij} u_{kl} u_{ij}$ by showing that $u_{ij} u_{kl} u_{ip} = 0$ for $p \neq j$. This is a task that can easily be 
                delegated to a computer and can therefore be applied to all possible combinations of generators easily.
            \item If the results of Lemmas~\ref{lem:choose_q_right} and \ref{lem:choose_q_middle} only yield partial results, i.e. they only
                show $u_{ij} u_{kl} u_{ip} = 0$ for some of the vertices $p$, one can try to use Corollary~\ref{cor:different_numbers_common_neighbours} to show 
                it for the rest.
            \item If at any point one has shown $u_{ij} u_{kl} = u_{kl} u_{ij}$ for any generators, one can use Lemma~\ref{lem:get_commutation_by_automorphism}
                to propagate this to other pairs of generators.
            \item Any remaining pairs of vertices will have to be considered individually, potentially using a combination of 
                Lemmas~\ref{lem:choose_q_right}, \ref{lem:choose_q_middle}, \ref{lem:get_commutation_by_automorphism} and Corollary~\ref{cor:different_numbers_common_neighbours}.
        \end{enumerate}
\end{enumerate}
\renewcommand{\labelenumi}{(\roman{enumi})}
\renewcommand{\theenumi}{(\roman{enumi})}

\counterwithin{thm}{section}
\section{Overview and Notation}
\label{section:overview_and_notation}

On twelve vertices, there are in total $74$ vertex-transitive graphs. Since the quantum automorphism group of a graph is the same as the one 
of its complement, it suffices to only consider  $37$ of these graphs, since the rest can be obtained by taking their complements.

We will sort these graphs in five different subclasses: \hyperref[section:non_connected]{disconnected graphs},
\hyperref[section:products]{pro\-ducts of smaller graphs}, \hyperref[section:circulant_graphs]{circulant graphs}, \hyperref[section:semicirculant_graphs]{semicirculant graphs}
and \hyperref[section:special_cases]{special cases} that do not fit into any of the other subclasses.

In total, we will consider the following graphs:
\renewcommand{\arraystretch}{1.5}
\begin{table}[H]
    \centering
    \label{tab:all_vtrans_graphs}
    \begin{tabular}{c|l}
    Subclass& Graphs\\
    \hline\\
    disconnected & $6 K_2,\, 4 K_3,\,3K_4,\, 3C_4,\, 2K_6,\,2 C_6,\,$\\
                  & $ 2 K_2 \square K_3,\,2C_6(2), \,2 C_6(3)$\\
    products      & $K_6 \times K_2, \,K_3 \times K_4, \,C_6 \square K_2,$\\
                  & $C_4 \square C_3, \,K_2 \square C_6(2), \,K_2 \square C_6^+$\\
    circulant     & $C_{12}, \,K_{12}, \,C_{12}^+, \,C_{12}(2), \,C_{12}(3), \,C_{12}(4),$\\
                  & $C_{12}(5),\, C_{12}(2, 6),\, C_{12}(4, 6),\, C_{12}(3, 6),$\\
                  & $C_{12}(4, 5),\, C_{12}(5, 6)$\\
    semi-circulant& $C_{12}(5^+),\, C_{12}(3^+, 6), \, C_{12}(5^+, 6), $\\
                  & $C_{12}(2, 5^+), \,C_{12}(4, 5^+)$\\
    special cases & $\Line(\text{Cube}),\, \Line(C_6(2)), $ Icosahedron,\\
                  & $\Trunc(K_4), \, \Antip(\Trunc(K_4))$
    \end{tabular}
    \caption{Vertex transitive graphs on $12$ vertices up to complements}
\end{table}

    In the proofs that a graph has no quantum symmetries, we will often repeatedly apply Lemma~\ref{lem:choose_q_right} and Lemma~\ref{lem:choose_q_middle}. 
    In order to prevent this from taking up too much space, we will now introduce the following notation:
\begin{notation}    
    \label{notation:compact_notation_lemma_choose_q_right}
    Let $\Gamma$ be a graph and  $j$ a fixed vertex in $\Gamma$.
    If we want to apply Lemma~\ref{lem:choose_q_right} to  $j$ and 
        $l_1, l_2, \ldots$ and $q_1, q_2, \ldots$ to get that  $u_{ij} u_{kl_a} = u_{ij} u_{kl_a} \sum_{\substack{p \in P_a}}$, we will write
        the following table:
        \begin{table}[H]
            \begin{tabularx}{0.2\linewidth}{|*{4}{X|}}
                \multicolumn{4}{c}{Lemma~\ref{lem:choose_q_right}}\\
                \hline
                $j$ &$l$ &  $q$ &  $p$\\
                 \hline
                $j$&$l_1$ & $q_1$ & $P_1$\\
                   &$l_2$ & $q_2$ & $P_2$\\
                   &$\vdots$ &  $\vdots$ &  $\vdots$\\
                   \hline
            \end{tabularx}
        \end{table}

        Here, the sets $P_a$ will be the sets of vertices  $p$ such that  $d(l_a, p) = d(j, l_a)$ and  $d(p, q_a) = d(j, q_a)$.
        Often, we will use this when the sets  $P_a$ are singleton sets, because this will mean that we have  $u_{ij} u_{kl_a} = u_{ij} u_{kl_a} u_{ij}$. 
\end{notation}
\begin{notation}
        \label{notation:compact_notation_lemma_choose_q_middle}
        Let $\Gamma$ be a graph and  $j$ a fixed vertex in $\Gamma$.
        If we want to apply Lemma~\ref{lem:choose_q_middle} to $j$, $l$, $p$ and $q$, we often want to apply it in such a way, 
        that we get $u_{ij} u_{kl} u_{ip} = 0$.
        This is the case if $l$ is the unique vertex in distance $d(l, q)$ from  $q$,  $d(l, j)$ from  $j$ and  $d(l, p)$ from  $p$. 
        If that is the case, we will write the application of Lemma~\ref{lem:choose_q_middle} to $j$,  $l$, $p_1, p_2, \ldots$ and 
        $q_1, q_2, \ldots$ as the following table:
        \begin{table}[H]
            \begin{tabularx}{0.2\linewidth}{|*{4}{X|}}
                \multicolumn{4}{c}{Lemma~\ref{lem:choose_q_middle}}\\
                \hline
                $j$ &$l$ &  $p$ &  $q$\\
                 \hline
                $j$&$l$ & $p_1$ & $q_1$\\
                   & & $p_2$ & $q_2$\\
                   & &  $\vdots$ &  $\vdots$\\
                   \hline
            \end{tabularx}
        \end{table}
        A particular row of such a table containing $j', l', p', q'$ then means that 
        $u_{ij'} u_{kl'} u_{ip'} = 0$ for all  vertices $i, k \in V(\Gamma)$.
\end{notation}

To see how an application of the newly introduced notations might look in practice and how one might start working on a new graph
in general, we give a simple example:
\begin{ex}
    We consider the $5$-cycle  $C_5$. We already know, that it has no quantum symmetries, but we now want to show this using
    Lemma~\ref{lem:choose_q_middle} and Lemma~\ref{lem:choose_q_right}.
    By Corollary~\ref{cor:commutation_with_one_implies_all}, we only need to show that the generators $u_{i 1}$ commute with 
    all generators $u_{kl}$.
    Note moreover, that by Lemma~\ref{lem:adj_commute}, we only need to show that $u_{i 1} u_{kl} = u_{i 1} u_{kl} u_{i 1}$.
    We do this by using the fact from the relations of $\QBan(C_5)$ that $\sum_{r} u_{i r} = 1$ and thus
    \[
         u_{i 1} u_{kl} = u_{i 1} u_{kl} \sum_{p} u_{i p}
    .\] 
    Using additionally Lemma~\ref{lem:distances}, we get that
    \[
        u_{i 1} u_{kl} \sum_{p} u_{i p} = u_{i 1} u_{kl} \sum_{\substack{p\\ \, d(p, l) = d(i, k)}} u_{i p}
    .\] 
    Now we just need to show that $u_{i 1} u_{kl} u_{ip} = 0$ for any vertex $l$ and any $p \neq 1$ with $d(p, l) = d(i, k)$.
    Here we can assume that $d(i, k) = d(1, l)$ without loss of generality, since whenever $d(i, k) \neq d(1, l)$, we have
    $u_{i 1} u_{kl} = u_{kl} u_{i 1} = 0$ and we are already done.
    Note, that in the case of $C_5$ for any vertex $l$, only one vertex  $p \neq 1$ is in the same distance to $l$ as  $1$
    and therefore the above means that for any  $l$, we only have to show  $u_{i 1} u_{kl} u_{i p} = 0$ for a single vertex
    $p$.

    In order to show that $u_{i 1} u_{kl} u_{ip} = 0$ for the above specified vertices $l$ and  $p$, we will first use 
    Lemma~\ref{lem:choose_q_middle}. 
    To get that $u_{i 1} u_{k 2} = u_{k 2} u_{i 1}$, we will now use the lemma to show that $u_{i 1} u_{k 2} u_{i 3} = 0$. 
    Looking at the conditions for applying 
    Lemma~\ref{lem:choose_q_middle}, we thus want to find a vertex $q$ such that $d(1, q) \neq d(q, 3)$ and such that
    $2$ is the only vertex that is both in distance  $s = d(2, q)$ from  $q$ and  in distance  $1$ from both  vertices $1$ and  $3$.
    Choosing  $q := 1$ satisfies this as can be seen by looking at the graph:
    
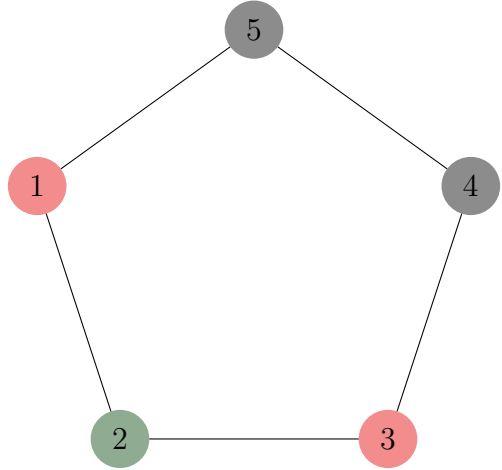
\begin{center}
        \begin{tikzpicture}[auto, node distance = 2cm and 2.5cm, on grid, 
            vertexRed/.style={circle, draw=xkcdRed!45, fill=xkcdRed!45, very thick, minimum size=5mm},
            vertexGreen/.style={circle, draw=xkcdForestGreen!45, fill=xkcdForestGreen!45, very thick, minimum size=5mm},
            vertexBlue/.style={circle, draw=xkcdRoyalBlue!45, fill=xkcdRoyalBlue!45, very thick, minimum size=5mm},
            vertexBlack/.style={circle, draw=black!45, fill=black!45, very thick, minimum size=5mm}
            ]
            \node[vertexRed] (1) at   (90 + 1*72:3cm) {$1$};
            \node[vertexGreen] (2) at (90 + 2*72:3cm) {$2$};
            \node[vertexRed] (3) at  (90 + 3*72:3cm) {$3$};
            \node[vertexBlack] (4) at   (90 + 4*72:3cm) {$4$};
            \node[vertexBlack] (5) at (90 + 5*72:3cm) {$5$};
            \foreach[evaluate ={
                \l = int(\j + 1);
            }] \j in {1,...,5}{
            \ifthenelse {\j  < 5}{\draw (\j) -- (\l);}{\draw (\j) -- (1);};
        }
        \end{tikzpicture}
    \end{center}
    We can therefore apply  Lemma~\ref{lem:choose_q_middle} to get 
    that  $u_{i 1} u_{k 2} u_{i 3} = 0$ and therefore
    \[
        u_{i 1} u_{k 2} = u_{i 1} u_{k 2} u_{i 1}
    \] 
    as desired.

    In a very similar manner, we can apply Lemma~\ref{lem:choose_q_middle} to $l = 5$. Both of these applications are
    summarised in the following table, using the notation from~\ref{notation:compact_notation_lemma_choose_q_middle}:
    \begin{table}[H]
        \begin{tabularx}{0.2\linewidth}{|*{4}{X|}}
            \multicolumn{4}{c}{Lemma~\ref{lem:choose_q_middle}}\\
            \hline
            $j$ &$l$ &  $p$ &  $q$\\
             \hline
            $1$&$2$ & $3$ & $1$\\
               & $5$&  $4$&  $1$\\
               \hline
        \end{tabularx}
    \end{table}
    This table thus already shows that the $u_{i 1}$ commute with all $u_{kl}$ where $l$ is adjacent to $1$. Using
    Lemma~\ref{lem:get_commutation_by_automorphism}, we thus get that all $u_{ij}$ and $u_{kl}$ commute whenever $j \sim l$.
    We now want to use this information to apply Lemma~\ref{lem:choose_q_right} to show that 
    $u_{i 1} u_{k 3} = u_{i 1} u_{k 3} u_{i 1}$. We thus have to find a vertex $q$ with  $d(q, 3) = 1$ such that
    there is exactly one vertex $p$ that satisfies  $d(3, p) = 2$ and $d(p, q) = d(j, q)$. Choosing  $q := 2$ only leaves the
    vertex  $p = 1$ satisfying these conditions and therefore Lemma~\ref{lem:choose_q_right} yields
     \[
        u_{i 1} u_{k 3} = u_{i 1} u_{k 3} u_{i 1}
    \] 
    and therefore also
    \[
        u_{i 1} u_{k 3} = u_{k 3} u_{i 1}
    .\] 
    In a similar manner, we can apply Lemma~\ref{lem:choose_q_right} to get  $u_{i 1} u_{k 4} = u_{i 1} u_{k 4} u_{i 1}$, and
    again, both of these applications are summarised in the following table according to Notation~\ref{notation:compact_notation_lemma_choose_q_right}:
    \begin{table}[H]
        \begin{tabularx}{0.25\linewidth}{|*{4}{X|}}
            \multicolumn{4}{c}{Lemma~\ref{lem:choose_q_middle}}\\
            \hline
            $j$ &$l$ &  $q$ &  $p$\\
             \hline
            $1$&$3$ & $2$ & $\{1\} $\\
               & $4$&  $3$&  $\{1\} $\\
               \hline
        \end{tabularx}
    \end{table}
    All in all, we see that $u_{i 1} u_{kl} = u_{kl} u_{i 1}$ for all vertices $i, k$ and  $l$ and thus by Lemma~\ref{lem:get_commutation_by_automorphism} we get  $u_{ij} u_{kl} = u_{kl} u_{ij}$ for all generators of the quantum automorphism group of $C_5$.
\end{ex}

\section{Constructions from smaller Graphs}
\label{section:constructions_from_smaller_graphs}
In this section we consider graphs that arise as constructions of smaller graphs. A lot of the time, their quantum automorphism groups can be computed in a simple manner
from the quantum automorphism groups of the smaller graphs.
\counterwithin{thm}{subsection}
\subsection{Disconnected graphs}

\label{section:non_connected}
Vertex transitive graphs that are disconnected are always disjoint copies of a single connected graph. 
For such disjoint copies of a connected graph, it was shown in~\cite{banica_bichon_2006_2} how to compute the quantum automorphism groups:
\begin{thm}[\cite{banica_bichon_2006_2}]
    If $\Gamma$ is a connected graph, then the graph consisting of  $n$ disjoint copies of  $\Gamma$ has quantum automorphism group
    \[
        \QBan(n\Gamma) = \QBan(\Gamma) \wr_* S_n^+
    .\] 
\end{thm}
We thus can compute the Quantum automorphism groups of these graphs easily.
\begin{table}[H]
    \centering
    \begin{tabular}{l|l|l}
        Graph& Automorphism Group& Quantum Automorphism Group\\
        \hline
        $6 K_2$ & $\Z_2 \wr S_6$ & $\Z_2 \wr_* S_6^+$\\
        $4 K_3$ & $S_3 \wr S_4$ & $S_3 \wr_* S_4^+$\\
        $3 K_4$ & $S_4 \wr S_3$ & $S_4^+ \wr_* S_3$\\
        $3C_4$ & $H_2 \wr S_3$ & $H_2^+ \wr_* S_3$ \\ 
        $2 K_6$ & $S_6 \wr \Z_2$ & $S_6^+ \wr_* \Z_2$\\
        $2 C_6$ & $D_6 \wr \Z_2$ & $D_6 \wr_* \Z_2$\\
        $2(K_2 \square K_3)$ & $D_{6} \wr \Z_2$ & $D_{6} \wr_* \Z_2$\\
         $2 C_6(2)$ &  $(\Z_2 \wr S_3) \wr \Z_2$ & $(\Z_2 \wr_* S_3) \wr_* \Z_2$\\
        $2 C_6(3)$ & $(S_3 \wr \Z_2) \wr \Z_2$ & $(S_3 \wr_* \Z_2) \wr_* \Z_2$
    \end{tabular}
\end{table}

\subsection{Products of smaller graphs}
\label{section:products}
For finite graphs one can define the following product operations:
 \begin{defn}
    Let $\Gamma$ and  $\Gamma'$ be finite graphs.
    \begin{enumerate}
        \item The direct product $\Gamma \times \Gamma'$ has vertex set $V(\Gamma) \times V(\Gamma')$ and the following edges:
            \[
                (i, k) \sim (j, l) \iff i \sim j \text{ and } k \sim l
            .\] 
        \item The Cartesian product $\Gamma \square \Gamma'$ has vertex set $V(\Gamma) \times V(\Gamma')$ and the following edges:
            \[
                (i, k) \sim (j, l) \iff i = j, k \sim l \text{ or } i \sim j, k = l
            .\] 
    \end{enumerate}
\end{defn}
In~\cite{banica_2005}, the following theorem about quantum automorphism groups of graph products was proven.
\begin{thm}
    Let $\Gamma$ and  $\Gamma'$ be finite connected regular graphs. If their spectra  $\{\lambda\} $ and $\{\mu\} $ do not contain $0$ and satisfy
     \[
        \{\lambda_i / \lambda_j\} \cap \{\mu_k / \mu_l\} = \{1\} 
    ,\] 
    then $C(\QBan(\Gamma \times \Gamma')) = C(\QBan(\Gamma)) \otimes C(\QBan(\Gamma'))$.

    If the spectra satisfy
    \[
        \{\lambda_i - \lambda_j\} \cap \{\mu_k - \mu_l\} = \{0\} 
    ,\] 
    then $C(\QBan(\Gamma \square \Gamma')) = C(\QBan(\Gamma)) \otimes C(\QBan(\Gamma'))$.
\end{thm}
This theorem yields the quantum automorphism group for a few of the product graphs already. 
\begin{table}[H]
    \centering
    \begin{tabular}{l|l|l}
        Graph& Automorphism Group& Quantum Automorphism Group\\
        \hline
        $K_6 \times K_2$& $S_6 \times \Z_2$ & $S_6^+ \times \Z_2$\\
        $K_3 \times K_4$ & $S_3 \times S_4$ & $S_3 \times S_4^+$\\
        $C_4 \square C_3$ & $H_2 \times S_3$ & $H_2^+ \times S_3$\\
        $K_2 \square C_6(3) $ & $\Z_2 \times (S_3 \wr \Z_2)$ & $\Z_2 \times (S_3 \wr_* \Z_2)$
    \end{tabular}
\end{table}
For the two product graphs $K_2 \square C_6$ and $K_2 \square C_6(2)$, the theorem does not apply, and they therefore have to be studied 
individually. 
We will first consider$K_2 \square C_6$:
\begin{prop}
    \label{prop:k2_c6}
    The graph $K_2 \square C_6$, pictured below, does not have quantum symmetry.
\end{prop}
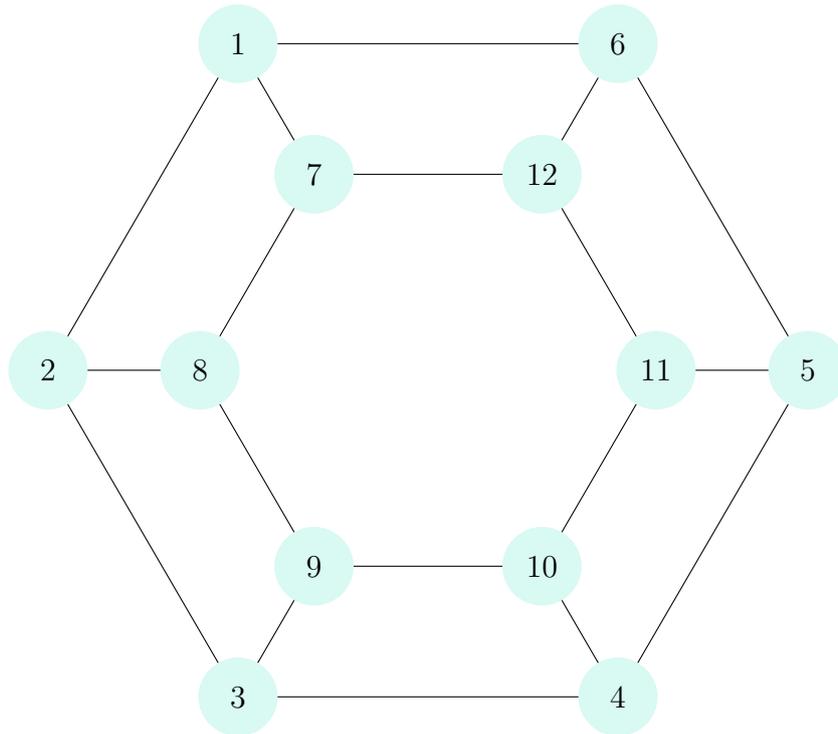
\begin{center}
    \begin{tikzpicture}[auto, node distance = 2cm and 2.5cm, on grid, vertex/.style={circle, draw=xkcdAquamarine!15, fill=xkcdAquamarine!15, very thick, minimum size=10mm}]
        \foreach[evaluate ={\e = int(\i + 6)}]
        \i in {1,...,6}{
        \node[vertex] (\i) at (60 + \i * 60:5cm) {$\i$};
        \node[vertex] (\e) at (60 + \i * 60:3cm) {$\e$};
    }
        \foreach[evaluate ={
            \l = int(\j + 1);
            \a = int(\j + 6);
            \b = int(\j + 7);
            \c = int(\j + 6);
        }] \j in {1,...,6}{
        \ifthenelse {\j  < 6}{\draw (\j) -- (\l);}{\draw (\j) -- (1);};
        \draw (\j) -- (\a);
        \ifthenelse{\a < 12}{\draw (\a) -- (\b);}{\draw(\a) -- (7);};
    }
    \end{tikzpicture}
\end{center}

\begin{proof}
    We will show that $u_{i 1} u_{kl} = u_{kl} u_{i 1}$ for all vertices $i, k, l$ of $K_2 \square C_6$. Then the vertex transitivity
    will imply the commutativity of all generators. 

    We begin with distance $2$. There are $4$ vertices in distance  $2$ from  $1$: $3, 5, 8$ and  $12$.
    First considering vertex $3$, we see that
     \[
    u_{i 1} u_{k 3} = u_{i 1} u_{k 3} (u_{i 1} + u_{i 5} + u_{i 8} + u_{i 10})
    .\] 
    We will now apply Lemma~\ref{lem:choose_q_middle} to $1$ and  $3$ using the notation introduced in~\ref{notation:compact_notation_lemma_choose_q_middle}:
        \begin{table}[H]
            \begin{tabularx}{0.2\linewidth}{|*{4}{X|}}
                \multicolumn{4}{c}{Lemma~\ref{lem:choose_q_middle}}\\
                \hline
                $j$ &$l$ &  $p$ &  $q$\\
                 \hline
                $1$&$3$ & $5$ & $2$\\
                   & & $8$ & $6$\\
                   \hline
            \end{tabularx}
        \end{table}
        We thus get that $u_{i 1} u_{k 3} u_{i 5} = 0 = u_{i 1} u_{k 3} u_{i 8}$. 
        Recall now, that by Definition~\ref{def:cn}, we denote by $\CN(i, j)$ the set of common neighbours of the vertices  $i$ and  $j$.
        We can now see that $\lvert\CN(1, 3)\rvert = 1$, but we have $\lvert\CN(3, 10)\rvert = 2$, and we thus get with 
        Corollary~\ref{cor:different_numbers_common_neighbours} that $u_{i 1} u_{k 3} u_{i 10} = 0$.
        We thus get $u_{i 1} u_{k 3} = u_{k 3} u_{i 1}$.

        Next, we consider vertex $8$ and apply Lemma~\ref{lem:choose_q_middle}:
        \begin{table}[H]
             \begin{tabularx}{0.2\linewidth}{|*{4}{X|}}
                \multicolumn{4}{c}{Lemma~\ref{lem:choose_q_middle}}\\
                \hline
                $j$ &$l$ &  $p$ &  $q$\\
                 \hline
                $1$&$8$ & $3$ & $4$\\
                   & & $12$ & $2$\\
                   \hline
            \end{tabularx}
        \end{table}
        Again, only $u_{i 1} u_{k 8} u_{i 10}$ is the missing term, but seeing that
        $\lvert \CN(1, 8)\rvert = 2 \neq \lvert \CN(8, 10) = 3$  we get again by Corollary~\ref{cor:different_numbers_common_neighbours}
        that $u_{i 1} u_{k 8} u_{i 10} = 0$ and thus  $u_{i 1} u_{k 8} = u_{k 8} u_{i 1}$.

        We observe, that mirroring on the line that goes through  $1, 7, 10 $ and  $4$ is an automorphism
        of $K_2 \square C_6$ that keeps $1$ fixed and maps  $3$ to  $5$ and  $8$ to  $12$, we get by
         Lemma~\ref{lem:get_commutation_by_automorphism} that $u_{i 1} u_{k 5} = u_{k 5} u_{i 1}$
         and also $u_{i 1} u_{k 12} = u_{k 12} u_{i 1}$. We thus have for all vertices  $j, l$ with 
         $d(j, l) = 2$ that  $u_{ij} u_{kl} = u_{kl} u_{ij}$.

         For distance $1$, we apply Lemma~\ref{lem:choose_q_right} to all neighbours of the vertex  $1$: 
        \begin{table}[H]
             \begin{tabularx}{0.225\linewidth}{|*{3}{X|}l|}
                \multicolumn{4}{c}{Lemma~\ref{lem:choose_q_right}}\\
                \hline
                $j$ &$l$ &  $q$ &  $p$\\
                 \hline
                $1$&$2$ & $6$ & $\{1\} $\\
                   &$6$& $2$ & $\{1\} $\\
                   &$7$ & $2$ &  $\{1, 8\} $\\
                   \hline
            \end{tabularx}
        \end{table}
        We see that $u_{i 1} u_{k 2} = u_{k 2} u_{i 1}$ and $u_{i 1} u_{k 6} = u_{k 6} u_{i 1}$.
        For $l = 7$, we moreover see that
       \begin{align*}
           u_{i 1} u_{k 7} = u_{i 1} u_{k 7} \left( u_{i 1} + u_{i 8} \right).
        \end{align*}
        Looking at $u_{i 1} u_{k 7} u_{i 8}$, we see that
        \begin{alignat*}{2}
             u_{i 1} u_{k 7} u_{i 8} &= 
             u_{i 1} u_{k 7} \sum_{\substack{r; d(r, i) = 1\\ d(r, k) = 2}} u_{r 9} u_{i 8}&\\
                     &= u_{i 1}  \sum_{\substack{r; d(r, i) = 1\\ d(r, k) = 2}} u_{r 9} u_{k 7} u_{i 8}
                     &\text{ since  } d(7, 9) = 2\\
                     &= 0 &\text{ since } d(1, 9) = 3 \neq d(r, i) = 1
        .\end{alignat*}

        For distance $3$, it is enough to only apply Lemma~\ref{lem:choose_q_right}:
        \begin{table}[H]
             \begin{tabularx}{0.225\linewidth}{|*{3}{X|}l|}
                \multicolumn{4}{c}{Lemma~\ref{lem:choose_q_right}}\\
                \hline
                $j$ &$l$ &  $q$ &  $p$\\
                 \hline
                $1$&$4$ & $10$ & $\{1\} $\\
                   &$9$& $10$ & $\{1\} $\\
                   &$11$ & $10$ &  $\{1\} $\\
                   \hline
            \end{tabularx}
        \end{table}
        
        Lastly, since $10$ is the only vertex in distance $4$ to  $1$, and thus also vice versa, we get
         \[
             u_{i 1} u_{k 10} = u_{i 1} u_{k 10} \sum_{r; d(r, 10) = 4} u_{i r} = u_{i 1} u_{k 10} u_{i 1}
        .\] 
        From this we get $u_{i 1} u_{k 10} = u_{k 10} u_{i 1}$ and therefore
        $C(\QBan(K_2 \square C_6))$ is commutative.
\end{proof}
A lot of the proofs that a graph does not have quantum symmetries are similar to the one above, repeatedly applying
Lemma~\ref{lem:choose_q_middle} and Lemma~\ref{lem:choose_q_right} and then using these results together with some
properties of the specific graph to get the commutation of the rest of the generators.
For this reason, we give the rest of such proofs in the Appendix~\ref{appendix:omitted-proofs}.

\begin{prop}
    The graph $K_2 \square C_6(2)$ has quantum symmetries.
\end{prop}
\begin{proof}
    We have the following automorphisms of $K_2 \square C_6(2)$:
    \begin{align*}
        ([1, 1] [1, 4])([2, 1] [2, 4]) \in \Aut(K_2 \square C_6(2))\\
        ([1, 2] [1, 5])([2, 2] [2, 5]) \in \Aut(K_2 \square C_6(2))
    .\end{align*}
    These are non-trivial and disjoint, and thus, by Lemma~\ref{lem:disjoint_automorphisms}, 
    $K_2 \square C_6(2)$ has quantum symmetries.
\end{proof}
\begin{table}[H]
    \centering
    \begin{tabular}{l | l | l}
        Graph & Automorphism Group & Quantum Automorphism Group\\
        \hline
        $K_2 \square C_6 $ & $\Z_2 \times D_6$ & $\Z_2 \times D_6$\\ 
        $K_2 \square C_6(2) $ & $\Z_2 \times (\Z_2 \wr S_3)$& ?
    \end{tabular}
\end{table}

\section{Circulant and Semi-circulant Graphs}
\label{section:circulant_and_semi_circulant_graphs}
Two related families of vertex-transitive graphs are circulant and semi-circulant graphs. Roughly speaking they arise when taking a cycle-graph and adding
edges between certain vertices.

Prior results yield the non-existence of quantum symmetries for a few of these graphs, however most of them have to be considered individually.
\subsection{Circulant graphs}

\begin{defn}
    For $1 < k_1 < \dots < k_r \le  [\frac{n}{2}]$, we define the graph $C_n(k_1, \dots, k_r)$, the \emph{circulant graph} on $n$ vertices with chords in distances $k_1, \dots, k_r$, as the graph obtained by drawing the $n-$cycle $C_n$ and then connecting all
    pairs of vertices in distance $k_i$, for any $i$. In other words, the vertex set is $V = \{1, \dots, n\}$ and for $i, j \in V$ we have
    \[
        i \sim j \Leftrightarrow |i - j| \mod \, n \in \{1, k_1, \dots, k_r\}
    .\] 
\end{defn}

\label{section:circulant_graphs}
For some circulant graphs, Theorem 3.1 in~\cite{banica_bichon_2006_1} shows that they do not have quantum symmetries.
\begin{thm}[\cite{banica_bichon_2006_1}]
    \label{thm:injective_f}
    Given a circulant graph $C_n(k_1, \ldots, k_r)$ and putting $k_0 := 1$, we define the function $f: \{1, 2, \ldots, [\frac{n}{2}]\} \to \R$ as follows:
    \[
        f(s) = \sum_{i = 0}^{r} \cos\left( \frac{2 k_i s \pi}{n} \right)
    .\] 
    If $n \neq 4$ and the function $f$ is injective, then the graph  $C_n(k_1, \ldots, k_r)$ does not have quantum symmetries.
\end{thm}
\begin{prop}
    The graphs $C_{12}$, $C_{12}(3)$ and $C_{12}(6)$ do not have quantum symmetries.
\end{prop}
\begin{proof}
    We apply Theorem~\ref{thm:injective_f}, and thus need to compute the values of the associated function $f$ on the set  $\{1, \ldots, 6\} $:
    \begin{table}[H]
        \centering
        \begin{tabular}{l| l l l l l l}
            & $1$&  $2$&  $3$& $4$& $5$&  $6$\\
            \hline
            $C_{12}$ & $0.87$&$0.5$& $0$&$-0.5$& $-0.87$& $-1$   \\
            $C_{12}(3)$&$0.87$& $-0.5$& $-1.23$& $0.5$& $-0.87$& $-2$    \\
            $C_{12}(6)$& $-0.13$& $1.5$& $-1$& $0.5$ &$-1.87$& $0$
        \end{tabular}
    \end{table}
    Here the rounding error is less than $0.01$, and thus for all three graphs the function is injective.
\end{proof}

For some other graphs, we see that they have disjoint automorphisms, and therefore 
they do have quantum symmetries:
\begin{prop}
    The graphs $K_{12}$, $C_{12}(5)$, $C_{12}(4, 5)$ and $C_{12}(5, 6)$ have quantum symmetries.
\end{prop}
\begin{proof}
    It is easy to see that the complete graph on $n$ vertices has quantum automorphism group  $S_n^+$ and thus, for 
     $n \ge 4$ does have quantum symmetries. 
     For the other graphs, we can check that they have disjoint automorphisms and thus, by Lemma~\ref{lem:disjoint_automorphisms}
     they have quantum symmetries.
     \begin{table}[H]
         \begin{tabular}{l | l | l}
          Graph & Automorphism $1$&  Automorphism $2$\\
          \hline 
          $C_{12}(5)$ & $(1\; 7)$& $(4\; 10)$ \\
          $C_{12}(4, 5)$ & $(1\; 7)(3\; 9)(5\; 11)$ &  $(2\; 8)(4\; 10)(6\; 12)$\\
          $C_{12}(5, 6)$ & $(1\; 7)$ &  $(4\; 10)$
         \end{tabular}
     \end{table}
\end{proof}
For $C_{12}(4, 5)$, we managed to compute the actual quantum automorphism group. We give the computation in 
Section~\ref{section:computation_qaut_groups}.

\begin{prop}
    The graphs $C_{12}(2)$, $C_{12}(4)$, $C_{12}(2, 6)$, $C_{12}(3, 6)$ and $C_{12}(4, 6)$ do not have 
    quantum symmetries
\end{prop}
The proofs of the above statement are very similar to the one of Proposition~\ref{prop:k2_c6}, but for completeness they are given in the Appendix~\ref{appendix:omitted-proofs}.

Summarising, we see that of the twelve circulant graphs on  $12$ points,  four have quantum symmetries 
and eight do not.
\begin{table}[H]
    \centering
    \begin{tabular}{l|l|l}
     Graph& Automorphism Group& Quantum Automorphism Group\\
        \hline
    $C_{12}$ & $D_{12}$ & $D_{12}$\\
     $C_{12}(3)$ & $D_{12}$ & $D_{12}$\\
     $C_{12}(6)$ & $D_{12}$ & $D_{12}$\\
     $K_{12}$ & $S_{12}$ & $S_{12}^+$\\
     $C_{12}(5)$ & $\mathcal{A}$&?\\
     $C_{12}(4, 5)$ & $H_2 \times S_3$ & $H_2^+ \times S_3$\\
     $C_{12}(5, 6)$ &$\mathcal{A}$& ?\\
     $C_{12}(2)$& $D_{12}$ & $D_{12}$\\
     $C_{12}(4)$& $D_{12}$ & $D_{12}$\\
     $C_{12}(2, 6)$&$D_{12}$&$D_{12}$\\
     $C_{12}(3, 6)$&$D_{12}$&$D_{12}$\\
     $C_{12}(4, 6)$&$D_{12}$&$D_{12}$\\
    \end{tabular}
\end{table}
Here we call
\[
    \mathcal{A} :=(\Z_2 \times ((((\Z_2 \times \Z_2 \times \Z_2 \times \Z_2)\rtimes \Z_2)\rtimes \Z_3)\rtimes \Z_2)) \rtimes \Z_2.
\]

\subsection{Semi-circulant graphs}
\label{section:semicirculant_graphs}
\begin{defn}
    For $1 < k_1 < \dots < k_r \le  [\frac{n}{2}]$, we define the \emph{semi-circulant} graph $C_n(k_1, \dots, k_r, l_1^+, \dots, l_s^+)$
    as the graph that is constructed by taking the circulant graph $C_n(k_1, \dots, k_r)$ and adding 
    edges between $i, j \in V$ if $i$ is even, $i < j$ and $j - i \in \{l_1^+, \dots, l_s^+\}$.
\end{defn}
As was the case with the circulant graphs, we see that some semi-circulant graphs have disjoint automorphisms
and therefore quantum symmetries.
\begin{prop}
    The graphs $C_{12}(5^+)$, $C_{12}(3^+, 6)$ and $C_{12}(5^+, 6)$ have quantum symmetries.
\end{prop}
\begin{proof}
    The graphs have the following disjoint automorphisms and thus have quantum symmetries by Lemma~\ref{lem:disjoint_automorphisms}.
     \begin{table}[H]
         \begin{tabular}{l | l | l}
          Graph & Automorphism $1$&  Automorphism $2$\\
          \hline 
          $C_{12}(5^+)$ & $(1\; 7)(2\;8)$& $(3\;9)(4\; 10)$ \\
          $C_{12}(3^+, 6)$ & $(2\; 7)(3\; 10)(6\; 11)$ &  $(1\; 8)(4\; 9)(5\; 12)$\\
          $C_{12}(5^+, 6)$ & $(1\; 7)(2\; 8)$ &  $(3\; 9)(4\; 10)$
         \end{tabular}
     \end{table}
\end{proof}
For $C_{12}(3^+, 6)$, we managed to compute the actual quantum automorphism group. Again, we give the proof
in Section~\ref{section:computation_qaut_groups}.
\begin{prop}
    The graphs $C_{12}(2, 5^+)$ and $C_{12}(4, 5^+)$ do not have quantum symmetries.
\end{prop}
The proofs of the above proposition are again very similar to the proof of Proposition~\ref{prop:k2_c6}, but for completeness they are given in Appendix~\ref{appendix:omitted-proofs}.

All in all, the following table summarises the results for semi-circulant graphs:
\begin{table}[H]
    \centering
    \begin{tabular}{l|l|l}
     Graph& Automorphism Group& Quantum Automorphism Group\\
        \hline
     $C_{12}(5^+)$ & $\Z_2 \times S_4$ & ?\\
     $C_{12}(3^+, 6)$ & $H_2 \times S_3$ & $H_2^+ \times S_3$\\
     $C_{12}(5^+, 6)$ & $\Z_2 \times S_4$ & ?\\
     $C_{12}(2, 5^+)$&$D_6$& $D_6$\\
     $C_{12}(4, 5^+)$&$D_6$& $D_6$
    \end{tabular}
\end{table}

\section{Special cases}
\label{section:special_cases}
The following five graphs do not fall in any of the previous subclasses of graphs and therefore have to be considered individually.

\subsection{The Icosahedron}
It was shown by Schmidt in his PhD thesis that the Icosahedron does not have quantum symmetries.
\begin{prop}[\cite{schmidt_2020_1}] 
    The Icosahedron does not have quantum symmetries.
\end{prop}

\subsection{The line graph of $C_6(2)$}
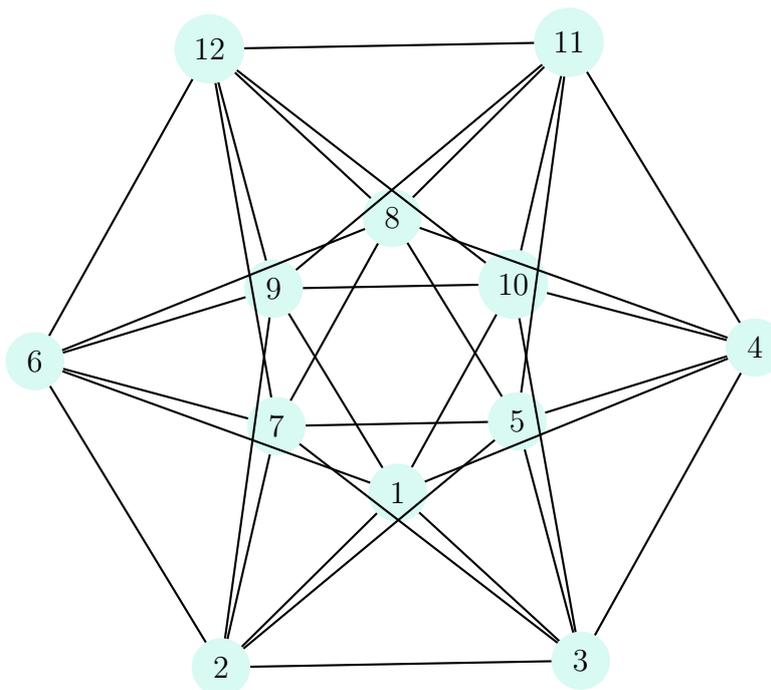
\begin{center}
\begin{tikzpicture}[auto, node distance = 2cm and 2.5cm, on grid, main_node/.style={circle, draw=xkcdAquamarine!15, fill=xkcdAquamarine!15, very thick, minimum size=5mm}]

\node[main_node] (0) at (-0.0842264184643855, -1.1148341614241533) {$1$};
\node[main_node] (1) at (-2.409026107411536, -3.4296008369929822) {$2$};
\node[main_node] (2) at (2.3238704226204305, -3.3420940576863107) {$3$};
\node[main_node] (3) at (4.617715805876051, 0.8098425976657726) {$4$};
\node[main_node] (4) at (1.488602891214901, -0.16917538619353145) {$5$};
\node[main_node] (5) at (-4.857142857142858, 0.6284915720060709) {$6$};
\node[main_node] (6) at (-1.6799853615464624, -0.23164268245813968) {$7$};
\node[main_node] (7) at (-0.15736902349282111, 2.5278328648716952) {$8$};
\node[main_node] (8) at (-1.7171365977648367, 1.5909835052045331) {$9$};
\node[main_node] (9) at (1.43593481414364, 1.6507081029838953) {$10$};
\node[main_node] (10) at (2.1688084622224926, 4.857142857142858) {$11$};
\node[main_node] (11) at (-2.5613385882061754, 4.769178207050151) {$12$};

 \path[draw, thick]
(0) edge node {} (1) 
(0) edge node {} (2) 
(0) edge node {} (3) 
(0) edge node {} (5) 
(0) edge node {} (8) 
(0) edge node {} (9) 
(1) edge node {} (2) 
(1) edge node {} (4) 
(1) edge node {} (5) 
(1) edge node {} (6) 
(1) edge node {} (8) 
(2) edge node {} (3) 
(2) edge node {} (4) 
(2) edge node {} (6) 
(2) edge node {} (9) 
(3) edge node {} (4) 
(3) edge node {} (7) 
(3) edge node {} (9) 
(3) edge node {} (10) 
(4) edge node {} (6) 
(4) edge node {} (7) 
(4) edge node {} (10) 
(5) edge node {} (6) 
(5) edge node {} (7) 
(5) edge node {} (8) 
(5) edge node {} (11) 
(6) edge node {} (7) 
(6) edge node {} (11) 
(7) edge node {} (10) 
(7) edge node {} (11) 
(8) edge node {} (9) 
(8) edge node {} (10) 
(8) edge node {} (11) 
(9) edge node {} (10) 
(9) edge node {} (11) 
(10) edge node {} (11) 
;

\end{tikzpicture}
\end{center}
\begin{prop}
    The graph $L(C_6(2))$ does not have quantum symmetries.
\end{prop}
\begin{proof}
    The commutation of $u_{i 1}$ with  $u_{kl}$ for  $l \in \{4, 5, 6, 7, 8, 11, 12\} $ can be seen by
    just applying Lemma~\ref{lem:choose_q_middle}:
    \begin{table}[H]
        \begin{tabularx}{0.2\linewidth}[t]{|*{4}{X|}}
            \multicolumn{4}{c}{Lemma~\ref{lem:choose_q_middle}}\\
            \hline
            $j$ &$l$ &  $p$ &  $q$\\
             \hline
            $1$&$4$ & $3$ & $9$\\
               & & $5$ & $7$\\
               && $8$&  $2$\\
               &&$10$& $2$\\
               && $11$& $5$\\
               &$5$ & $6$ & $3$\\
               & & $9$ & $12$\\
               & &$10$&$6$\\
               &&$12$& $1$\\
               \hline
        \end{tabularx}
        \quad \quad
        \begin{tabularx}{0.2\linewidth}[t]{|*{4}{X|}}
            \multicolumn{4}{c}{Lemma~\ref{lem:choose_q_middle}}\\
            \hline
            $j$ &$l$ &  $p$ &  $q$\\
             \hline
               $1$&$6$& $2$& $10$\\
               & & $7$& $5$\\
               && $8$& $2$\\
               && $9$& $3$\\
               && $12$& $7$\\
               & $7$& $4$& $2$\\
               & &$9$& $4$\\
               && $10$& $11$\\
               && $11$& $1$\\
               \hline
        \end{tabularx}
        \quad \quad
        \begin{tabularx}{0.2\linewidth}[t]{|*{4}{X|}}
            \multicolumn{4}{c}{Lemma~\ref{lem:choose_q_middle}}\\
            \hline
            $j$ &$l$ &  $p$ &  $q$\\
             \hline
               $1$&$8$& $2$& $10$\\
               && $3$& $9$\\
               && $9$& $3$\\
               &&$10$& $2$\\
               &$11$& $2$& $7$\\
               && $3$& $6$\\
               && $6$& $3$\\
               &&$7$& $1$\\
               \hline
        \end{tabularx}
        \quad\quad
        \begin{tabularx}{0.2\linewidth}[t]{|*{4}{X|}}
            \multicolumn{4}{c}{Lemma~\ref{lem:choose_q_middle}}\\
            \hline
            $j$ &$l$ &  $p$ &  $q$\\
             \hline
               $1$&$12$& $2$& $4$\\
               && $3$& $5$\\
               && $4$& $2$\\
               && $5$& $1$\\
               \hline
        \end{tabularx}
    \end{table}
    The only vertices $l$ for which we still need to show commutativity of  $u_{i 1}$ and $u_{kl}$
    are thus  $l \in \{2, 3, 9, 10\} $. However, since we have the following three automorphisms
    in $G_{aut}(L(C_6(2)))$, we only need to show it for one such $l$ and the rest follows 
    immediately by Lemma~\ref{lem:get_commutation_by_automorphism}:
     \begin{align*}
         \phi &= (2\;3)(4\;6)(5\;7)(9\;10)(11\;12)\\
         \sigma &= (2\;9)(3\;10)(5\;11)(7\;12)\\
         \tau &= (2\;10)(3\;9)(4\;6)(5\;12)(7\;11)
    .\end{align*}
    We will therefore now show commutation of $u_{ i 1 }$ and  $u_{k 2}$. We do this, by first applying
    Lemma~\ref{lem:choose_q_right} and then using Lemma~\ref{lem:choose_q_middle} to see that
    the remaining monomials vanish:
    \begin{table}[H]
        \begin{tabularx}{0.26\linewidth}[t]{|*{3}{X|}l|}
            \multicolumn{4}{c}{Lemma~\ref{lem:choose_q_right}}\\
            \hline
            $j$ &$l$ &  $q$ &  $p$\\
             \hline
              $1$ &$2$& $4$ & $\{1, 3, 5\} $\\
               \hline
        \end{tabularx}
        \quad \quad
        \begin{tabularx}{0.2\linewidth}[t]{|*{4}{X|}}
            \multicolumn{4}{c}{Lemma~\ref{lem:choose_q_middle}}\\
            \hline
            $j$ &$l$ &  $p$ &  $q$\\
             \hline
            $1$&$2$ & $3$ & $6$\\
               && $5$ & $6$\\
               \hline
        \end{tabularx}
    \end{table}
    All in all we thus get that $\QBan(L(C_6(2))) = G_{aut}(L(C_6(2)))$.
\end{proof}
\subsection{The truncated tetraeder} 
We want to show, that the Quantum Automorphism group of $Trunc(K_4)$ is commutative, i. e. $G_{aut}(Trunc(K_4)) = \QBan(Trunc(K_4))$.

\begin{thm}
        \label{thm:TruncK4}
        The truncated tetraeder, $\Gamma = Trunc(K_4)$, shown below, does not have quantum symmetry, i. e. $\QBan(\Gamma)$ is commutative.
\end{thm}
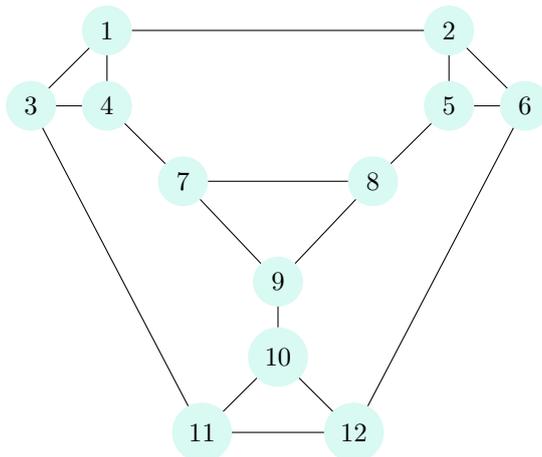
\begin{center}
        \begin{tikzpicture}[scale=0.5, auto, on grid, vertex/.style={circle, draw=xkcdAquamarine!15, fill=xkcdAquamarine!15, very thick, minimum size=3mm}]
            \footnotesize
            \node[vertex] (A) at (1, 1) {1};
            \node[vertex] (B) at (10, 1)  {2};
            \node[vertex] (C) [below left =of A] {3};
            \node[vertex] (D) [below =of A] {4};
            \node[vertex] (E) [below =of B] {5};
            \node[vertex] (F) [below right =of B] {6};
            \node[vertex] (G) [below right =of D] {7};
            \node[vertex] (H) [below left=of E] {8};
            \node[vertex] (I) [below =of $(G)!0.5!(H)$] {9};
            \node[vertex] (J) [below =of I] {10};
            \node[vertex] (K) [below left =of J] {11}; 
            \node[vertex] (L) [below right =of J] {12}; 

            \draw (A) -- (B);
            \draw (A) -- (C);
            \draw (A) -- (D);
            \draw (B) -- (E);
            \draw (B) -- (F);
            \draw (C) -- (D);
            \draw (C) -- (K);
            \draw (D) -- (G);
            \draw (E) -- (F);
            \draw (E) -- (H);
            \draw (F) -- (L);
            \draw (G) -- (H);
            \draw (G) -- (I); 
            \draw (H) -- (I);
            \draw (I) -- (J);
            \draw (J) -- (K);
            \draw (J) -- (L); 
            \draw (K) -- (L); 

        \end{tikzpicture}
\end{center}
Before we prove the above theorem, we want to collect some properties of $\Gamma$.
\begin{propr}
        \noindent
        \begin{itemize}
                \item $\Gamma$ has diameter $3$, i.e. the largest distance between two vertices is $3$.
                \item $\Gamma$ is $3$-regular, i.e. every vertex has exactly $3$ neighbours.
                \item $\Gamma$ consists of $4$ triangles.
                \item Each vertex of $\Gamma$ is connected exactly to two vertices in the same triangle and one vertex in a different triangle.
        \end{itemize}
\end{propr}
We now want to introduce some notation to make talking about the different triangles easier.
\begin{defn}
        \leavevmode
        \begin{enumerate}
                \item We denote by $\mathcal{T}$ the set of the four triangles that make up $\Gamma$.
                \item Let $v \in V(\Gamma)$. We denote by $T(v)$ the triangle containing $v$. 
                \item Let $v \in V(\Gamma)$. We denote by $AT(v)$ the unique triangle that is \emph{adjacent} to $v$, i.e. that is connected to $v$ by an edge but is different from $T(v)$.
                \item Let $v \in V(\Gamma)$. We denote by $\mathcal{NT}(v) := \mathcal{T} \backslash \left\{ T(v), AT(v) \right\}$ the set of \emph{non-adjacent triangles} of $v$.
        \end{enumerate}
\end{defn}
\begin{lem}
        Let $v, w \in V(\Gamma)$ be two vertices. If they are in the same triangle, then their adjacent triangles are distinct:
        \[
            T(v) = T(w) \Longrightarrow AT(v) \neq AT(w)
        .\] 
\end{lem}
\begin{lem}
        \label{lem:d2AT}
        Let $v, w \in V(\Gamma)$ be two vertices with $d(v, w) = 2$. Then it holds that either $v$ is in the adjacent triangle of $w$ or vice versa:
        \[
            d(v, w) = 2 \Longrightarrow v \in AT(w) \lor w \in AT(v)
        .\] 
\end{lem}
\begin{proof}
        If $v$ is not in $AT(w)$, then in order to reach $v$ from $w$, one step has to be made inside $T(w)$. Since $d(v, w) = 2$, the next step, which is the first step out of $T(w)$, will have to 
        reach $v$ already. Thus, by the definition of $AT(v)$, the triangle we just left is $AT(v)$. Since it is also $T(w)$ we get $w \in AT(v)$.
        The proof can be visualized with the following sketch:
        \begin{center}
		\tiny
                \begin{tikzpicture}[node distance = 1cm, auto, on grid, semithick, state/.style={circle, draw=black, fill=black, text=black, inner sep = 0pt, minimum size = 3pt}]
                        \node[state, label=above: $v$] (1) {};
                        \node[state] (2) [right=of 1] {};
                        \node[state] (3) [below=of 1] {};
			\node[state] (4) [left=of 1]{};
			\node[state, label=below:$\vdots$, label=right: $w_2$] (5) [below=of 4]{};
			\node[state, label=left:$\cdots$, label=above: $w_1$] (6) [left=of 4]{};
			\node[state, label=right: $\cdots$, label=above: $w_4$] (7) [right=of 2]{};
			\node[state, label=below: $\vdots$, label=right: $w_3$] (8) [below=of 3]{};
			\draw (1) -- (2);
			\draw (1) -- (3);
			\draw (3) -- (2);
			\draw (1) -- (4);
			\draw (4) -- (5);
			\draw (4) -- (6);
			\draw (5) -- (6);
			\draw (2) -- (7);
			\draw (3) -- (8); 
                \end{tikzpicture}
        \end{center}

\end{proof}
\begin{lem}
        \label{lem:d3NT}
        Let $v \in V(\Gamma)$ be a vertex. All vertices $w \in V(\Gamma)$ that fulfill $d(v, w) = 3$ are in the triangles in $\mathcal{NT}(v) = \mathcal{T} \backslash \left\{ T(v), AT(v) \right\}$.
\end{lem}
\begin{proof}
        Indeed, if $w$ was in $T(v)$, then the distance of $w$ to $v$ would be $1$. On the other hand if $w \in AT(v)$ holds, then either 
	$d(v, w) = 1$, if $w$ is the unique vertex in $AT(v)$ that shares an edge with $v$, or $d(v, w) = 2$, as $w$ is connected to the vertex in $AT(v)$ that shares an edge with $v$.
\end{proof}

We now come to the proof of Theorem~\ref{thm:TruncK4}.
\begin{proof}
        Let $(u_{ij})_{1 \leq i, j \leq n}$ be the generators of $C(\QBan(\Gamma))$. We want to show that $u_{ij} u_{kl} = u_{kl} u_{ij}$ for 
        all $i, j, k, l \in V(\Gamma)$. By Lemma~\ref{lem:distances} it suffices to show the statement for $i, j, k, l$ such that $d(i, k) = d(j, l)$.
        \textbf{\emph{Step 1:}} $d(i, k) = d(j, l) = 1$: 
         Since $\Gamma$ does not contain any quadrangles, by Lemma~\ref{lem:quadrangle} we have
        $\QBan(\Gamma) = \QBic(\Gamma)$ and therefore $u_{ij} u_{kl} = u_{kl} u_{ij}$ for all $d(i, k) = d(j, l) = 1$. 

        \textbf{\emph{Step 2:}} $d(i, k) = d(j, l) = 2$. Let us fix $j \in V(\Gamma)$. Then there are four vertices at distance $2$ from $j$: two of them are in $AT(j)$, reached by following
        the edge into $AT(j)$ and then following either edge within the triangle, and one is in each of the triangles in $\mathcal{NT}(j)$, reached by following either edge within $T(j)$ and then taking the
        single edge that leads out of $T(j)$. Let us sketch the situation:
        \begin{center}
		\tiny
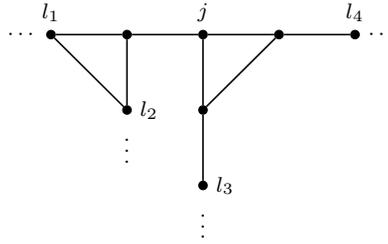
                \begin{tikzpicture}[node distance = 1cm, auto, on grid, semithick, state/.style={circle, draw=black, fill=black, text=black, inner sep = 0pt, minimum size = 3pt}]
                        \node[state, label=above: $j$] (1) {};
                        \node[state] (2) [right=of 1] {};
                        \node[state] (3) [below=of 1] {};
			\node[state] (4) [left=of 1]{};
			\node[state, label=below:$\vdots$, label=right: $l_2$] (5) [below=of 4]{};
			\node[state, label=left:$\cdots$, label=above: $l_1$] (6) [left=of 4]{};
			\node[state, label=right: $\cdots$, label=above: $l_4$] (7) [right=of 2]{};
			\node[state, label=below: $\vdots$, label=right: $l_3$] (8) [below=of 3]{};
			\draw (1) -- (2);
			\draw (1) -- (3);
			\draw (3) -- (2);
			\draw (1) -- (4);
			\draw (4) -- (5);
			\draw (4) -- (6);
			\draw (5) -- (6);
			\draw (2) -- (7);
			\draw (3) -- (8); 
                \end{tikzpicture}
        \end{center}
	In this situation, $l_1$ and $l_2$ are the vertices at distance $2$ that are in $AT(j)$ and $l_3$ and $l_4$ are each in one of the triangles of $\mathcal{NT}(j)$. Note, that while 
	$l_3$ and $l_4$ are not in $AT(j)$, it holds that $j \in AT(l_3)$ and $j \in AT(l_4)$ by Lemma~\ref{lem:d2AT}.
	Let now $l$ be one of $l_1$ and $l_2$, i.e. $l$ is a vertex such that $d(j, l) = 2$ and $l \in AT(j)$. We will show, that for all $i, k \in V(\Gamma)$ with $d(i, k) = 2$ the following holds:
	\[
		u_{ij} u_{kl} = u_{ij} u_{kl} u_{ij}.
	\]
	Then, by Lemma~\ref{lem:adj_commute}, we have $u_{ij} u_{kl} = u_{kl} u_{ij}$.
	We want to use Lemma~\ref{lem:choose_q_right}, with distance $m = 2$ and choosing $s = t = 1$, which we can do since $\QBan(\Gamma) = \QBic(\Gamma)$, we want to find a vertex $q$ for which
	it holds that $d(j, q) = 1$ and $d(q, l) = 1$. If we choose $q$ such that the only vertex $p$ fulfilling $d(p, l) = 2$ and $d(p, q) = 1$ is $j$, then we get the desired result. 
	Let us again take a look at the situation we have.
	\begin{center}
	\tiny
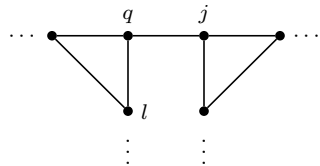
		\begin{tikzpicture}[node distance = 1cm, auto, on grid, semithick, state/.style={circle, draw=black, fill=black, text=black, inner sep = 0pt, minimum size = 3pt}]
			\node[state, label=above: $j$] (1) {};
			\node[state, label=right: $\cdots$] (2) [right=of 1] {};
			\node[state, label=below: $\vdots$] (3) [below=of 1] {};
			\node[state, label=above: $q$] (4) [left=of 1]{};
			\node[state, label=below:$\vdots$, label=right: $l$] (5) [below=of 4]{};
			\node[state, label=left:$\cdots$] (6) [left=of 4]{};
			\draw (1) -- (2);
			\draw (1) -- (3);
			\draw (3) -- (2);
			\draw (1) -- (4);
			\draw (4) -- (5);
			\draw (4) -- (6);
			\draw (5) -- (6);
		\end{tikzpicture}
	\end{center} 
	Choosing $q$ as indicated in the sketch above, we see that there are three vertices adjacent to $q$, one of which is $l$ and one is $j$. The third vertex is in the same triangle as
	$l$ and thus is at distance $1$ from $l$. Therefore, the only vertex adjacent to $q$ and at distance $2$ from $l$ is in fact $j$. Lemma~\ref{lem:choose_q_right} yields
	\[ 
            u_{ij} u_{kl} = u_{ij} u_{kl} \sum_{\substack{p; \, d(p, l) = 2\\ (p, q) \in E(\Gamma)}} u_{ip} = u_{ij} u_{kl} u_{ij}.
	\]
	Thus, for all $i, j, k, l \in V(\Gamma)$ with $d(i, k) = d(j, l) = 2$ and $l \in AT(j)$, $u_{ij}$ and $u_{kl}$ commute. As noted above, it holds that if $l \notin AT(j)$, we have
	$j \in AT(l)$. Therefore, we get that $u_{ij}$ and $u_{kl}$ commute for all $i, j, k, l$ with $d(i, k) = d(i, j) = 2$.

	\textbf{\emph{Step 3:}} $d(i, k) = d(j, l) = 3$. Fixing again $j \in V(\Gamma)$, we know by Lemma~\ref{lem:d3NT} that all vertices $l$ with $d(j, l) = 3$ are in the triangles in $\mathcal{NT}(j)$, in other words
	$l \notin \left\{T(j), AT(j)\right\}$. Note moreover, that $\mathcal{NT}(j) = \left\{ AT(s), AT(t)\right\}$ if $T(j) = \left\{j, s, t \right\}$. Indeed, there are $4$ triangles in total, and since $T(j) = T(s)$ implies
	$AT(j) \neq AT(s)$, we can write the set of all triangles as $\mathcal{T} = \left\{ T(j), AT(j), AT(s), AT(k)\right\}$. 
	Let now $l$ be a fixed vertex at distance $3$ from $j$ and let $v \in T(j)$ be the vertex that satisfies $l \in AT(v)$. We are thus in the situation sketched below.
	\begin{center}
	\tiny
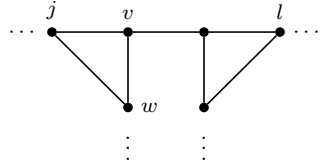
		\begin{tikzpicture}[node distance = 1cm, auto, on grid, semithick, state/.style={circle, draw=black, fill=black, text=black, inner sep = 0pt, minimum size = 3pt}]
			\node[state] (1) {};
			\node[state, label=right: $\cdots$, label=above: $l$] (2) [right=of 1] {};
			\node[state, label=below: $\vdots$] (3) [below=of 1] {};
			\node[state, label=above: $v$] (4) [left=of 1]{};
			\node[state, label=below:$\vdots$, label=right: $w$] (5) [below=of 4]{};
			\node[state, label=left:$\cdots$, label=above: $j$] (6) [left=of 4]{};
			\draw (1) -- (2);
			\draw (1) -- (3);
			\draw (3) -- (2);
			\draw (1) -- (4);
			\draw (4) -- (5);
			\draw (4) -- (6);
			\draw (5) -- (6);
		\end{tikzpicture}
	\end{center} 
    In this situation, we can use Lemma~\ref{lem:choose_q_right} with $s = 1$, $t = 2$, putting $q := v$, to get for all $i, k \in V(\Gamma)$, $d(i, k) = 3$ that
	\[
            u_{ij} u_{kl} = u_{ij} u_{kl} \sum_{\substack{p; \, d(p, l) = 3\\ (p, q) \in E(\Gamma)}} u_{ip}
        .\]
     Note that we can apply Lemma~\ref{lem:choose_q_right} here as we have shown above, that
	$u_{xy}$ and $u_{x^\prime y^\prime}$ commute with $d(x, x^\prime) = d(y, y^\prime) = 2$ and thus in particular for all $a$ with $d(a, k) = 2$ it holds that $u_{kl}$ and $u_{aq}$ commute.
     If we now have $d(w, l) \neq 3$ then we are already done, since we then get by the above application of Lemma~\ref{lem:choose_q_right}, that $u_{ij} u_{kl} = u_{ij} u_{kl} u_{ij}$, since then 
     $j$ is the only vertex connected to $q$ and at distance $3$ to $l$. Thus $u_{ij}$ and $u_{kl}$ commute by Lemma~\ref{lem:adj_commute}.

    If however $d(w, l) = 3$, then both $j$ and $w$ are connected to $q$ and at distance $3$ to $l$ and Lemma~\ref{lem:choose_q_right} thus yields
    \[
        u_{ij} u_{kl} = u_{ij} u_{kl} \left( u_{ij} + u_{iw} \right)
    .\] 
    It remains to show, that if $d(w, l) = 3$ we have $u_{ij} u_{kl} u_{iw} = 0$ for all $d(i, k) = 3$, since then $u_{ij} u_{kl} = u_{ij} u_{kl} u_{ij}$ and thus $u_{ij}$ and $u_{kl}$ commute by Lemma~\ref{lem:adj_commute}.
    We want to apply Lemma~\ref{lem:choose_q_middle} to $i, j, k, l$ and $w$ in order to show this. We thus have $m = 3$. The lemma is applicable, since we have $w \neq j$ and $d(w, l) = 3$. We want to find a vertex 
    $q$ and a distance $s$, such that $d(q, l) = s$ and $d(j, q) \neq d(q, w)$ and that moreover fulfills that $l$ is the only vertex satisfying $d(q, l) = s$, $d(l, j) = 3$ and $d(l, p) = 3$. If we find such a $q$ and $s$
    we are done, since then, by Lemma~\ref{lem:choose_q_middle}, $u_{ij} u_{kl} u_{iw} = 0$.

    We put $q$ as the unique vertex in $AT(j)$ that is adjacent to $j$ and $s := d(q, l)$. We are thus in the situation sketched below:
	\begin{center}
	\tiny
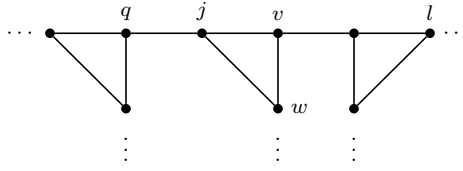
		\begin{tikzpicture}[node distance = 1cm, auto, on grid, semithick, state/.style={circle, draw=black, fill=black, text=black, inner sep = 0pt, minimum size = 3pt}]
			\node[state] (1) {};
			\node[state, label=right: $\cdots$, label=above: $l$] (2) [right=of 1] {};
			\node[state, label=below: $\vdots$] (3) [below=of 1] {};
			\node[state, label=above: $v$] (4) [left=of 1]{};
			\node[state, label=below:$\vdots$, label=right: $w$] (5) [below=of 4]{};
			\node[state, label=above: $j$] (6) [left=of 4]{};
            \node[state, label=above: $q$] (7) [left=of 6] {};
            \node[state, label=left: $\cdots$, ] (8) [left=of 7] {};
            \node[state, label=below: $\vdots$] (9) [below=of 7] {};
            \draw (6) -- (7);
            \draw (7) -- (8);
            \draw (7) -- (9);
            \draw (8) -- (9);
			\draw (1) -- (2);
			\draw (1) -- (3);
			\draw (3) -- (2);
			\draw (1) -- (4);
			\draw (4) -- (5);
			\draw (4) -- (6);
			\draw (5) -- (6);
		\end{tikzpicture}
	\end{center} 
    Then obviously $s = d(q, l)$ holds. Moreover, $1 = d(j, q) \neq d(w, q) = 2$ holds, since $T(q) \in \mathcal{NT}(w)$.
    It remains to show, that for any vertex $x \in V(\Gamma)$ the following holds:
    \[
            d(x, q) = s \mbox{ and } d(x, j) = 3 \mbox{ and } d(x, w) = 3 \Longrightarrow t = l.
    \]
    Let therefore such a vertex $x$ be given. We claim that $x \in AT(v)$.

    Indeed, if $x \notin AT(v)$, then either $x$ would have to be in $T(v)$ or $T(x)$ would have to be in $\mathcal{NT}(v)$. But if $x \in T(v) = T(j)$, then $d(x, j) \leq 1$ which 
    is a contradiction to $d(x, j) = 3$. 
    If on the other hand we have $T(x) \in \mathcal{NT}(v) = \left\{ AT(j), AT(w) \right\}$, then we have either $x \in AT(j)$ or $x \in AT(w)$. If $x \in AT(j)$, then $d(x, j) \leq 2$ as was argued above, which is again a contradiction
    to $d(x, j) = 3$. If $x \in AT(w)$ we have by the same argument that $d(x, w) \leq 2$, which is a contradiction to $d(x, w) = 3$.

    We thus have $x \in AT(v) = T(l)$. Let us look at another sketch of the situation, in order to name the relevant vertices.
	\begin{center}
	\tiny
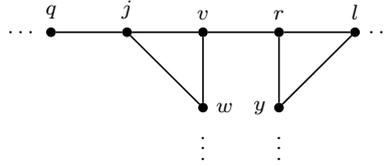
		\begin{tikzpicture}[node distance = 1cm, auto, on grid, semithick, state/.style={circle, draw=black, fill=black, text=black, inner sep = 0pt, minimum size = 3pt}]
            \node[state, label=above: $r$] (1) {};
			\node[state, label=right: $\cdots$, label=above: $l$] (2) [right=of 1] {};
            \node[state, label=below: $\vdots$, label=left: $y$] (3) [below=of 1] {};
			\node[state, label=above: $v$] (4) [left=of 1]{};
			\node[state, label=below:$\vdots$, label=right: $w$] (5) [below=of 4]{};
			\node[state, label=above: $j$] (6) [left=of 4]{};
            \node[state, label=above: $q$, label=left: $\cdots$] (7) [left=of 6] {};
            \draw (6) -- (7);
			\draw (1) -- (2);
			\draw (1) -- (3);
			\draw (3) -- (2);
			\draw (1) -- (4);
			\draw (4) -- (5);
			\draw (4) -- (6);
			\draw (5) -- (6);
		\end{tikzpicture}
	\end{center} 
    We know so far, that $x \in \left\{ r, y, l \right\}$. Since however $d(j, r) = d(w, r) = 2 \neq 3$ we can narrow it down to $x \in \left\{ y, l \right\}$.
    It remains to show that $d(q, l) \neq d(q, y)$, since then we have shown that $x = l$.
    In the sketch above, we see that $T(q) \in \mathcal{NT}(r)$ and thus we have either $q \in AT(l)$ or $q \in AT(y)$. 
    We will show the claim for $q \in AT(l)$, the other situation can be shown similarly.
    If $q$ is in $AT(l)$, then we have $d(q, l) \leq 2$, and since the unique vertex connected to $q$ that is in a different triangle is already $j$ we even know $d(q, l) = 2$. 
    We will now show, that $d(q, y) = 3$. Since $\Gamma$ has diameter $3$, we already know, that $d(q, y) \leq 3$.
    Since a shortest path from $l$ to $q$ is of length $2$, any path from $y$ to $q$ via $l$ will have at least length $3$ and there is also a path of length $3$ via $l$.
    Moreover, any path from $y$ to $q$ via $r$ will be at least of length $4$, as can be seen in the sketch above. 
    The last possibility to reach $q$ from $y$ is thus by first traversing $AT(y)$. But since $q \notin AT(y)$, which is the case as $q \in AT(l)$, any such path will consist of first taking the 
    edge from $y$ to $AT(y)$, then taking one edge within $AT(y)$ and then leaving $AT(y)$ via a third edge, and thus such a path is at least of length $3$ again.
    We thus conclude $d(q, l) = 2 \neq 3 = d(q, y)$.

    Therefore, we can apply Lemma~\ref{lem:choose_q_middle} and get 
    \[
            u_{ij} \left( \sum_{\substack{x; \, d(x, j) = d(x, w) = 3 \\ d(q, x) = d(q, l)}} u_{kx} \right) u_{iw} = u_{ij} u_{kl} u_{iw} = 0.
    \]
    We thus conclude that $u_{ij}$ and $u_{kl}$ commute for any $i, j, k, l \in V(\Gamma)$ such that $d(i,k) = d(j, l) = 3$.

    Since the diameter of $\Gamma$ is $3$, we now know that all generators of $\QBan(\Gamma)$ commute and thus $\QBan(\Gamma) = G_{aut}(\Gamma)$.

\end{proof}

\subsection{The distance 3 graph of the truncated tetraeder}
In this section, we will show that the distance $3$ graph of the truncated tetraeder, $Antip(Trunc(K_4))$, does not have quantum symmetry.
\begin{thm}
        \label{thm:antip_trunc_k4}
        The distance $3$ graph of the truncated tetraeder, $\Gamma = Antip(Trunc(K_4))$, shown below, does not have quantum symmetry, i.e. $\QBan(\Gamma)$ is commutative.
\end{thm}
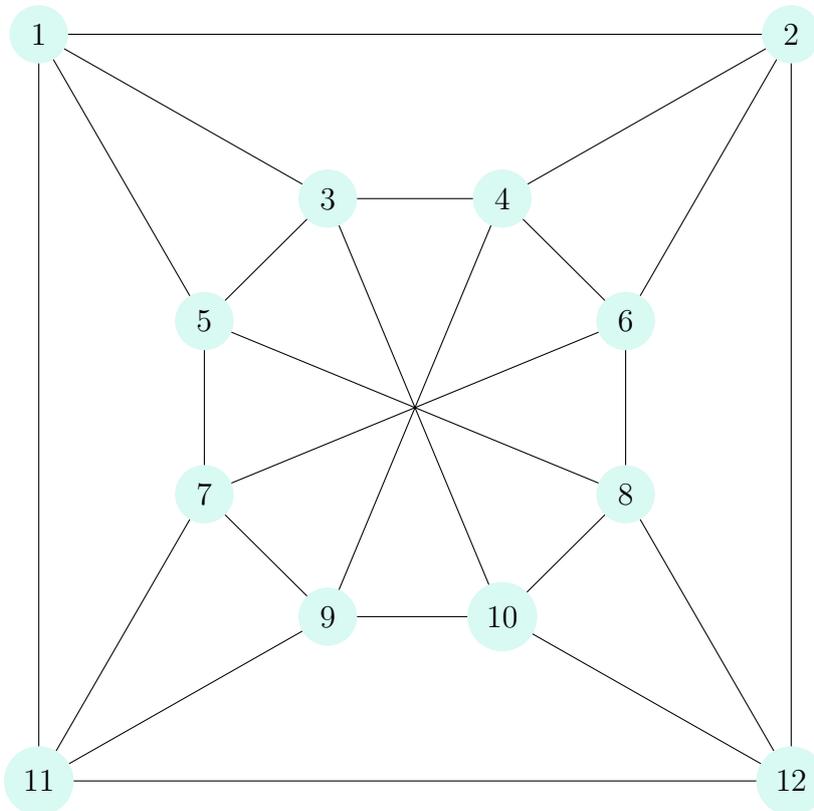
\begin{center}
        \begin{tikzpicture}[auto, node distance = 2cm and 1.5cm, on grid, vertex/.style={circle, draw=xkcdAquamarine!15, fill=xkcdAquamarine!15, very thick, minimum size=5mm}]
            \node[vertex] (A) at (135:7cm) {1};
            \node[vertex] (B) at (45:7cm)  {2};
            \node[vertex] (K) at (225:7cm) {11}; 
            \node[vertex] (L) at (315:7cm) {12}; 
            \node[vertex] (D) at (22.5 + 45:3cm) {4};
            \node[vertex] (C) at (22.5 + 90:3cm) {3};
            \node[vertex] (E) at (22.5 + 135:3cm) {5};
            \node[vertex] (G) at (22.5 + 180:3cm) {7};
            \node[vertex] (I) at (22.5 + 225:3cm) {9};
            \node[vertex] (J) at (22.5 + 270:3cm) {10};
            \node[vertex] (H) at (22.5 + 315:3cm) {8};
            \node[vertex] (F) at (22.5 + 360:3cm) {6};

            \draw (A) -- (B);
            \draw (A) -- (C);
            \draw (A) -- (K);
            \draw (A) -- (E);
            \draw (B) -- (D);
            \draw (B) -- (F);
            \draw (B) -- (L);
            \draw (C) -- (E);
            \draw (C) -- (D);
            \draw (C) -- (J);
            \draw (D) -- (F);
            \draw (D) -- (I);
            \draw (E) -- (G);
            \draw (E) -- (H);
            \draw (F) -- (G);
            \draw (F) -- (H);
            \draw (G) -- (I);
            \draw (G) -- (K); 
            \draw (H) -- (J);
            \draw (H) -- (L);
            \draw (I) -- (J);
            \draw (I) -- (K);
            \draw (J) -- (L);
            \draw (K) -- (L); 

        \end{tikzpicture}
\end{center}
We collect some properties about the graph before proving the theorem.
\begin{propr}
        \label{propr:antip_trunc_k4}
        \leavevmode
        \begin{itemize}
                \item $\Gamma$ has diameter $2$.
                \item $\Gamma$ is $4$-regular.
                \item The graph is made up of $4$ disjoint triangles, i.e. for each vertex $v \in V$, it holds 
                        that there are two other vertices $w$ and $x$, such that $\left\{ v, w, x \right\}$ forms
                        a triangle and for all other vertices $y \in V \backslash \left\{v, w, x\right\}$ it holds that
                        if $(v, y) \in E$ then $v$ and $y$ have no common neighbour.
                \item Every vertex $v$ is  also part of three distinct quadrangles. On the one hand, 
                        it is part of one containing vertices from all $4$ triangles and 
                        on the other hand it is part of two triangles each containing $2$ vertices from 
                        $T(v)$ and $2$ vertices from another triangle. Both of the latter vertices share a
                        triangle however.
                \item All vertices that are not connected have either one or two common neighbours.
                \item For every pair of vertices that are in the same triangle, there is a unique quadrangle containing the same two vertices.
                \item For every pair of vertices that are in distance $2$, there is a unique quadrangle containing the same two vertices.
                \item For every pair of adjacent vertices that do not share a triangle, there are $2$ quadrangles containing both vertices.
        \end{itemize}
\end{propr}
\begin{defn}
        \label{def:antip_quadrangles_triangles}
        \leavevmode
        \begin{enumerate}
                \item If $v \in V$ is a vertex, we denote the unique triangle containing $v$ as $T(v)$.
                \item If $v, w \in V$ and  $d(v, w) = 2$, we want to count the number of distinct triangles
                        connected by a possible quadrangle containing $v$ and $w$. For this, we first note, that
                        if a quadrangle $\mathcal{Q}(v, w)$ containing $v$ and $w$ exists, then it is unique, i.e. there are no 
                        two distinct quadrangles containing both $v$ and $w$, for $d(v, w) = 2$, as was noted in the Properties~\ref{propr:antip_trunc_k4} above.
                        We now put for an existing quadrangle $\mathcal{Q}(v, w)$ consisting of the vertices 
                        $\left\{ v, w, x, y \right\}$ the set of triangles connected by this quadrangle as
                        $\mathcal{T}(\mathcal{Q}(v, w)) := \left\{ T(v), T(w), T(x), T(y) \right\}$.
                        With this, we define
                        \[
                        Q(v, w) = \begin{cases}
                                        0, \text{ if } v \text{ and } w \text{ do not share a quadrangle}\\
                                        |\mathcal{T}(\mathcal{Q}(v, w))| \text{ otherwise.}
                                  \end{cases}
                        \] 
                        Note, that by the properties of $\Gamma$ listed above, we have that 
                        $Q(v, w) \in \left\{ 0, 2, 4 \right\}$ for all vertices $v$ and $w$ at distance $2$.
        \end{enumerate}
        
\end{defn}
\begin{lem}
        \label{lem:antip_trunc_k4_q0}
        Let $v$ and $w$ be two vertices in $\Gamma$ such that $d(v, w) = 2$ and $Q(v, w) = 0$. Then the unique common neighbour $p$ of
        $v$ and $w$ fulfills that we have either $p \in T(v)$ or $p \in T(w)$.
\end{lem}
\begin{proof}
        The fact that $p$ exists comes from the fact, that $d(v, w) = 2$ and therefore, there is a path of length $2$ from $v$ to $w$.
        Since however there is no quadrangle containing both $v$ and $w$, there can not be two common neighbours.

        The fact, that $p \in T(v)$ or $p \in T(w)$ holds, comes from the fact that each vertex has $4$
        neighbours, two of which are in the same triangle as the vertex itself, and the other two are in two distinct different triangles.
        Thus, if we take one edge incident to $v$ that leads to a triangle different from $T(v)$, the vertex we reach has only one neighbour 
        apart from $v$ that is in a different triangle than itself. Therefore, there is only one option to take another edge to a completely new triangle.
        Repeating this twice more, we again reach $v$, and since each time there was only one choice, we have traversed the unique quadrangle containing 
        $v$ and connecting all $4$ triangles. Since $v$ and $w$ do not share a quadrangle however, none of the vertices we traversed was $w$, and 
        thus on the path from $v$ to $w$, at least one of the two steps has to be made within the same triangle.
\end{proof}
\begin{lem}
        \label{lem:antip_q0_unique}
        Let $v \neq w$ be two vertices in $\Gamma$ such that $T(v) = T(w)$, in particular we have $v \sim w$.
        There is at most one vertex $s \in V$ such that $d(v, s) = 2$, $Q(v, s) = 0$ and $w \sim s$ hold.
\end{lem}
\begin{proof}
        Observe, that both $v$ and $w$ have $4$ neighbours, as $\Gamma$ is $4$-regular, and $2$ of these neighbours are in the same triangle.
        Thus, each of $v$ and $w$ have two neighbours in triangles that are distinct from $T(v) = T(w)$. We note, that as soon as $v$ has 
        a neighbour $t$ in the same triangle as the neighbour $s$ of $w$, we have that $\left\{ v, w, s, t \right\}$ forms a quadrangle, as
        $s \sim t$ follows from the fact that they share a triangle. Note moreover, that $s = t$ can not be the case, since otherwise
        $\left\{ v, w, s \right\}$ would form a triangle, which contradicts $T(s) \neq T(v)$. Therefore, we have that $d(v, t) = 2$ and 
        $Q(v, t) > 0$.

        If we have on the other hand that as soon as a neighbour $s$ of $w$ is in a triangle, in which $v$ does not have a neighbour, then
        $d(v, s) = 2$ and $Q(v, s) = 0$.

        However, as there are only $3$ triangles in $\Gamma$ distinct from $T(v)$, and as both $v$ and $w$ have each $2$ neighbours in distinct
        triangles, at least one pair of neighbours has to be in the same triangle. Therefore, at most one neighbour $s$ of $v$ fulfills $d(v, s) = 2$ and
        $Q(v, s) = 0$.
\end{proof}

\begin{proof}{of Theorem~\ref{thm:antip_trunc_k4}}
        Let $(u_{ij})_{1 \leq i, j \leq n}$ be the generators of $C(\QBan(\Gamma))$. We will prove that $u_{ij} u_{kl} = u_{kl} u_{ij}$ in several steps. Recall, that by Lemma~\ref{lem:distances}, we only need to consider the case where $d(i, k) = d(j, l)$, since otherwise the product is already
        $0$ and thus the generators commute.
       \begin{pfsteps}
       \item \label{step:dist1_1} Let first $d(i, k) = d(j, l) = 1$. Let now $i$ and $k$ be in the same triangle and let also $j$ and $l$ share a triangle. Then, by the Properties~\ref{propr:antip_trunc_k4} of $\Gamma$, $i, j, k$ and $l$ fulfill the premise of Lemma~\ref{lem:one_common_neighbour_gen}.
        We thus get $u_{ij} u_{kl} = u_{kl} u_{ij}$ by Lemma~\ref{lem:one_common_neighbour_gen}.
        
        If we have that one of the pairs of vertices $\left\{ i, k \right\}$ or $\left\{ j, l \right\}$ share a triangle and the other does not, we are in the premise of Corollary~\ref{cor:one_triangle_one_not} and therefore we get $u_{ij} u_{kl} = u_{kl} u_{ij} = 0$ for these choices of $i, j, k$ and $l$.

        For $d(i, k) = d(j, l) = 1$, we now still need to show the case where neither $i$ and $k$ nor $j$ and $l$ share a triangle. However, in order to show the claim for this case, we will first need to consider the case of distance $2$.

        Next, we will consider vertices, such that $d(i, k) = d(j, l) = 2$. Here we will make a case distinction on $Q(i, k)$ and $Q(j, l)$ as defined above in Definition~\ref{def:antip_quadrangles_triangles}.

        \item \label{step:dist2_Q2_4} We first consider the cases where $Q(i, k) \neq Q(j, l)$. In these cases, we will prove $u_{ij} u_{kl} = 0$.
         Let therefore $Q(i, k) = 2$, $Q(j, l) = 4$. Let $q$ be a common neighbour of $j$ and $l$. Then $T(q) \notin \left\{ T(j), T(l) \right\}$ 
        since $Q(j, l) = 4$ and $q$ lies in the quadrangle connecting $j$ and $l$. For any common neighbour $s$ of $i$ and $k$ however, 
        we have either $s \in T(i)$ or $s \in T(k)$, since $s$ is in the quadrangle containing $i$ and $k$ and as $Q(i, k) = 2$ there 
        are vertices from only two distinct triangles in this quadrangle, which are $T(i)$ and $T(k)$. Since $i$ and $k$ have exactly $2$ common neighbours, we call them $s_1$ and $s_2$,
        where $s_1 \in T(i)$ and $s_2 \in T(k)$. We thus get
        \begin{align}
                \label{eqn_1_antip}
            u_{ij} u_{kl} = u_{ij} \sum_{s \in V} u_{sq} u_{kl} = u_{ij} \sum_{\substack{s\in V\\s \sim i\\ s \sim k}} u_{sq} u_{kl}
             = u_{ij} u_{s_1q} u_{kl} + 
             u_{ij} u_{s_2q} u_{kl}
        .\end{align}
        Since however, by Step~\ref{step:dist1_1}, we know that $u_{ab} u_{cd} = 0$ if $T(a) = T(c)$ but $T(b) \neq T(d)$ or vice versa, we get
        \[
            u_{ij} u_{s_1q} = 0
        \] 
        since $q \notin T(j)$ and
        \[
            u_{s_2q} u_{kl} = 0
        \] 
        since $q \notin T(l)$. This together with \ref{eqn_1_antip} yields
        \[
            u_{ij} u_{kl} = 0
        .\] 
        The same argument can be used to show that $u_{ij} u_{kl} = 0$ if $Q(i, k) = 4$ and $Q(j, l) = 2$.

        \item \label{step:dist2_Q0_4} Let now $Q(i, k) = 0$ and $Q(j, l) = 4$. Since $Q(i, k) = 0$, there is no quadrangle containing both $i$ and $k$. 
        Then $i$ and $k$ have exactly one common
        neighbour, let us call it $p$. By Lemma~\ref{lem:antip_trunc_k4_q0}, we have, that either $p \in T(i)$ or $p \in T(k)$. 
        Let now $q$ be one of the common neighbours of $j$ and $l$. As was argued in Step~\ref{step:dist2_Q2_4}, we have that $T(q) \notin \left\{ T(j), T(l) \right\}$.
        We get
        \[
            u_{ij} u_{kl} = u_{ij} \sum_{s \in V} u_{sq} u_{kl} = u_{ij} u_{pq} u_{kl} 
        \] 
        as $q$ is a common neighbour of $j$ and $l$ and $p$ is the only common neighbour of $i$ and $k$.
        If now $q \in T(i)$, we get by Step~\ref{step:dist1_1} that 
        \[
            u_{ij} u_{pq} = 0
        \] 
        as $p \notin T(j)$ and similarly, if $q \in T(k)$ we get
        \[
            u_{pq} u_{kl} = 0
        .\] 
        Thus, we get
        \[
            u_{ij} u_{kl} = u_{ij} u_{pq} u_{kl} = 0
        .\] 
        We can show similarly, that $u_{ij} u_{kl} = 0$ if $Q(i, k) = 4$ and $Q(j, l) = 0$.

        \item \label{step:dist2_Q0_2} Let $Q(i, k) = 0$ and $Q(j, l) = 2$ and let $p$ be the unique common neighbour of $i$ and $k$.
        We argue as in Step~\ref{step:dist2_Q0_4} that either $p \in T(i)$ or $p \in T(k)$ holds. 
        We know moreover that $j$ and $l$ have two common neighbours, let us call them $q_1$ and $q_2$. We assume without loss of generality, that
        $q_1 \in T(j)$ and $q_2 \in T(l)$. 
        We know by Step~\ref{step:dist1_1}, that $p \in T(i)$ implies $u_{ij} u_{p q_1} = u_{p q_1} u_{ij}$, while $p \in T(k)$ implies
        $u_{kl} u_{p q_2} = u_{p q_2} u_{kl}$. In both cases, we can apply Lemma~\ref{lem:dist_two_one_vs_two_common_neighbours} to get that
        $u_{ij} u_{kl} = 0$. 
        Similarly, we get that $u_{ij} u_{kl} = 0$ if $Q(i, k) = 2$ and $Q(j, l) = 0$ holds.

        All in all, we see that \[u_{ij} u_{kl} = 0 = u_{kl} u_{ij}\] holds whenever we have $Q(i, k) \neq Q(j, l)$.

        \item \label{step:dist2_Q4_4} We now consider the three cases for $Q(i, k) = Q(j, l)$.
        Let first $Q(i, k) = Q(j, l) = 4$. Let now $p \in V \backslash \left\{ j\right\}$ be a vertex that is distinct from $j$.
        We want to show that for any such $p$, we have
        \[
            u_{ij} u_{kl} u_{ip} = 0
        \] 
        since then we get
        \[
            u_{ij} u_{kl} = u_{ij} u_{kl} \sum_{v \in V} u_{iv} = u_{ij} u_{kl} u_{ij}
        \] 
        and then, by Lemma~\ref{lem:adj_commute}, $u_{ij}$ and $u_{kl}$ commute.
        First, we note, that if $p = l$, we have $u_{kl} u_{il} = 0$, since $k \neq i$.
        Next, if $d(l, p) = 1$, we have $u_{kl} u_{ip} = 0$, since $d(k, i) = 2$, and then the statement follows from Lemma~\ref{lem:distances}.
        The last case is thus $d(l, p) = 2$. We note, that by the Properties~\ref{propr:antip_trunc_k4} of $\Gamma$, $l$ is part of exactly
        one quadrangle connecting all $4$ triangles, i.e. there is only one vertex $v$, such that $d(l, v) = 2$ and $Q(l, v) = 4$. 
        Since $j$ already fulfills both of these properties, we know that $Q(l, p) \neq 4$. Therefore, we have $Q(k, i) \neq Q(l, p)$ and 
        thus $u_{kl} u_{ip} = 0$.

        \item \label{step:dist2_Q2_2} We now consider $Q(i, k) = Q(j, l) = 2$. 
        Let $p$ be again a vertex in $V \backslash\left\{ j \right\}$. As above, we want to show
        \[
            u_{ij} u_{kl} u_{ip} = 0
        .\] 
        Again, we already get $u_{kl} u_{ip} = 0$ if $d(l, p) \in \left\{ 0, 1 \right\}$.
        Let therefore $d(l, p) = 2$. If now $Q(l, p) \neq 2$, we get $u_{kl} u_{ip} = 0$ by Steps~\ref{step:dist2_Q2_4} and \ref{step:dist2_Q0_2}, as $Q(i, k) = 2$. 
        We thus consider $Q(l, p) = 2$.
        Let now $q$ be a common neighbour of $i$ and $k$ such that $q \in T(k)$ and let $s$ and $t$ be the common neighbours of 
        $j$ and $l$, such that $s \in T(j)$ and $t \in T(l)$. In particular, we have that $q \notin T(i)$.
        We denote by $\mathcal{Q}(j, l)$ the unique quadrangle containing $j$ and $l$ and by $\mathcal{Q}(l, p)$ the unique quadrangle containing
        $l$ and $p$. We now argue, that $d(t, p) \neq 1$.
        If $t$ and $p$ were connected, there would be a quadrangle consisting of vertices $l \sim t \sim p \sim v \sim l$, with the fourth vertex being a common neighbour $v \neq t$
        of $l$ and $p$. Such a vertex $v$ exists since $Q(l, p) = 2$. Since we know by the Properties~\ref{propr:antip_trunc_k4} of $\Gamma$, that two vertices in the same
        triangle only share exactly one quadrangle, we thus get a contradiction, since we already have $t \in \mathcal{Q}(j, l) = l \sim t \sim j \sim s \sim l$.
        Therefore, $d(t, p) \neq 1$.
        We compute
        \[
            u_{ij} u_{kl} u_{ip} = u_{ij} \sum_{v \in V} u_{qv} u_{kl} u_{ip} = u_{ij} \sum_{\substack{v \in V\\ v \sim j\\ v\sim l}} u_{qv} u_{kl} u_{ip}
            = \underbrace{u_{ij} u_{qs} u_{kl} u_{ip}}_{\substack{=0 \text{ since } q \notin T(i)\\ \text{but } s \in T(j)}} + u_{ij} u_{qt} u_{kl} u_{ip}
        .\] 
        Since $q \in T(k)$ and $t \in T(l)$ holds however, we get by Step~\ref{step:dist1_1}, that $u_{qt}$ and $u_{kl}$ commute, and we can 
        continue the above calculation:
        \[
            u_{ij} u_{qt} u_{kl} u_{ip} = u_{ij} u_{kl} u_{qt} u_{ip}
        .\] 
        However, $q$ is a common neighbour of $i$ and $k$, which means we have $d(i, q) = 1$, but as we argued above $d(t, p) \neq 1$, and 
        thus, by Lemma~\ref{lem:distances}, we have $u_{qt} u_{ip} = 0$ and thus
        \[
            u_{ij} u_{kl} u_{ip} = u_{ij} u_{kl} u_{qt} u_{ip} = 0
        .\] 
        Therefore, we get $u_{ij} u_{kl} = u_{kl} u_{ij}$ by Lemma~\ref{lem:adj_commute}.

        \item \label{step:dist1_2dist_triangles} Before we can go on to the case where $Q(i, k) = Q(j, l) = 0$, we need to consider the last case of $d(i, k) = d(j, l) = 1$.
        Let therefore now be $d(i, k) = d(j, l) = 1$ and $T(i) \neq T(k)$ and $T(j) \neq T(l)$. 
        We will now show, that $u_{ij} u_{kl} u_{ip} = 0$ for $p \neq j$. Observe, that for $d(l, p) \neq 1$, we already get the statement.
        Moreover, if $T(p) = T(l)$, then we have $u_{kl} u_{ip} = 0$ by Step~\ref{step:dist1_1}, since $T(k) \neq T(i)$ and we are done.

        Let therefore $T(p) \neq T(l)$.
        We observe, that there is a 
        quadrangle, containing both $l$ and $j$, that connects $2$ triangles. 
        We denote by $q$ the vertex in this quadrangle, that is 
        in the same triangle as $t$. Then we have in particular $d(q, l) = 2$ and $Q(q, l) = 2$. 
        Then, we have $d(p, q) \neq 1$: Assume $p \sim q$. Then $j \sim q \sim p \sim l \sim j$ would form a quadrangle. 
        However, if $p$ were in $T(j)$, then $p$ would be a common neighbour of $j$ and $l$. Since such a common neighbour does not exist, we conclude
        $T(p) \neq T(j)$.
        Since we moreover have $T(j) = T(q)$, but $T(j) \neq T(l)$, $T(l) \neq T(p)$ and $T(j) \neq T(p)$, this quadrangle would connect $3$ distinct
        triangles. Such a quadrangle does not exist in $\Gamma$ however, and we conclude $d(p, q) \neq 1$.
        We can now compute
        \[
            u_{ij} u_{kl} u_{ip} = u_{ij} \sum_{\substack{v \in V\\v \sim i\\ d(k, v) = 2\\ Q(k, v)= 2}} u_{vq} u_{kl} u_{ip} 
            = u_{ij} u_{kl} \underbrace{\sum_{\substack{v \in V\\v \sim i\\ d(k, v) = 2\\ Q(k, v)= 2}} u_{vq} u_{ip}}_{\substack{= 0 \text{ since }\\
            d(i, v) = 1 \text{ but }\\ d(q, p) = 2}} = 0
        .\] 
        Here, the sum commutes with $u_{kl}$ by Step~\ref{step:dist2_Q2_2}, as $d(k, s) = 2 = d(q, l)$ and $Q(k, s) = 2 = Q(q, l)$ for all summands holds.
        We conclude 
        \[
            u_{ij} u_{kl} = u_{ij} u_{kl} u_{ij}
        \] 
        and therefore $u_{ij}$ and $u_{kl}$ commute by Lemma~\ref{lem:adj_commute}.

        \item \label{step:dist2_Q0_0} Let now $Q(i, k) = Q(j, l) = 0$.
        We show again that 
        \[
            u_{ij} u_{kl} u_{ip} = 0
        \] 
        for all vertices $p \neq j$. As above, we already get that $u_{kl} u_{ip} = 0$, whenever $d(l, p) \neq d(i, k) = 2$ or when
        $d(l, p) = 2$ but $Q(l, p) \neq Q(i, k) = 0$. We thus now consider $p \in V \backslash\left\{ j \right\}$ such that $d(l, p) = 2$ and $Q(l, p) = 0$.
        We will denote by $s$ the unique neighbour of $i$ and $k$, and by $t$ the unique neighbour of $l$ and $p$.
        We compute
        \[
            u_{ij} u_{kl} u_{ip} = u_{ij} u_{kl} \sum_{v \in V} u_{vt} u_{ip} = u_{ij} u_{kl} u_{st} u_{ip}
        .\] 
        We moreover denote by $t'$ the unique neighbour of $j$ and $l$. Then in particular $t \neq t'$, since $p \neq j$. Similarly to above, we get
        \[
            u_{ij} u_{kl} u_{st} u_{ip} = u_{ij} u_{st'} u_{kl} u_{st} u_{ip}
        .\] 
        Since however $d(s, k) = 1 = d(l, t')$, we know that $u_{s t'}$ and $u_{kl}$ commute, and thus
        \[
                u_{ij} u_{st'} u_{kl} u_{st} u_{ip} = u_{ij} u_{kl} u_{st'} u_{st} u_{ip} = 0
        \] 
        as $t \neq t'$.
        We thus have $u_{ij} u_{kl} u_{ip} = 0$ for $p \neq j$ and thus 
        \[
            u_{ij} u_{kl} = u_{ij} u_{kl} u_{ij}
        \] 
        and by Lemma~\ref{lem:adj_commute}, $u_{ij}$ and $u_{kl}$ commute.
       \end{pfsteps} 
\end{proof}

\subsection{The cuboctahedral graph}
In this section, we show that the cuboctahedral graph, which is the line graph of the cube, has no quantum symmetries. By $\Gamma$ we will denote the cuboctahedral graph throughout the entire section,
even if it is sometimes not explicitly stated.

\begin{thm}
        \label{thm:cuboct}
        Let $\Gamma$ be the cuboctahedral graph as shown below. It has no quantum symmetries, i.e. $\QBan(\Gamma)$ is commutative.
\end{thm}
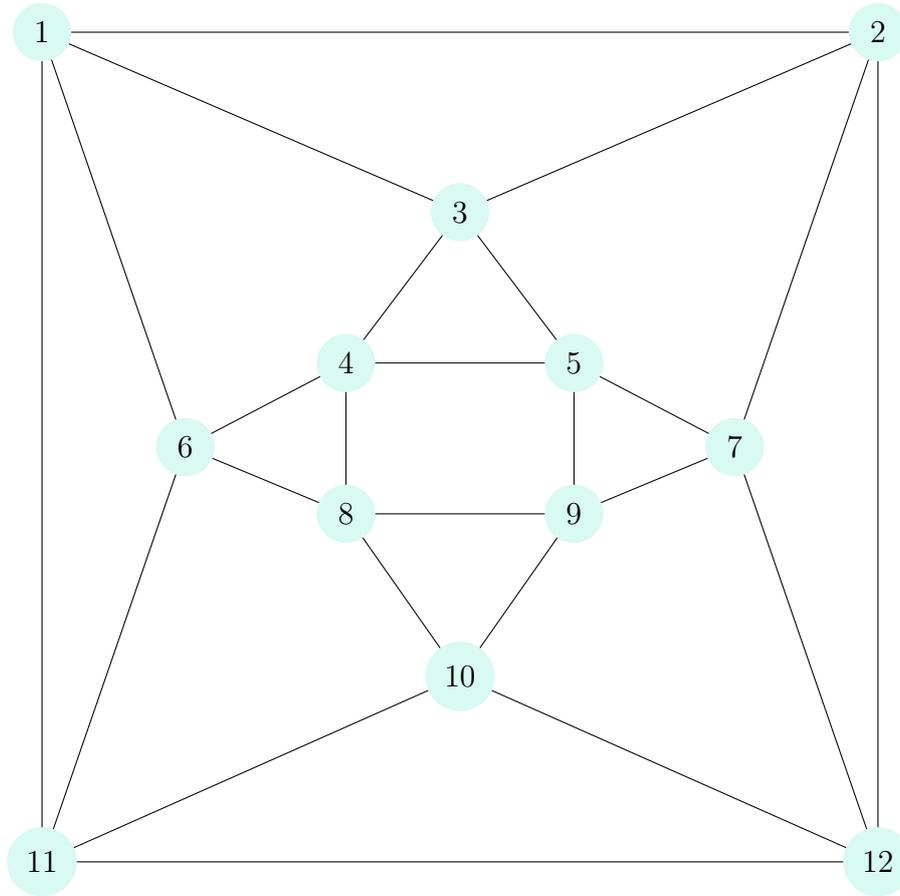
\begin{center}
        \begin{tikzpicture}[auto, node distance = 2cm and 1.5cm, on grid, vertex/.style={circle, draw=xkcdAquamarine!15, fill=xkcdAquamarine!15, very thick, minimum size=5mm}]
            \node[vertex] (A) at (1, 12) {1};
            \node[vertex] (B) at (12, 12)  {2};
            \node[vertex] (K) at (1, 1) {11}; 
            \node[vertex] (L) at (12, 1) {12}; 
            \node[vertex] (C) [below =of $(A)!0.5!(B)$] {3};
            \node[vertex] (D) [below left=of C] {4};
            \node[vertex] (E) [below right=of C] {5};
            \node[vertex] (F) [right =of $(A)!0.5!(K)$] {6};
            \node[vertex] (G) [left =of $(B)!0.5!(L)$] {7};
            \node[vertex] (H) [below =of D] {8};
            \node[vertex] (I) [below =of E] {9};
            \node[vertex] (J) [above =of $(K)!0.5!(L)$] {10};

            \draw (A) -- (B);
            \draw (A) -- (C);
            \draw (A) -- (K);
            \draw (A) -- (F);
            \draw (B) -- (C);
            \draw (B) -- (G);
            \draw (B) -- (L);
            \draw (C) -- (D);
            \draw (C) -- (E);
            \draw (D) -- (F);
            \draw (D) -- (H);
            \draw (D) -- (E);
            \draw (E) -- (G);
            \draw (E) -- (I);
            \draw (F) -- (H);
            \draw (F) -- (K);
            \draw (G) -- (I);
            \draw (G) -- (L); 
            \draw (H) -- (I);
            \draw (H) -- (J);
            \draw (I) -- (J);
            \draw (J) -- (K);
            \draw (J) -- (L); 
            \draw (K) -- (L); 
        \end{tikzpicture}
\end{center}
Before proving the statement, we will again first collect some properties of the graph in question.
\begin{propr}
        \noindent
        \begin{itemize}
                \item $\Gamma$ has diameter $3$, i.e. the largest distance between two vertices is $3$.
                \item $\Gamma$ is $4$-regular, i.e. every vertex has exactly $4$ neighbours.
                \item $\Gamma$ contains $8$ triangles.
                \item Each vertex of $\Gamma$ is contained in exactly two triangles and every pair of triangles overlaps in at most one vertex.
        \end{itemize}
\end{propr}

\begin{lem}
        \label{lem:cuboct:one_common_neighbour}
        Let $\Gamma = (V, E)$ be the cuboctahedral graph. Then all vertices that are adjacent have exactly one common neighbour.
\end{lem}
\begin{proof}
       $\Gamma$ is the line graph of the cube $C = (V' , E')$, i.e. the vertices of $\Gamma$ are the edges of $C$ and two vertices of $\Gamma$ are connected, if they share a vertex as edges of $C$.
       Let now $d, e \in V = E'$ be adjacent vertices in $\Gamma$, i.e. there is a vertex $v \in V'$ of $C$, such that $v \in d$ and $v \in e$. Since the cube is $3$-regular, there is
       exactly one other edge, let us call it $f$, that contains $v$. Then $f$ is a common neighbour of $d$ and $e$ in $\Gamma$.

       It remains to be shown, that $d$ and $e$ have no other common neighbour. 
       For this, let $g \in E' \backslash \left\{ d, e, f \right\}$ be another edge of $C$ and let us assume that $g$ is a common neighbour of $d$ and $e$ in $\Gamma$. It holds that $v \notin g$ since
       $C$ is $3$-regular and the three edges containing $v$ are already $d$, $e$ and $f$. Let us therefore fix $g = (x, y)$ for $x, y \in V' \backslash \left\{ v \right\}$.
       Since we have $(g, d) \in E$ and $(g, e) \in E$ by assumption, there is a vertex in $g$ for both $d$ and $e$ that they share with $g$. Since $(d, e) \in E$ we have $d \neq e$ and thus the 
       vertices they share with $g$ are distinct. Thus, without loss of generality, we have $d = (v, x)$ and $e = (v, y)$. But then $\left\{ x, y, v \right\}$ form a triangle. Since the cube
       does not contain any triangles, we have a contradiction, and thus $g$ can not be a common neighbour of $d$ and $e$.
\end{proof}
\begin{lem}
        \label{lem:cuboct:dist_3_unique}
        For each vertex $v$ of $\Gamma$, there is exactly one vertex $w$ of $\Gamma$, such that $d(v, w) = 3$. We denote the map that maps $v$ to $w$ by $a$.
\end{lem}
\begin{proof}
        We give the proof by giving illustrations of the $6$ pairs of vertices in distance $3$.
\begin{center}
        \begin{tikzpicture}[node distance = 0.875cm, auto, on grid, semithick, vertex/.style={circle, draw=black, fill=black, text=black, inner sep = 0pt, minimum size = 4pt}, pair/.style={circle, draw=red, fill=red, text=black, inner sep = 0pt, minimum size = 4pt}]
            \node[pair] (A) at (1, 5) {};
            \node[vertex] (B) at (5, 5)  {};
            \node[vertex] (K) at (1, 1) {}; 
            \node[vertex] (L) at (5, 1) {}; 
            \node[vertex] (C) [below =of $(A)!0.5!(B)$] {};
            \node[vertex] (D) [below left=of C] {};
            \node[vertex] (E) [below right=of C] {};
            \node[vertex] (F) [right =of $(A)!0.5!(K)$] {};
            \node[vertex] (G) [left =of $(B)!0.5!(L)$] {};
            \node[vertex] (H) [below =of D] {};
            \node[pair] (I) [below =of E] {};
            \node[vertex] (J) [above =of $(K)!0.5!(L)$] {};

            \draw (A) -- (B);
            \draw (A) -- (C);
            \draw (A) -- (K);
            \draw (A) -- (F);
            \draw (B) -- (C);
            \draw (B) -- (G);
            \draw (B) -- (L);
            \draw (C) -- (D);
            \draw (C) -- (E);
            \draw (D) -- (F);
            \draw (D) -- (H);
            \draw (D) -- (E);
            \draw (E) -- (G);
            \draw (E) -- (I);
            \draw (F) -- (H);
            \draw (F) -- (K);
            \draw (G) -- (I);
            \draw (G) -- (L); 
            \draw (H) -- (I);
            \draw (H) -- (J);
            \draw (I) -- (J);
            \draw (J) -- (K);
            \draw (J) -- (L); 
            \draw (K) -- (L); 
        \end{tikzpicture}
        \begin{tikzpicture}[node distance = 0.875cm, auto, on grid, semithick, vertex/.style={circle, draw=black, fill=black, text=black, inner sep = 0pt, minimum size = 4pt}, pair/.style={circle, draw=red, fill=red, text=black, inner sep = 0pt, minimum size = 4pt}]
            \node[vertex] (A) at (1, 5) {};
            \node[pair] (B) at (5, 5)  {};
            \node[vertex] (K) at (1, 1) {}; 
            \node[vertex] (L) at (5, 1) {}; 
            \node[vertex] (C) [below =of $(A)!0.5!(B)$] {};
            \node[vertex] (D) [below left=of C] {};
            \node[vertex] (E) [below right=of C] {};
            \node[vertex] (F) [right =of $(A)!0.5!(K)$] {};
            \node[vertex] (G) [left =of $(B)!0.5!(L)$] {};
            \node[pair] (H) [below =of D] {};
            \node[vertex] (I) [below =of E] {};
            \node[vertex] (J) [above =of $(K)!0.5!(L)$] {};

            \draw (A) -- (B);
            \draw (A) -- (C);
            \draw (A) -- (K);
            \draw (A) -- (F);
            \draw (B) -- (C);
            \draw (B) -- (G);
            \draw (B) -- (L);
            \draw (C) -- (D);
            \draw (C) -- (E);
            \draw (D) -- (F);
            \draw (D) -- (H);
            \draw (D) -- (E);
            \draw (E) -- (G);
            \draw (E) -- (I);
            \draw (F) -- (H);
            \draw (F) -- (K);
            \draw (G) -- (I);
            \draw (G) -- (L); 
            \draw (H) -- (I);
            \draw (H) -- (J);
            \draw (I) -- (J);
            \draw (J) -- (K);
            \draw (J) -- (L); 
            \draw (K) -- (L); 
        \end{tikzpicture}
        \begin{tikzpicture}[node distance = 0.875cm, auto, on grid, semithick, vertex/.style={circle, draw=black, fill=black, text=black, inner sep = 0pt, minimum size = 4pt}, pair/.style={circle, draw=red, fill=red, text=black, inner sep = 0pt, minimum size = 4pt}]
            \node[vertex] (A) at (1, 5) {};
            \node[vertex] (B) at (5, 5)  {};
            \node[pair] (K) at (1, 1) {}; 
            \node[vertex] (L) at (5, 1) {}; 
            \node[vertex] (C) [below =of $(A)!0.5!(B)$] {};
            \node[vertex] (D) [below left=of C] {};
            \node[pair] (E) [below right=of C] {};
            \node[vertex] (F) [right =of $(A)!0.5!(K)$] {};
            \node[vertex] (G) [left =of $(B)!0.5!(L)$] {};
            \node[vertex] (H) [below =of D] {};
            \node[vertex] (I) [below =of E] {};
            \node[vertex] (J) [above =of $(K)!0.5!(L)$] {};

            \draw (A) -- (B);
            \draw (A) -- (C);
            \draw (A) -- (K);
            \draw (A) -- (F);
            \draw (B) -- (C);
            \draw (B) -- (G);
            \draw (B) -- (L);
            \draw (C) -- (D);
            \draw (C) -- (E);
            \draw (D) -- (F);
            \draw (D) -- (H);
            \draw (D) -- (E);
            \draw (E) -- (G);
            \draw (E) -- (I);
            \draw (F) -- (H);
            \draw (F) -- (K);
            \draw (G) -- (I);
            \draw (G) -- (L); 
            \draw (H) -- (I);
            \draw (H) -- (J);
            \draw (I) -- (J);
            \draw (J) -- (K);
            \draw (J) -- (L); 
            \draw (K) -- (L); 
        \end{tikzpicture}
        \\
        \begin{tikzpicture}[node distance = 0.875cm, auto, on grid, semithick, vertex/.style={circle, draw=black, fill=black, text=black, inner sep = 0pt, minimum size = 4pt}, pair/.style={circle, draw=red, fill=red, text=black, inner sep = 0pt, minimum size = 4pt}]
            \node[vertex] (A) at (1, 5) {};
            \node[vertex] (B) at (5, 5)  {};
            \node[vertex] (K) at (1, 1) {}; 
            \node[pair] (L) at (5, 1) {}; 
            \node[vertex] (C) [below =of $(A)!0.5!(B)$] {};
            \node[pair] (D) [below left=of C] {};
            \node[vertex] (E) [below right=of C] {};
            \node[vertex] (F) [right =of $(A)!0.5!(K)$] {};
            \node[vertex] (G) [left =of $(B)!0.5!(L)$] {};
            \node[vertex] (H) [below =of D] {};
            \node[vertex] (I) [below =of E] {};
            \node[vertex] (J) [above =of $(K)!0.5!(L)$] {};

            \draw (A) -- (B);
            \draw (A) -- (C);
            \draw (A) -- (K);
            \draw (A) -- (F);
            \draw (B) -- (C);
            \draw (B) -- (G);
            \draw (B) -- (L);
            \draw (C) -- (D);
            \draw (C) -- (E);
            \draw (D) -- (F);
            \draw (D) -- (H);
            \draw (D) -- (E);
            \draw (E) -- (G);
            \draw (E) -- (I);
            \draw (F) -- (H);
            \draw (F) -- (K);
            \draw (G) -- (I);
            \draw (G) -- (L); 
            \draw (H) -- (I);
            \draw (H) -- (J);
            \draw (I) -- (J);
            \draw (J) -- (K);
            \draw (J) -- (L); 
            \draw (K) -- (L); 
        \end{tikzpicture}
        \begin{tikzpicture}[node distance = 0.875cm, auto, on grid, semithick, vertex/.style={circle, draw=black, fill=black, text=black, inner sep = 0pt, minimum size = 4pt}, pair/.style={circle, draw=red, fill=red, text=black, inner sep = 0pt, minimum size = 4pt}]
            \node[vertex] (A) at (1, 5) {};
            \node[vertex] (B) at (5, 5)  {};
            \node[vertex] (K) at (1, 1) {}; 
            \node[vertex] (L) at (5, 1) {}; 
            \node[pair] (C) [below =of $(A)!0.5!(B)$] {};
            \node[vertex] (D) [below left=of C] {};
            \node[vertex] (E) [below right=of C] {};
            \node[vertex] (F) [right =of $(A)!0.5!(K)$] {};
            \node[vertex] (G) [left =of $(B)!0.5!(L)$] {};
            \node[vertex] (H) [below =of D] {};
            \node[vertex] (I) [below =of E] {};
            \node[pair] (J) [above =of $(K)!0.5!(L)$] {};

            \draw (A) -- (B);
            \draw (A) -- (C);
            \draw (A) -- (K);
            \draw (A) -- (F);
            \draw (B) -- (C);
            \draw (B) -- (G);
            \draw (B) -- (L);
            \draw (C) -- (D);
            \draw (C) -- (E);
            \draw (D) -- (F);
            \draw (D) -- (H);
            \draw (D) -- (E);
            \draw (E) -- (G);
            \draw (E) -- (I);
            \draw (F) -- (H);
            \draw (F) -- (K);
            \draw (G) -- (I);
            \draw (G) -- (L); 
            \draw (H) -- (I);
            \draw (H) -- (J);
            \draw (I) -- (J);
            \draw (J) -- (K);
            \draw (J) -- (L); 
            \draw (K) -- (L); 
        \end{tikzpicture}
        \begin{tikzpicture}[node distance = 0.875cm, auto, on grid, semithick, vertex/.style={circle, draw=black, fill=black, text=black, inner sep = 0pt, minimum size = 4pt}, pair/.style={circle, draw=red, fill=red, text=black, inner sep = 0pt, minimum size = 4pt}]
            \node[vertex] (A) at (1, 5) {};
            \node[vertex] (B) at (5, 5)  {};
            \node[vertex] (K) at (1, 1) {}; 
            \node[vertex] (L) at (5, 1) {}; 
            \node[vertex] (C) [below =of $(A)!0.5!(B)$] {};
            \node[vertex] (D) [below left=of C] {};
            \node[vertex] (E) [below right=of C] {};
            \node[pair] (F) [right =of $(A)!0.5!(K)$] {};
            \node[pair] (G) [left =of $(B)!0.5!(L)$] {};
            \node[vertex] (H) [below =of D] {};
            \node[vertex] (I) [below =of E] {};
            \node[vertex] (J) [above =of $(K)!0.5!(L)$] {};

            \draw (A) -- (B);
            \draw (A) -- (C);
            \draw (A) -- (K);
            \draw (A) -- (F);
            \draw (B) -- (C);
            \draw (B) -- (G);
            \draw (B) -- (L);
            \draw (C) -- (D);
            \draw (C) -- (E);
            \draw (D) -- (F);
            \draw (D) -- (H);
            \draw (D) -- (E);
            \draw (E) -- (G);
            \draw (E) -- (I);
            \draw (F) -- (H);
            \draw (F) -- (K);
            \draw (G) -- (I);
            \draw (G) -- (L); 
            \draw (H) -- (I);
            \draw (H) -- (J);
            \draw (I) -- (J);
            \draw (J) -- (K);
            \draw (J) -- (L); 
            \draw (K) -- (L); 
        \end{tikzpicture}
\end{center}
\end{proof}

\begin{lem}
        \label{lem:cuboct:dist2}
        Let $v, w \in V$ be two vertices of $\Gamma$. If $d(v, w) = 2$ then there exists exactly one vertex $x \in V$ that shares a triangle with both $v$ and $w$.
\end{lem}
\begin{proof}
        Consider a shortest path from $v$ to $w$ and name the vertex that is passed by $x$.
        Since all edges that contain vertex $v$ only lead to vertices that share a triangle with $v$, $x$ shares a triangle with $v$. By the same argument, $x$ and $w$ share a triangle.
        The uniqueness of $x$ is due to the fact that all triangles only overlap in exactly one vertex and that each vertex lies in exactly two triangles.
\end{proof}
\begin{lem}
        \label{lem:cuboct:dist2_dist3}
        Let $v \in V$ be a vertex in $\Gamma$. There are exactly two vertices that have distance $2$ to both $v$ and $a(v)$. Moreover, if $w$ is one of these vertices, the other one is $a(w)$.
\end{lem}
\begin{proof}
       We denote by $N(v)$ the neighbourhood of $v$, i.e. all vertices that are adjacent to $v$. 
       Note, that $\#(N(v) \cup N(a(v))) = 8$, i.e. $8$ vertices are either adjacent to $v$ or to $a(v)$. This is due to the fact that $\Gamma$ has degree $4$ and thus all vertices have $4$ neighbours and moreover
       $N(v) \cap N(a(v)) = \left\{  \right\}$. This holds, since if there was a vertex adjacent to both $v$ and $a(v)$, this would yield a path from $v$ to $a(v)$ of length $2$, which is a contradiction to $d(v, a(v)) = 3$.
       We thus have 
       \[
        V \backslash (N(v) \cup N(a(v)) \cup\left\{ v, a(v) \right\}) = \left\{ w, x \right\}
       \] 
       for some vertices $w, x$. One can check by hand, that $d(w, x) = 3$ for all choices of $v$.
\end{proof}

We now come to the proof of Theorem~\ref{thm:cuboct}.
\begin{proof}{of Theorem~\ref{thm:cuboct}}
        Let $(u_{ij})_{1 \leq i, j \leq n}$ be the generators of $C(\QBan(\Gamma))$.
        We consider again vertices $i, j, k, l$ in increasing order of $m = d(i, k) = d(j, l)$. Since $\Gamma$ has diameter $3$, $m$ is maximally $3$.

        \textbf{\emph{Step 1:}} $m = 1$. Since we know by Lemma~\ref{lem:cuboct:one_common_neighbour} that in $\Gamma$ all adjacent vertices have exactly one common neighbour, we can 
        conclude by Lemma~\ref{lem:one_common_neighbour} that $\QBan(\Gamma) = \QBic(\Gamma)$ and therefore $u_{ij} u_{kl} = u_{kl} u_{ij}$.

        \textbf{\emph{Step 2:}} $m = 2$. We consider two cases:

         \emph{Step 2.1:} $d(a(j), l) = 1$. 
         If the unique vertex at distance $3$ from $j$ is adjacent to $l$, we can apply Lemma~\ref{lem:choose_q_right} by putting $q:= a(j)$ and get
         \[
         u_{ij} u_{kl} = u_{ij} u_{kl} \sum_{\substack{p; d(l, p) = 2\\d(a(j), p) = 3}} u_{ip} = u_{ij} u_{kl} u_{ij}
         \] 
         since $j$ is the only vertex at distance $3$ to $a(j)$. We can apply the lemma, since we have $\QBan(\Gamma) = \QBic(\Gamma)$ by \emph{Step 1}.
         We thus can apply Lemma~\ref{lem:adj_commute} to see that $u_{ij}$ and $u_{kl}$ commute.

         \emph{Step 2.2:} $d(a(j), l) = 2$. 
         We apply again Lemma~\ref{lem:choose_q_right} and this time we put $q$ as the unique vertex that shares a triangle with both $j$ and $l$, which exists by Lemma~\ref{lem:cuboct:dist2}.
         Let us give a sketch of the situation.
         \begin{center}
            \tiny
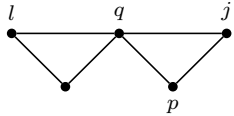
                \begin{tikzpicture}[node distance = 1cm, auto, on grid, semithick, state/.style={circle, draw=black, fill=black, text=black, inner sep = 0pt, minimum size = 3pt}]
                    \node[state,label=above: $l$] (1) {};
                    \node[state] (2) [below right=of 1] {};
                    \node[state, label=above: $q$] (3) [above right=of 2] {};
                    \node[state, label=below: $p$] (4) [below right=of 3] {};
                    \node[state, label=above: $j$] (5) [above right=of 4] {};
                    \draw (1) -- (2);
                    \draw (1) -- (3);
                    \draw (3) -- (2);
                    \draw (3) -- (4);
                    \draw (4) -- (5);
                    \draw (5) -- (3);
                \end{tikzpicture}
            \end{center} 

         This yields
         \[
         u_{ij} u_{kl} = u_{ij} u_{kl} \sum_{\substack{p; d(p, l) = 2\\ (p, q) \in E}} u_{ip}
         .\] 
         The vertex $q$ has $4$ neighbours, two of which are in the same triangle as $l$ and thus have distance smaller that $2$ to $l$. The remaining vertices are $j$ and $p$ and we thus get
         \[
         u_{ij} u_{kl} = u_{ij} u_{kl} (u_{ij} +  u_{ip})
         .\] 
         We now want to show, that $d(a(p), l) = 1$. For this, first note, that by Lemma~\ref{lem:cuboct:dist2_dist3} there are exactly $2$ vertices that have distance $2$ to both $j$ and $a(j)$ and
         since $l$ is one of those, the other one is $a(l)$. Moreover, this yields that the only vertices at distance $2$ to both $l$ and $a(l)$ are in turn $j$ and $a(j)$. Therefore we get
         \begin{align}
                 \label{eqn_1}
                 x \notin \left\{ j, a(j), l, a(l) \right\} \text{ it holds that } d(x, l) = 1\text{ or }d(x, a(l)) = 1.
         \end{align}
         Since we have that $d(p, j) = 1$, and thus $p \neq j$, and $d(p, l) = 2$ we thus get $d(p, a(l)) = 1$. But then we get $d(a(p), a(l)) > 1$, since otherwise there would be a path from $p$ to $a(p)$
         of length $2$ via $a(l)$, which can not be. But then we get, again applying~\ref{eqn_1}, that $d(a(p), l) = 1$.
         Now, we can use our results from \emph{Step 2.1} to see that $u_{kl}$ and $u_{ip}$ commute. This yields
         \[
         u_{ij} u_{kl} u_{ip} = u_{ij} u_{ip} u_{kl} = 0
         \] 
         since $j \neq p$. Therefore 
         \[
         u_{ij} u_{kl} = u_{ij} u_{kl} u_{ij}
         \] 
         and again by Lemma~\ref{lem:adj_commute}, $u_{ij}$ and $u_{kl}$ commute.
         
         \textbf{\emph{Step 3:}} $m = 3$. If $d(i, k) = d(j, l) = 3$ holds, we get with Lemma~\ref{lem:distances} that
         \[
                 u_{ij} u_{kl} = u_{ij} u_{kl} \sum_{p \in V} u_{ip} = u_{ij} u_{kl} \sum_{p; d(l, p) = 3} u_{ip}
         .\] 
         But since the only vertex at distance $3$ from $l$ is $j$ this yields
         \[
                 u_{ij} u_{kl} = u_{ij} u_{kl} u_{ij}
         \] 
         and by Lemma~\ref{lem:adj_commute}, $u_{ij}$ and $u_{kl}$ commute.
\end{proof}
 
\counterwithin{thm}{section}
\section{Computation of quantum automorphism groups}
\label{section:computation_of_quantum_automorphism_groups}
\label{section:computation_qaut_groups}
In this section, we give the computation of the quantum automorphism groups of $C_{12}(4, 5)$ and $C_{12}(3^+, 6)$.
For these computations, we used in some places a noncommutative Groebner basis implementation in the 
\texttt{OSCAR}~\cite{oscar} package in the \texttt{julia} programming language~\cite{julia}.

\begin{prop}
    The quantum automorphism group of $C_{12}(4, 5)$ is $H_2^+ \times S_3$.
\end{prop}
\begin{proof}
    For the purpose of this proof, we will call $C_{12}(4, 5) =: \Gamma$.
    We first give an illustration of the graph that highlights the action of the automorphism group on its vertices, which will help
    in constructing the isomorphism between $G_{aut}^+(\Gamma)$ and $H_2^+ \times S_3$.
    We also label the vertices of this graph differently than before to make it easier to talk about the construction of the 
    $*$-isomorphism between the two quantum groups. Note that the colors in this illustration are only to make it easier to 
    see the action of the automorphism group on $\Gamma$ and are not an actual graph coloring.
    
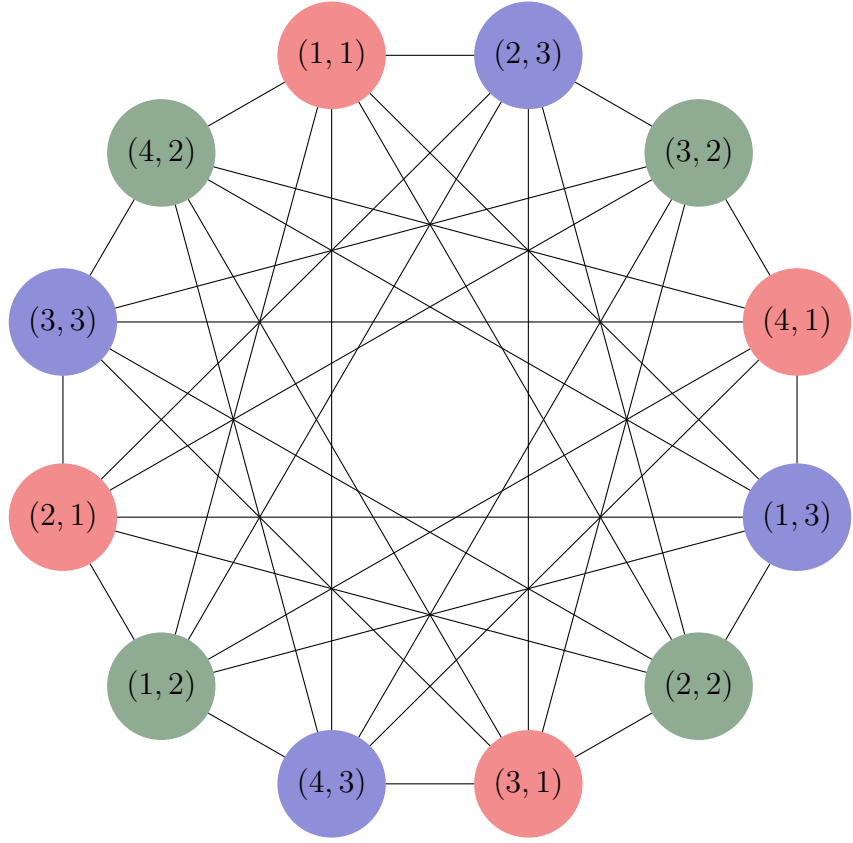
\begin{center}
        \begin{tikzpicture}[auto, node distance = 2cm and 2.5cm, on grid, 
            vertexRed/.style={circle, draw=xkcdRed!45, fill=xkcdRed!45, very thick, minimum size=9mm},
            vertexGreen/.style={circle, draw=xkcdForestGreen!45, fill=xkcdForestGreen!45, very thick, minimum size=9mm},
            vertexBlue/.style={circle, draw=xkcdRoyalBlue!45, fill=xkcdRoyalBlue!45, very thick, minimum size=9mm}
            ]
            \node[vertexRed] (1) at (75 + 1*30:5cm) {$(1, 1)$};
            \node[vertexGreen] (2) at (75 + 2*30:5cm) {$(4, 2)$};
            \node[vertexBlue] (3) at (75 + 3*30:5cm) {$(3, 3)$};
            \node[vertexRed] (4) at (75 + 4*30:5cm) {$(2, 1)$};
            \node[vertexGreen] (5) at (75 + 5*30:5cm) {$(1, 2)$};
            \node[vertexBlue] (6) at (75 + 6*30:5cm) {$(4, 3)$};
            \node[vertexRed] (7) at (75 + 7*30:5cm) {$(3, 1)$};
            \node[vertexGreen] (8) at (75 + 8*30:5cm) {$(2, 2)$};
            \node[vertexBlue] (9) at (75 + 9*30:5cm) {$(1, 3)$};
            \node[vertexRed] (10) at (75 + 10*30:5cm) {$(4, 1)$};
            \node[vertexGreen] (11) at (75 + 11*30:5cm) {$(3, 2)$};
            \node[vertexBlue] (12) at (75 + 12*30:5cm) {$(2, 3)$};

            \foreach[evaluate ={
                \l = int(\j + 1);
                \a = int(\j + 5);
                \b = int(\j - 7);
                \c = int(\j + 4);
                \d = int(\j - 8);
            }] \j in {1,...,12}{
            \ifthenelse {\j  < 12}{\draw (\j) -- (\l);}{\draw (\j) -- (1);};
            \ifthenelse {\j  < 8}{\draw (\j) -- (\a);}{\draw (\j) -- (\b)};
            \ifthenelse {\j  < 9 }{\draw (\j) -- (\c);}{\draw (\j) -- (\d);};
        }
        \end{tikzpicture}
    \end{center}

    The universal $C^*-$algebra defining $H_2^+ \times S_3$ is as follows:
    \begin{align*}
        C(H_2^+ \times S_3) = C^*(& v_{i j}, w_{kl}, 1 \le i, j \le 4, 1 \le k,l \le 3\\
                                  & \sum_{a = 1}^4 v_{i a} = \sum_{a=1}^4 v_{aj} = 1 \; \forall 1 \le i, j \le 4 \\
                                  & \sum_{b = 1}^3 w_{k b} = \sum_{b=1}^3 w_{bl} = 1 \; \forall 1 \le k, l \le 3 \\
                                  & v_{ij}^2 = v_{ij} = v_{ij}^*\\
                                  & w_{ij}^2  = w_{ij} = w_{ij}^*\\
    & A_{C_4} v = v A_{C_4}\\
                                  & v_{ij} w_{kl} = w_{kl} v_{ij}
        ) 
    .\end{align*}
    Here, $A_{C_4}$ is the adjacency matrix of the $4$-cycle  $C_4$ and $v$ is the matrix with entries $v_{ij}$ for $1 \le i, j \le 4$.
    We use the fact that $G_{aut}^+(C_4) = H_2^+$, which was shown in~\cite{schmidt_weber_2018}.

    We now claim:
    \begin{enumerate}
        \item \label{c12_4_5:claim1} The matrix $\hat{u}$ with entries from $C(H_2^+ \times S_3)$ defined by
            \[
                \hat{u}_{(i_1, i_2)(j_1, j_2)} = v_{i_1 j_1} w_{i_2 j_2}
            \] 
            satisfies the relations of $C(G_{aut}^+(\Gamma))$.
        \item \label{c12_4_5:claim2} The elements $\hat{v}_{ij}$ and $\hat{w}_{kl}$ in $C(G_{aut}^+(\Gamma))$ defined by
            \begin{align*}
                \hat{v}_{ij} &= u_{(i, 1)(j, 1)} + u_{(i, 1)(j, 2)} + u_{(i, 1)(j, 3)}\\
                \hat{w}_{kl} &= u_{(1, k)(1, l)} + u_{(1, k)(2, l)} + u_{(1, k)(3, l)} + u_{(1, k)(4, l)}
            \end{align*}
            satisfy the relations of $C(H_2^+ \times S_3)$.
    \end{enumerate}
    Let us first prove \ref{c12_4_5:claim1}:

    For this, observe the following: We have
    \begin{align*}
        &(i_1, i_2) \sim (k_1, k_2)\text{ in } \Gamma \\
        \Longleftrightarrow \;& i_2 \neq k_2 \land (i_1 = k_1 \lor i_1 \sim_{C_4} k_1)
    \end{align*}
    Let now $(i_1, i_2)\sim (k_1, k_2)$ and $(j_1, j_2) \not \sim (l_1, l_2)$, i.e. for $(i_1, i_2)$ and $(k_1, k_2)$ the above 
    observation holds, while for $(j_1, j_2)$ and $(l_1, l_2)$ we have:
    \begin{align*}
        j_2 = l_2 \lor (j_1 \neq l_1 \land j_1 \not \sim_{C_4} l_1)
    .\end{align*}
    In order to show that indeed $\hat{u}_{(i_1,i_2)(j_1, j_2)} \hat{u}_{(k_1, k_2)(l_1, l_2)} = 0$, we make the following case distinction:
    \begin{itemize}
        \item Case $j_2 = l_2$:\\ Since we already know $i_2 \neq k_2$, we get $w_{i_2 j_2} w_{k_2 j_2} = 0$ and thus\\
            $\hat{u}_{(i_1,i_2)(j_1, j_2)} \hat{u}_{(k_1, k_2)(l_1, l_2)} = v_{i_1 j_1} w_{i_2 j_2} v_{k_1 l_1} w_{k_2 j_2}
            =v_{i_1 j_1} w_{i_2 j_2}w_{k_2 j_2} v_{k_1 l_1} = 0$
        \item Case $j_2 \neq l_2$:
            \begin{itemize}
                \item Case $i_1 = k_1$:\\ We get $v_{i_1 j_1} v_{i_1 l_1} = 0$ and thus 
                    $\hat{u}_{(i_1,i_2)(j_1, j_2)} \hat{u}_{(k_1, k_2)(l_1, l_2)} = 0$
                    \item Case $i_1 \sim_{C_4} k_1$ :\\ Since $j_1 \not \sim_{C_4} l_1$, we get $v_{i_1 j_1} v_{k_1 l_1} = 0$ 
                    and thus $\hat{u}_{(i_1,i_2)(j_1, j_2)} \hat{u}_{(k_1, k_2)(l_1, l_2)} = 0$
            \end{itemize}
    \end{itemize}
    It is moreover easy to see, that the relations of $H_2^+ \times S_3$ yield that all rows and columns of $\hat{u}$ 
    sum up to $1$ and that the entries of  $\hat{u}$ are projections, as they are just products of commuting projections.
    From the universal property of $C(G_{aut}^+(\Gamma))$ we thus get a $*$-homomorphism
     \begin{align*}
         \phi: C(G_{aut}^+(\Gamma)) &\to C(H_2^+ \times S_3)\\
         u_{(i_1, i_2)(j_1, j_2)} &\mapsto \hat{u}_{(i_1, i_2)(j_1, j_2)} = v_{i_1 j_1}w_{i_2j_2}
    .\end{align*} 
    It is also easy to see that all the generators of $C(H_2^+ \times S_3)$ are in the image of $\phi$:
    for example, $v_{i_1j_1}$ is the image of 
    \[
        u_{(i_1, 1)(j_1, 1)} + u_{(i_1, 1)(j_1, 2)} + u_{(i_1, 1)(j_1, 3)}
    \]
    since
    \[
        w_{11}+ w_{12} + w_{13} = 1
    ,\] 
    and a similar argument holds for $w_{i_2 j_2}$.
    Therefore, $\phi$ is surjective.

    We will next prove \ref{c12_4_5:claim2}:

    We first note that using the Groebner basis for the ideal generated by the relations of $C(G_{aut}^+(\Gamma))$, we see that
    \begin{align}
        \label{c12_4_5:gb_result}
        u_{(i, a)(k,b)} = u_{(i', a)(k', b)} \text{ for  }i \not\sim_{C_4} i' \text{ and } k \not\sim_{C_4} k'
    .\end{align} 
    Let $i \not \sim_{C_4} k$ and $j \sim_{C_4} l$ be given. We then have $(i, 1) \not \sim (k, 1)$ but
    $(j, a) \sim (l, b)$ whenever $a \neq b$. Therefore, $u_{(i, 1)(j, a)} u_{(k, 1)(l, b)} = 0$ whenever $a \neq b$ and thus
    in the product
    \begin{align*}
        \hat{v}_{ij}\hat{v}_{kl} = (u_{(i, 1)(j, 1)} + u_{(i, 1)(j, 2)} + u_{(i, 1)(j, 3)})(u_{(k, 1)(l, 1)} + u_{(k, 1)(l, 2)} + u_{(k, 1)(l, 3)})
    \end{align*}
    we only keep those products, where $a = b$ and are left with
    \begin{align*}
        \hat{v}_{ij}\hat{v}_{kl} = u_{(i, 1)(j, 1)}u_{(k, 1)(l, 1)} + u_{(i, 1)(j, 2)} u_{(k, 1)(l, 2)} + u_{(i, 1)(j, 3)} u_{(k, 1)(l, 3)}
    .\end{align*}
    By \ref{c12_4_5:gb_result}, we know that $u_{(k, 1)(l, a)} = u_{(i, 1)(l', a)}$, where $l' \not \sim_{C_4} l$. Since $j \sim_{C_4} l$, we
    know in particular that $l' \neq j$ and get $u_{(i, 1)(j, a)} u_{(i, 1)(l', a)} = 0$ and thus
    \[
        \hat{v}_{ij}\hat{v}_{kl} = 0
    .\] 
    Let now $i \sim_{C_4} k$ and $j \not \sim_{C_4} l$ be given. Using a similar trick to above, we see that 
    $u_{(i, 1)(j, a)} u_{(k, 1)(l, a)} = 0$, only this time using the fact that $u_{(k, 1)(l, a)} = u_{(k', 1)(j, a)}$, and get 
    that all summands of that form in $\hat{v}_{ij}\hat{v}_{kl}$ disappear to get
    \begin{align*}
        \hat{v}_{ij}\hat{v}_{kl} = &u_{(i, 1)(j, 1)}u_{(k, 1)(l, 2)} + u_{(i, 1)(j, 1)} u_{(k, 1)(l, 3)}\\
                                + &u_{(i, 1)(j, 2)}u_{(k, 1)(l, 1)} + u_{(i, 1)(j, 2)} u_{(k, 1)(l, 3)}\\
                                +& u_{(i, 1)(j, 3)}u_{(k, 1)(l, 1)} + u_{(i, 1)(j, 2)} u_{(k, 1)(l, 2)}
    .\end{align*}
    Using again \ref{c12_4_5:gb_result} to see that $u_{(k, 1)(l, a)} = u_{(k', 1)(j, a)}$, we get
    \begin{align*}
        \hat{v}_{ij}\hat{v}_{kl} = &u_{(i, 1)(j, 1)}u_{(k', 1)(j, 2)} + u_{(i, 1)(j, 1)} u_{(k', 1)(j, 3)}\\
                                 + &u_{(i, 1)(j, 2)}u_{(k', 1)(j, 1)} + u_{(i, 1)(j, 2)} u_{(k', 1)(j, 3)}\\
                                 +& u_{(i, 1)(j, 3)}u_{(k', 1)(j, 1)} + u_{(i, 1)(j, 2)} u_{(k', 1)(j, 2)}
    .\end{align*}
    However, we know, that still $(i, 1) \not \sim (k', 1)$ holds but we also have $(j, a) \sim (j, b)$ for  $a \neq b$.
    Therefore, we have $u_{(i, 1)(j, a)} u_{(k', 1)(j, b)} = 0$ and get
    \[
        \hat{v}_{ij} \hat{v}_{kl} = 0
    .\] 
    It is easy to see, that the $\hat{v}_{(ij)}$ and the $\hat{w}_{kl}$ are self-adjoint, as they are just sums of self-adjoint elements.
    Moreover, since $u_{(i_1, i_2)(j_1, j_2)} u_{(i_1, i_2)(l_1, l_2)} = 0$ for $(j_1, j_2) \neq (l_1, l_2)$, seeing 
    $\hat{v}_{ij}^2 = \hat{v}_{ij}$ and $\hat{w}_{kl}^2 = \hat{w}_{kl}$ is also straightforward.

    It just remains to be shown that $\sum_{a} \hat{v}_{ia} = \sum_{a} \hat{v}_{aj} = 1 = \sum_{b} \hat{w}_{kb} = \sum_{b} \hat{w}_{bl}$.
    This is easy to see for $\sum_{a} \hat{v}_{ia}$ and $\sum_{b} \hat{w}_{bl}$, since these sums are just sums over one row of
    $u$.
    For $\sum_a \hat{aj} = 1$ and $\sum_b \hat{w}_{kb}$, the Groebner basis yields that the sums are equal to $1$.
    All in all, we get again by the universal property of $C(H_2^+ \times S_3)$ a $*$-homomorphism
     \begin{align*}
         \phi' : C(H_2^+ \times S_3) &\to C(G_{aut}^+(\Gamma))\\
         v_{ij} &\mapsto \hat{v}_{ij}\\
         w_{ij} & \mapsto \hat{w}_{ij}
    .\end{align*}
    By \ref{c12_4_5:gb_result}, we see that all elements of $u$ are hit by  $\phi'$ and in particular, we have that  $\phi'$ is 
    the inverse of  $\phi$. We thus get, that  $C(H_2^+ \times S_3)$ and $C(G_{aut}^+(\Gamma))$ are $*$-isomorphic.
\end{proof}

\begin{prop}
    The quantum automorphism group of $C_{12}(3^+, 6)$ is $H_2^+ \times S_3$.
\end{prop}
\begin{proof}
    For this proof, we will write $C_{12}(3^+, 6) =: \Gamma$.
    We will give a picture of the graph in question with vertices colored in such a way, that the action of the classical
    automorphism group $H_2 \times S_3$ is visible:

    
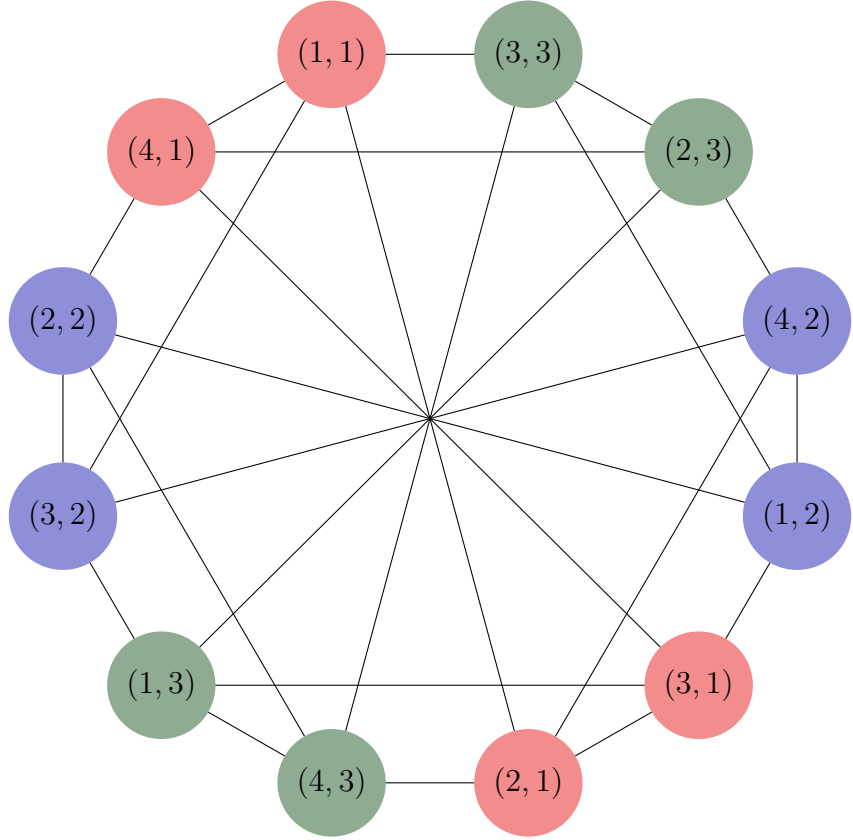
\begin{center}
        \begin{tikzpicture}[auto, node distance = 2cm and 2.5cm, on grid, 
            vertexRed/.style={circle, draw=xkcdRed!45, fill=xkcdRed!45, very thick, minimum size=9mm},
            vertexGreen/.style={circle, draw=xkcdForestGreen!45, fill=xkcdForestGreen!45, very thick, minimum size=9mm},
            vertexBlue/.style={circle, draw=xkcdRoyalBlue!45, fill=xkcdRoyalBlue!45, very thick, minimum size=9mm}
            ]
            \node[vertexRed] (1) at (75 + 1*30:5cm) {$(1, 1)$};
            \node[vertexRed] (2) at (75 + 2*30:5cm) {$(4, 1)$};
            \node[vertexBlue] (3) at (75 + 3*30:5cm) {$(2, 2)$};
            \node[vertexBlue] (4) at (75 + 4*30:5cm) {$(3, 2)$};
            \node[vertexGreen] (5) at (75 + 5*30:5cm) {$(1, 3)$};
            \node[vertexGreen] (6) at (75 + 6*30:5cm) {$(4, 3)$};
            \node[vertexRed] (7) at (75 + 7*30:5cm) {$(2, 1)$};
            \node[vertexRed] (8) at (75 + 8*30:5cm) {$(3, 1)$};
            \node[vertexBlue] (9) at (75 + 9*30:5cm) {$(1, 2)$};
            \node[vertexBlue] (10) at (75 + 10*30:5cm) {$(4, 2)$};
            \node[vertexGreen] (11) at (75 + 11*30:5cm) {$(2, 3)$};
            \node[vertexGreen] (12) at (75 + 12*30:5cm) {$(3, 3)$};

            \foreach[evaluate ={
                \l = int(\j + 1);
                \a = int(\j + 6);
                \c = int(\j + 3);
                \d = int(\j - 9);
            }] \j in {1,...,12}{
            \ifthenelse {\j  < 12}{\draw (\j) -- (\l);}{\draw (\j) -- (1);};
            \ifthenelse {\j  < 7}{\draw (\j) -- (\a);}{};
            \ifthenelse {\j  < 10 }{\ifthenelse{\isodd{\j}}{\draw (\j) -- (\c);}{}}{\ifthenelse{\isodd{\j}}{\draw (\j) -- (\d);}{}};
        }
        \end{tikzpicture}
    \end{center}

    Recall from above that he universal $C^*-$algebra defining $H_2^+ \times S_3$ is as follows:
    \begin{align*}
        C(H_2^+ \times S_3) = C^*(& v_{i j}, w_{kl}, 1 \le i, j \le 4, 1 \le k,l \le 3\\
                                  & \sum_{a = 1}^4 v_{i a} = \sum_{a=1}^4 v_{aj} = 1 \; \forall 1 \le i, j \le 4 \\
                                  & \sum_{b = 1}^3 w_{k b} = \sum_{b=1}^3 w_{bl} = 1 \; \forall 1 \le k, l \le 3 \\
                                  & v_{ij}^2 = v_{ij} = v_{ij}^*\\
                                  & w_{ij}^2  = w_{ij} = w_{ij}^*\\
    & A_{C_4} v = v A_{C_4}\\
                                  & v_{ij} w_{kl} = w_{kl} v_{ij}
        ) 
    .\end{align*}
    Here, $A_{C_4}$ is the adjacency matrix of the $4$-cycle  $C_4$ and $v$ is the matrix with entries $v_{ij}$ for $1 \le i, j \le 4$.

    We now claim:
    \begin{enumerate}
        \item \label{c12_3plus_6:claim1} The matrix $\hat{u}$ with entries from $C(H_2^+ \times S_3)$ defined by
            \[
                \hat{u}_{(i_1, i_2)(j_1, j_2)} = v_{i_1 j_1} w_{i_2 j_2}
            \] 
            satisfies the relations of $C(G_{aut}^+(\Gamma))$.
        \item \label{c12_3plus_6:claim2} The elements $\hat{v}_{ij}$ and $\hat{w}_{kl}$ in $C(G_{aut}^+(\Gamma))$ defined by
            \begin{align*}
                \hat{v}_{ij} &= u_{(i, 1)(j, 1)} + u_{(i, 1)(j, 2)} + u_{(i, 1)(j, 3)}\\
                \hat{w}_{kl} &= u_{(1, k)(1, l)} + u_{(1, k)(2, l)} + u_{(1, k)(3, l)} + u_{(1, k)(4, l)}
            \end{align*}
            satisfy the relations of $C(H_2^+ \times S_3)$.
    \end{enumerate}
    Let us first prove \ref{c12_3plus_6:claim1}:

    For this, observe the following: We have
    \begin{align*}
        &(i_1, i_2) \sim (k_1, k_2)\text{ in } \Gamma \\
        \Longleftrightarrow \;&(i_2 = k_2\, \land \,i_1 \sim_{C_4} k_1 ) \lor (i_2 \neq k_2 \,\land\, i_1 \neq k_1\, \land\, i_1 \not \sim_{C_4} k_1)
    \end{align*}
    Let now $(i_1, i_2)\sim (k_1, k_2)$ and $(j_1, j_2) \not \sim (l_1, l_2)$, i.e. for $(i_1, i_2)$ and $(k_1, k_2)$ the above 
    observation holds, while for $(j_1, j_2)$ and $(l_1, l_2)$ we have:
    \begin{align*}
        (j_2 \neq l_2 \lor j_1 \not \sim_{C_4} l_1) \land (j_2 = l_2 \lor j_1 = l_1 \lor j_1 \sim_{C_4} l_1)
    .\end{align*}
    In order to show that $\hat{u}_{(i_1, i_2)(j_1, j_2)} u_{(k_1, k_2)(l_1, l_2)} = 0$, we will make the following case distinction:
    \begin{itemize}
        \item Case $i_2 = k_2$:
            \begin{itemize}
                \item Case $j_2 \neq  l_2$: \\
                    We have $\hat{u}_{(i_1, i_2)(j_1,j_2)}\hat{u}_{(k_1, k_2)(l_1, l_2)} = 
                    v_{i_1 j_1}w_{i_2 j_2} v_{k_1 l_1} w_{i_2 l_2} = v_{i_1 j_1}w_{i_2j_2}w_{i_2l_2} v_{k_1l_1}$ \\
                    but $w_{i_2j_2} w_{i_2l_2} = 0$ since $j_2 \neq l_2$ and thus $\hat{u}_{(i_1, i_2)(j_1, j_2)} u_{(k_1, k_2)(l_1, l_2)} = 0$.
                \item Case $j_2 = l_2$:\\ We have $i_1 \sim_{C_4} k_1$ but $j_1 \not \sim_{C_4} l_1$  and thus
                    $v_{i_1 j_1}v_{k_1l_1} = 0$ and therefore also $\hat{u}_{(i_1, i_2)(j_1, j_2)} u_{(k_1, k_2)(l_1, l_2)} = 0$.
            \end{itemize}
        \item Case $i_2 \neq k_2$:
            \begin{itemize}
                \item Case $j_2 = l_2$:\\ We have $w_{i_2j_2} w_{k_2j_2} = 0$ since $i_2 \neq k_2$
                    and thus \\$\hat{u}_{(i_1, i_2)(j_1, j_2)} u_{(k_1, k_2)(l_1, l_2)} = 0$.
                \item Case $j_2 \neq l_2, j_1 = l_1$:\\
                    We have $v_{i_1 j_1} v_{k_1 j_1} = 0$ since $i_1 \neq  k_1$ and thus $\hat{u}_{(i_1, i_2)(j_1, j_2)} u_{(k_1, k_2)(l_1, l_2)} = 0$.
                \item Case $j_2 \neq l_2, j_1 \sim_{C_4} l_1$:\\
                    We have $v_{i_1 j_1} v_{k_1l_1} = 0$ since $i_1 \not\sim_{C_4} k_1$ but $j_1 \sim_{C_4} l_1$ and thus \\$\hat{u}_{(i_1, i_2)(j_1, j_2)} u_{(k_1, k_2)(l_1, l_2)} = 0$.
            \end{itemize}
    \end{itemize}
    A similar computation will show, that $\hat{u}_{(i_1, i_2)(j_1, j_2)} u_{(k_1, k_2)(l_1, l_2)} = 0$
    if $(i_1, i_2) \not \sim (k_1, k_2)$ and $(j_1, j_2) \sim (l_1, l_2)$ holds.

    It is immediate to see, that $\hat{u}_{(i_1, i_2)(j_1, j_2)}$ is a projection again and that the rows and columns 
    of $\hat{u}$ sum up to one from the relations of $C(H_2^+ \times S_3)$.
    We thus have that $\hat{u}$ satisfies the relations of $C(G_{aut}^+(\Gamma))$ and from the universal property of 
    $C(G_{aut}^+(\Gamma))$ we get a $*$-homomorphism
    \begin{align*}
        \phi : C(G_{aut}^+(\Gamma)) &\to C(H_2^+ \times S_3)\\
        u_{(i_1, i_2)(j_1, j_2)} &\mapsto \hat{u}_{(i_1, i_2)(j_1, j_2)} = v_{i_1 j_1} w_{i_2 j_2}.
    \end{align*} 
    It is also easy to see that all the generators of $C(H_2^+ \times S_3)$ are in the image of $\phi$:
    for example, $v_{i_1j_1}$ is the image of 
    \[
        u_{(i_1, 1)(j_1, 1)} + u_{(i_1, 1)(j_1, 2)} + u_{(i_1, 1)(j_1, 3)}
    \]
    since
    \[
        w_{11}+ w_{12} + w_{13} = 1
    ,\] 
    and a similar argument holds for $w_{i_2 j_2}$.
    Therefore, $\phi$ is surjective.

    We will now prove \ref{c12_3plus_6:claim2}.
    Firstly, using the Groebner basis for the ideal generated by the relations of $C(G_{aut}^+(\Gamma))$, we can see that
    \begin{align}
        \label{c12_3plus_6:gb_comp}
        \hat{v}_{i_1j_1} \hat{w}_{i_2j_2} = u_{(i_1, i_2)(j_1, j_2)} = \hat{w}_{i_2j_2} \hat{v}_{i_1 j_1} 
    .\end{align} 
    In particular, we see that all the elements $\hat{v}_{ij}$ commute with all the $\hat{w}_{kl}$ as desired.
    It is moreover easy to see, that $\hat{v}_{ij}^* = \hat{v}_{ij}$ and $\hat{w}_{kl}^* = \hat{w}_{kl}$
    since all the entries of $u$ are already selfadjoint.

    Next, we compute 
    \begin{align*}
        \hat{v}_{ij}^2 =& (u_{(i, 1)(j, 1)} + u_{(i, 1)(j, 2)} + u_{(i, 1)(j, 3)})^2\\
                    =& u_{(i, 1)(j, 1)} + \underbrace{u_{(i, 1)(j, 1)} u_{(i, 1)(j, 2)} + u_{(i, 1)(j, 1)}u_{(i, 1)(j, 3)}}_{=0} +\\
                     & u_{(i, 1)(j, 2)} + \underbrace{u_{(i, 1)(j, 2)} u_{(i, 1)(j, 1)} + u_{(i, 1)(j, 2)}u_{(i, 1)(j, 3)}}_{=0} +\\
                     &  u_{(i, 1)(j, 3)} + \underbrace{u_{(i, 1)(j, 3)} u_{(i, 1)(j, 1)} + u_{(i, 1)(j, 3)}u_{(i, 1)(j, 2)}}_{=0}\\
                    =& \hat{v}_{ij}
    .\end{align*}
    A similar computation can be done to see that $\hat{w}_{kl}^2 = \hat{w}_{kl}$, and thus all elements $\hat{v}_{ij}$ and
    $\hat{w}_{kl}$ are projections.

    Next, it is easy to see that 
    \begin{align*}
        \sum_{a=1}^4 \hat{v}_{ia} &= \sum_{a=1}^4 \sum_{j=1}^3 u_{(i, 1)(a, j)} = 1\\
        \sum_{b=1}^3 \hat{w}_{kb} &= \sum_{b = 1}^3 \sum_{l=1}^4 u_{(k, 1)(b, l)} = 1
    .\end{align*} 
    Using again the Groebner basis, we see that also
    \begin{align*}
        \sum_{a=1}^4 \hat{v}_{ai} = 1 = \sum_{b=1}^3 \hat{w_{kb}}
    \end{align*}
    holds.

    Lastly, we now need to check, that $\hat{v}$ satisfies 
    \[
        A_{C_4} v = v A_{C_4}
    ,\]
    in other words, we need to show that 
    \[
        \hat{v}_{ij} \hat{v}_{kl} = 0 \text{ if } i \sim_{C_4} k, \; j \not\sim_{C_4} l \text{ or }i \not \sim_{C_4} k\; j \sim_{C_4} l
    .\]
    Let therefore $i \sim_{C_4} k$ and $j \not\sim_{C_4} l$ be given.
    Using the identity in~\ref{c12_3plus_6:gb_comp}, we write
    \begin{align*}
        \hat{v}_{ij} = \hat{v}_{ij}\hat{w}_{1, 1} + \hat{v}_{ij}\hat{w}_{1, 2} + \hat{v}_{ij} \hat{w}_{1, 3}
    \end{align*}
    and compute
    \begin{align*}
        \hat{v}_{ij}\hat{v}_{kl} = &(\hat{v}_{ij}\hat{w}_{1, 1} + \hat{v}_{ij}\hat{w}_{1, 2} + \hat{v}_{ij} \hat{w}_{1, 3})
        (\hat{v}_{kl}\hat{w}_{1, 1} + \hat{v}_{kl}\hat{w}_{1, 2} + \hat{v}_{kl} \hat{w}_{1, 3})\\
        = &\hat{v}_{ij} \hat{w}_{1, 1} \hat{v}_{kl}\hat{w}_{1, 1} + \hat{v}_{ij} \hat{w}_{1, 1} \hat{v}_{kl} \hat{w}_{1, 2}+
           \hat{v}_{ij} \hat{w}_{1, 1} \hat{v}_{kl}\hat{w}_{1, 3} + \\
          &\hat{v}_{ij} \hat{w}_{1, 2} \hat{v}_{kl} \hat{w}_{1, 1} + \hat{v}_{ij} \hat{w}_{1, 2}\hat{v}_{kl}\hat{w}_{1, 2}
          + \hat{v}_{ij}\hat{w}_{1, 2} \hat{v}_{kl}\hat{w}_{1, 3}+\\
          &\hat{v}_{ij} \hat{w}_{1, 3} \hat{v}_{kl} \hat{w}_{1, 1} + \hat{v}_{ij} \hat{w}_{1, 3}\hat{v}_{kl}\hat{w}_{1, 2}
          + \hat{v}_{ij}\hat{w}_{1, 3} \hat{v}_{kl}\hat{w}_{1, 3}
    .\end{align*}
    We have by~\ref{c12_3plus_6:gb_comp}, that $\hat{v}_{kl}$ and $\hat{w}_{ab}$ commute. 
    Moreover, we have seen above that
    $\hat{w}_{ab}$ and $\hat{w}_{ac}$ are projections that sum up to one and thus $\hat{w}_{ab} \hat{w}_{ac} = 0$ for $b \neq c$.
    We thus get that any of the products in the above sum where both
    $\hat{w}_{1b}$ and $\hat{w}_{1c}$ appear for $b \neq c$ are $0$ and we are only left with
     \begin{align*}
         \hat{v}_{ij} \hat{v}_{kl} =& \hat{v}_{ij}\hat{w}_{1, 1} \hat{v}_{kl}\hat{w}_{1, 1} +
                                   \hat{v}_{ij}\hat{w}_{2, 2} \hat{v}_{kl}\hat{w}_{2, 2} +
                                   \hat{v}_{ij}\hat{w}_{3, 3} \hat{v}_{kl}\hat{w}_{3, 3}
    .\end{align*}
    Using again~\ref{c12_3plus_6:gb_comp}, we rewrite this to 
    \begin{align*}
        \hat{v}_{ij} \hat{v}_{kl} = & u_{(i, 1)(j, 1)} u_{(k, 1)(l, 1)}
        + u_{(i, 1)(j, 2)} u_{(k, 1)(l, 2)} + u_{(i, 1)(j, 3)}u_{(k, 1)(l, 3)}
    .\end{align*}
    However, since by assumption $i \sim_{C_4} k$ but $j \not \sim_{C_4} l$, we have $(i, 1) \sim (k, 1)$
    but $(j, a)\not \sim (l, a)$ for any  $1 \le a \le 3$. We thus get 
    \[
        u_{(i, 1)(j, a)} u_{(k, 1)(l, a)} = 0
    \] 
    and thus
    \[
        \hat{v}_{ij} \hat{v}_{kl} = 0
    .\] 
    A similar computation yields the same result whenever $i \not \sim_{C_4} k$ and $j \sim_{C_4} l$ holds.
    All in all, the universal property of $C(H_2^+ \times S_3)$ yields again a $*$-homomorphism
    \begin{align*}
        \phi' : C(H_2^+ \times S_3) &\to C(G_{aut}^+(\Gamma)\\
        v_{i_1 j_1}w_{i_2 j_2} &\mapsto u_{(i_1, i_2)(j_1, j_2)}
    .\end{align*}
    It is easy to see, that all generators $u_{(i_1, i_2)(j_1, j_2)}$ of $C(G_{aut}^+(\Gamma))$ are in the image and thus
    $\phi'$ is surjective. We note moreover, that  $\phi'$ is exactly the inverse of  $\phi$ and we thus get
    that $H_2^+ \times S_3$ and $G_{aut}^+(\Gamma)$ are $*$-isomorphic.
  \end{proof}

\printbibliography
\newpage
\appendix
\section{Omitted proofs}
\counterwithin{thm}{section}
\label{appendix:omitted-proofs}
In this section we collect again all propositions that have not been proven above.
\begin{prop}
    The graph $C_{12}(2)$ does not have quantum symmetries.
\end{prop}
\begin{proof}
    We will again show the commutation of $u_{i 1}$ with all other other generators of $\QBan(C_{12}(2))$.

    First, we consider the vertices in distance $1$. These are the vertices $2, 3, 11$ and  $12$.
    Since we can write  $u_{i 1} u_{k l} = u_{i 1} u_{k l} \sum_{\substack{p \sim 1}} u_{i p} $ for a vertex $l \sim 1$, we see that
    \begin{align*}
        u_{i 1} u_{k 2} = u_{i 1} u_{k 2}\left( u_{i 1} + u_{i 3} + u_{i 4} + u_{i 12}\right)\\
        u_{i 1} u_{k 3} = u_{i 1} u_{k 3}\left( u_{i 1} + u_{i 2} + u_{i 4} + u_{i 5}\right)\\
        u_{i 1} u_{k 11} = u_{i 1} u_{k 11} \left( u_{i 1} + u_{i 9} + u_{i 10} + u_{i 12} \right) \\
        u_{i 1} u_{k 12} = u_{i 1} u_{k 12} \left( u_{i 1} + u_{i 2} + u_{i 10} + u_{i 11} \right)
    .\end{align*}
    Applying Lemma~\ref{lem:choose_q_middle} to these vertices, we see that on the right side, only the term with $u_{i 1}$ multiplied 
    from the right does not vanish, and we thus get commutation of $u_{i 1}$ and $u_{k l}$ for $l \sim 1$ by Lemma~\ref{lem:adj_commute}.
    \begin{table}[H]
         \begin{tabularx}{0.2\linewidth}{|*{4}{X|}}
            \multicolumn{4}{c}{Lemma~\ref{lem:choose_q_middle}}\\
            \hline
            $j$ &$l$ &  $p$ &  $q$\\
             \hline
            $1$&$2$ & $3$ & $12$\\
               & & $4$ & $12$\\
               & & $12$ & $10$\\
               &$3$& $2$& $11$\\
               && $4$& $12$\\
               && $5$& $12$\\
               \hline
        \end{tabularx}
        \quad \quad
        \begin{tabularx}{0.2\linewidth}{|*{4}{X|}}
            \multicolumn{4}{c}{Lemma~\ref{lem:choose_q_middle}}\\
            \hline
            $j$ &$l$ &  $p$ &  $q$\\
             \hline
               $1$&$11$& $9$& $12$\\
               && $10$& $9$\\
               && $12$& $10$\\
               &$12$& $2$& $11$\\
               && $10$& $9$\\
               && $11$& $11$\\
               \hline
        \end{tabularx}
    \end{table}
    Next, we consider vertices in distance $2$, which are the vertices $4, 5, 9$ and  $10$. For these, 
    we can apply Lemma~\ref{lem:choose_q_right}, since we already know that for all vertices $j \sim l$, that
    $u_{ij} u_{kl} = u_{kl} u_{ij}$:
        \begin{table}[H]
             \begin{tabularx}{0.225\linewidth}{|*{3}{X|}l|}
                \multicolumn{4}{c}{Lemma~\ref{lem:choose_q_right}}\\
                \hline
                $j$ &$l$ &  $q$ &  $p$\\
                 \hline
                $1$&$4$ & $3$ & $\{1\} $\\
                   &$5$& $6$ & $\{1\} $\\
                   &$9$& $8$ & $\{1\} $\\
                   &$10$ & $11$ &  $\{1\} $\\
                   \hline
            \end{tabularx}
        \end{table}
        In distance $3$, there are only three vertices,  $6, 7$ and  $8$. Applying Lemma~\ref{lem:choose_q_right}
        to them leads to the following table:
        \begin{table}[H]
             \begin{tabularx}{0.225\linewidth}{|*{3}{X|}l|}
                \multicolumn{4}{c}{Lemma~\ref{lem:choose_q_right}}\\
                \hline
                $j$ &$l$ &  $q$ &  $p$\\
                 \hline
                $1$&$6$ & $5$ & $\{1\} $\\
                   &$7$& $5$ & $\{1, 2\} $\\
                   &$8$& $6$ & $\{1\} $\\
                   \hline
            \end{tabularx}
        \end{table}
        We thus have to consider $u_{i 1} u_{k 7}$ individually. 
        By the above table, we know that
         \begin{align*}
             u_{i 1} u_{k 7} &= u_{i 1} u_{k 7} \left( u_{i 1} + u_{i 2} \right)
        .\end{align*}
        Applying Lemma~\ref{lem:monomial_zero} with $q := 4$ yields that
        \[
            u_{i 1} u_{k 7} u_{i 2} = 0
        \] 
        and we thus get
        \[
            u_{i 1} u_{k 7} = u_{i 1} u_{k 7} u_{i 1} = u_{k 7} u_{i 1}
        .\]
        Lemma~\ref{lem:monomial_zero} is applicable since $d(7, 4) = 2$ and we have shown above that for any vertices 
        in distance  $2$, the generators commute.

        All in all, we get that all generators commute, and thus  $\QBan(C_{12}(2)) = G_{aut}(C_{12}(2))$.
\end{proof}
\begin{prop}
    The graph $C_{12}(4)$ does not have quantum symmetries.
\end{prop}
\begin{proof}
    We first note, that for each vertex there is only a single vertex in distance $3$, and thus for vertices $i, j, k, l$ with
     $d(i, k) = 3 = d(j, l)$ we have
      \[
     u_{ij} u_{kl} = u_{ij} u_{kl} u_{ij} = u_{kl} u_{ij}
     .\] 
    Next, we will see that $u_{i 1}$ commutes with all $u_{kl}$ where $d(l, 1) = 1$ by applying Lemma~\ref{lem:choose_q_middle}:
    \begin{table}[H]
         \begin{tabularx}{0.2\linewidth}{|*{4}{X|}}
            \multicolumn{4}{c}{Lemma~\ref{lem:choose_q_middle}}\\
            \hline
            $j$ &$l$ &  $p$ &  $q$\\
             \hline
            $1$&$2$ & $3$ & $1$\\
               & & $6$ & $9$\\
               & & $10$ & $5$\\
               &$5$& $4$& $8$\\
               && $6$& $9$\\
               && $9$& $1$\\
               \hline
        \end{tabularx}
        \quad \quad
        \begin{tabularx}{0.2\linewidth}{|*{4}{X|}}
            \multicolumn{4}{c}{Lemma~\ref{lem:choose_q_middle}}\\
            \hline
            $j$ &$l$ &  $p$ &  $q$\\
             \hline
               $1$&$9$& $5$& $1$\\
               && $8$& $4$\\
               && $10$& $5$\\
               &$12$& $4$& $8$\\
               && $8$& $4$\\
               && $11$& $1$\\
               \hline
        \end{tabularx}
    \end{table}
    For $d(l, 1) = 2$, we can apply Lemma~\ref{lem:choose_q_right} and get the following:
        \begin{table}[H]
             \begin{tabularx}{0.225\linewidth}{|*{3}{X|}l|}
                \multicolumn{4}{c}{Lemma~\ref{lem:choose_q_right}}\\
                \hline
                $j$ &$l$ &  $q$ &  $p$\\
                 \hline
                $1$&$3$ & $7$ & $\{1\} $\\
                   &$4$& $3$ & $\{1, 6\} $\\
                   &$6$& $7$ & $\{1\} $\\
                   &$8$& $7$ & $\{1\} $\\
                   &$10$& $2$ & $\{1, 3\} $\\
                   &$11$& $7$ & $\{1\} $\\
                   \hline
            \end{tabularx}
        \end{table}
    We thus already get commutation for all vertices except for $l = 4$ and  $l = 10$.
    Applying Lemma~\ref{lem:monomial_zero} with $q := 10$ to  $1$ and  $4$ yields 
    $u_{i 1} u_{k 4} u_{i 6} = 0$ and thus we have commutation of $u_{i 1}$ and $u_{k 4}$. Similarly, Lemma~\ref{lem:monomial_zero} with $q:= 4$ applied
    to  $1$ and  $10$ yields commutation of  $u_{i 1}$ and $u_{k 10}$.

    All in all, we get that all generators commute and thus $\QBan(C_{12}(4)) = G_{Aut}(C_{12}(4))$.
\end{proof}
\begin{prop}
    The graph $C_{12}(2, 6)$ does not have quantum symmetries.
\end{prop}
\begin{proof}
    As was the case with $C_{12}(4)$, we can show commutation for all $l$ with  $d(l, 1) = 1$ by applying Lemma~\ref{lem:choose_q_middle}:
    \begin{table}[H]
        \begin{tabularx}{0.2\linewidth}[t]{|*{4}{X|}}
            \multicolumn{4}{c}{Lemma~\ref{lem:choose_q_middle}}\\
            \hline
            $j$ &$l$ &  $p$ &  $q$\\
             \hline
            $1$&$2$ & $3$ & $1$\\
               & & $4$ & $5$\\
               & & $8$ & $3$\\
               & & $12$ & $3$\\
               &$3$& $2$& $4$\\
               && $4$& $5$\\
               && $5$& $2$\\
               && $9$& $2$\\
               \hline
        \end{tabularx}
        \quad \quad
        \begin{tabularx}{0.2\linewidth}[t]{|*{4}{X|}}
            \multicolumn{4}{c}{Lemma~\ref{lem:choose_q_middle}}\\
            \hline
            $j$ &$l$ &  $p$ &  $q$\\
             \hline
               $1$&$7$& $5$& $6$\\
               && $6$& $2$\\
               && $8$& $3$\\
               &&$9$& $8$\\
               &$11$& $5$& $12$\\
               && $9$& $10$\\
               && $10$& $2$\\
               &&$12$& $3$\\
               \hline
        \end{tabularx}
        \quad \quad
        \begin{tabularx}{0.2\linewidth}[t]{|*{4}{X|}}
            \multicolumn{4}{c}{Lemma~\ref{lem:choose_q_middle}}\\
            \hline
            $j$ &$l$ &  $p$ &  $q$\\
             \hline
               $1$&$12$& $2$& $4$\\
               && $6$& $2$\\
               && $10$& $2$\\
               &&$11$& $1$\\
               \hline
        \end{tabularx}
    \end{table}
    Using Lemma~\ref{lem:choose_q_middle} is also sufficient for vertices $4, 6, 8$ and  $10$:
    \begin{table}[H]
        \begin{tabularx}{0.2\linewidth}[t]{|*{4}{X|}}
            \multicolumn{4}{c}{Lemma~\ref{lem:choose_q_middle}}\\
            \hline
            $j$ &$l$ &  $p$ &  $q$\\
             \hline
            $1$&$4$ & $7$ & $2$\\
               & & $8$ & $10$\\
               & & $9$ & $2$\\
               & & $11$ & $3$\\
               & & $12$ & $7$\\
               &$6$& $2$& $11$\\
               && $3$& $5$\\
               && $9$& $8$\\
               && $10$& $3$\\
               && $11$& $2$\\
               \hline
        \end{tabularx}
        \quad \quad
        \begin{tabularx}{0.2\linewidth}[t]{|*{4}{X|}}
            \multicolumn{4}{c}{Lemma~\ref{lem:choose_q_middle}}\\
            \hline
            $j$ &$l$ &  $p$ &  $q$\\
             \hline
               $1$&$8$& $3$& $4$\\
               && $4$& $6$\\
               && $5$& $2$\\
               &&$11$& $5$\\
               &&$12$& $3$\\
               &$10$& $2$& $7$\\
               && $3$& $7$\\
               && $5$& $4$\\
               &&$6$& $3$\\
               &&$7$& $2$\\
               \hline
        \end{tabularx}
    \end{table}
    The only remaining vertices are then $5$ and  $9$. 
    We thus have to consider 
    \[
        u_{i 1} u_{k 5} \quad \text{ and } \quad u_{i 1} u_{k 9}
    .\] 
    Multiplying both of these with $1 = \sum_{r \in V(C_{12}(2, 6))}$, we get
    \begin{align*}
        u_{i 1} u_{k 5} &= u_{i 1} u_{k 5} \left( u_{i 1} + u_{i 2} + u_{i 8} + u_{i 9} + u_{i 10} + u_{i 12} \right) \\
        u_{i 1} u_{k 9} &= u_{i 1} u_{k 9} \left( u_{i 1} + u_{i 2} + u_{i 4} + u_{i 5} + u_{i 6} + u_{i 12} \right) 
    .\end{align*}
    Applying Lemma~\ref{lem:choose_q_middle} simplifies this:
    \begin{table}[H]
        \begin{tabularx}{0.2\linewidth}[t]{|*{4}{X|}}
            \multicolumn{4}{c}{Lemma~\ref{lem:choose_q_middle}}\\
            \hline
            $j$ &$l$ &  $p$ &  $q$\\
             \hline
            $1$&$5$ & $2$ & $8$\\
               & & $8$ & $10$\\
               & & $10$ & $3$\\
               & & $12$ & $10$\\
               \hline
        \end{tabularx}
        \quad \quad
        \begin{tabularx}{0.2\linewidth}[t]{|*{4}{X|}}
            \multicolumn{4}{c}{Lemma~\ref{lem:choose_q_middle}}\\
            \hline
            $j$ &$l$ &  $p$ &  $q$\\
             \hline
               $1$&$9$& $2$& $4$\\
               && $4$& $6$\\
               && $6$& $3$\\
               &&$12$& $6$\\
               \hline
        \end{tabularx}
    \end{table}
    to
    \begin{align*}
        u_{i 1} u_{k 5} &= u_{i 1} u_{k 5} \left( u_{i 1} + u_{i 9}\right) \\
        u_{i 1} u_{k 9} &= u_{i 1} u_{k 9} \left( u_{i 1} + u_{i 5}\right) 
    .\end{align*}
    We would now like to apply Lemma~\ref{lem:monomial_zero} for the last step. However so far, we only know
    that for any vertex $q$ with  $d(q, 5) = 1$ the generators $u_{k 5}$ and  $u_{r q}$ commute. But for all
    of these vertices, we have $d(q, 1) = d(q, 9)$, and thus the lemma is not applicable. A similar situation 
    holds for any  $q$ with  $d(q, 9) = 1$.

    Consider the following two automorphisms of $C_{12}(2, 6)$:
    \begin{align*}
        \phi_1 &= (1 , 5)\; (2, 4)\; (6, 12)\;(7, 11)\;(8, 10)\\
        \phi_2 &= (1, 9)\;(2, 8)\;(3, 7)\;(4, 6)\;(10, 12)
    .\end{align*}
    It holds that $\phi_1(1) = 5$ and $\phi_1(4) = 2$. Since we already know by the above that 
    $u_{i 1} u_{k 4} = u_{k 4} u_{i 1}$ for any  $i, k \in V(C_{12}(2, 6))$, we get by
    Lemma~\ref{lem:get_commutation_by_automorphism} that also $u_{i 5} u_{k 2}$ commute for any  $i, k$.
    We thus can apply Lemma~\ref{lem:monomial_zero} with  $q := 2$ to  $1, 5$ and  $9$ and get
     \[
         u_{i 1} u_{k 5} u_{i 9} = 0
    ,\]
    and therefore $u_{i 1}$ and  $u_{k 5}$ commute.

    Similarly, we observe that  $\phi_2(1)= 9$ and  $\phi_2(6) = 4$. Again, we already know that $u_{i 1}$
    and  $u_{k 6}$ commute for any  $i, k$ and thus by Lemma~\ref{lem:get_commutation_by_automorphism}
    so do  $u_{i 9}$ and  $u_{k 4}$.
    Then applying Lemma~\ref{lem:monomial_zero} with  $q:= 4$ to  $1, 9$ and $5$, we get that 
    $u_{i 1}$ and  $u_{k 9}$ commute.

    We thus showed that  $u_{i 1}$ commutes with every other generator, and since  $C_{12}(2, 6)$ is 
    vertex-transitive this is enough to show that $\QBan(C_{12}(2, 6))$ is commutative.
\end{proof}
\begin{prop}
    The graph $C_{12}(3, 6)$ does not have quantum symmetries.
\end{prop}
\begin{proof}
     We first note, that the following permutation is an automorphism of $C_{12}(3, 6)$:
    \[
        \phi := (2, 12)\;(3, 11)\;(4, 10)\;(5, 9)\;(6, 8)
    .\] 
    This is the mirroring on the edge that goes through $1$ and  $7$. In particular, we have that
    $\phi(1) = 1$, and thus it suffices to show that  $u_{i 1}$ and  $u_{k l}$ commute for 
    $l \in \{2, 3, 4, 5, 6, 7\} $, since the commutation with the rest of the generators then 
    follows from Lemma~\ref{lem:get_commutation_by_automorphism} with $\phi$.

    We will now show, that $u_{i 1}$ commutes with $u_{k 5}$. 
    To see this, we apply Lemma~\ref{lem:choose_q_middle}:
    \begin{table}[H]
        \begin{tabularx}{0.2\linewidth}[t]{|*{4}{X|}}
            \multicolumn{4}{c}{Lemma~\ref{lem:choose_q_middle}}\\
            \hline
            $j$ &$l$ &  $p$ &  $q$\\
             \hline
            $1$&$5$ & $3$ & $6$\\
               & & $7$ & $12$\\
               & & $9$ & $4$\\
               & & $10$ & $9$\\
               && $12$& $4$\\
               \hline
        \end{tabularx}
    \end{table}
    It then immediately follows that $u_{i 1}$ also commutes with $u_{k 9}$ by 
    Lemma~\ref{lem:get_commutation_by_automorphism} with $\phi$.

    Next, we consider all vertices $l$ such that $d(l, 1) = 1$, i.e. $l \in \{2, 4, 7\} $.
    We again apply Lemma~\ref{lem:choose_q_middle}:
    \begin{table}[H]
        \begin{tabularx}{0.2\linewidth}[t]{|*{4}{X|}}
            \multicolumn{4}{c}{Lemma~\ref{lem:choose_q_middle}}\\
            \hline
            $j$ &$l$ &  $p$ &  $q$\\
             \hline
            $1$&$2$ & $5$ & $7$\\
               & & $8$ & $4$\\
               & & $11$ & $5$\\
               &$4$ & $3$ & $7$\\
               && $5$& $7$\\
               &$7$&$6$&$3$\\
               && $8$& $4$\\
               \hline
        \end{tabularx}
    \end{table}
    We get that $u_{i 1} u_{k 2} u_{i 3} = 0$ by applying Lemma~\ref{lem:monomial_zero} with $q := 6$. It is applicable
    since applying Lemma~\ref{lem:get_commutation_by_automorphism} with the automorphism $v \mapsto v + 1 \mod 12$
    yields that $u_{k 2}$ and $u_{r 6}$ commute, since we already have commutation of $u_{k 1}$ and $u_{r 5}$. 
    We thus have commutation of $u_{i 1}$ and $u_{k 2}$.

    For $l = 4$, we still need to show that  $u_{i 1} u_{k 4} u_{i 7} = 0$ and $u_{i 1} u_{k 4} u_{i 10} = 0$. 
    We have commutation of $u_{k 4}$ with both $u_{r 8}$ and $u_{r 12}$ by Lemma~\ref{lem:get_commutation_by_automorphism}
    with the automorphism $v \mapsto v + 3 \mod 12$, since we have commutation of $u_{k 1}$ with both $u_{r 5}$ and $u_{r 9}$.
    We can thus apply Lemma~\ref{lem:monomial_zero} to both monomials, once with $q := 8$ and  once with  $q := 12$,
    and get that $u_{i 1}$ and $u_{k 5}$ commute.

    For $l = 7$, we need to show that  $u_{i 1} u_{k 7} u_{i 4} = 0$ and $u_{i 1} u_{k 7} u_{i 10} = 0$.
    Again, similar to above, we have commutation of $u_{k 7}$ with both $u_{r 3}$ and $u_{r 11}$ by
    Lemma~\ref{lem:get_commutation_by_automorphism} with the automorphism $v \mapsto v + 6$.
    We can then again apply Lemma~\ref{lem:monomial_zero} once with  $q := 3$ and once with  $q := 11$ and get that
     $u_{i 1}$ and $u_{k 7}$ commute. 
     
     We now have commutation of all generators $u_{i 1}$ and $u_{kl}$ where $d(l, 1) = 1$
     and thus we also have commutation of all generators $u_{ij}$ and $u_{kl}$ where $d(i, k) = d(j, l) = 1$.

     Only missing are now the vertices $3$ and  $6$. For these, applying first Lemma~\ref{lem:choose_q_right}
     and then  Lemma~\ref{lem:choose_q_middle} shows the commutation of  $u_{i 1}$ with $u_{k 3}$ and $u_{k 6}$ :
    \begin{table}[H]
        \begin{tabularx}{0.26\linewidth}[t]{|*{3}{X|}l|}
            \multicolumn{4}{c}{Lemma~\ref{lem:choose_q_right}}\\
            \hline
            $j$ &$l$ &  $q$ &  $p$\\
             \hline
            $1$&$3$ & $12$ & $\{1, 11\} $\\
               &$6$& $5$ & $\{1, 10\} $\\
               \hline
        \end{tabularx}
        \quad \quad
        \begin{tabularx}{0.2\linewidth}[t]{|*{4}{X|}}
            \multicolumn{4}{c}{Lemma~\ref{lem:choose_q_middle}}\\
            \hline
            $j$ &$l$ &  $p$ &  $q$\\
             \hline
            $1$&$3$ & $11$ & $4$\\
               &$6$ & $10$ & $2$\\
               \hline
        \end{tabularx}
    \end{table}
    We thus have that all generators commute and we get $\QBan(C_{12}(3, 6)) = G_{aut}(C_{12}(3, 6))$.

\end{proof}
\begin{prop}
    The graph $C_{12}(4, 6)$ does not have quantum symmetries.
\end{prop}
\begin{proof}
    Again we note that we have the following automorphism of $C_{12}(4, 6)$:
    \[
        \phi := (2, 12)\;(3, 11)\;(4, 10)\;(5, 9)\;(6, 8)
    ,\] 
    i.e. the mirroring on the edge that goes through $1$ and  $7$. 
    As was the case for $C_{12}(3, 6)$, it suffices to show that  $u_{i 1}$ and  $u_{k l}$ commute for 
    $l \in \{2, 3, 4, 5, 6, 7\} $, since the commutation with the rest of the generators then 
    follows from Lemma~\ref{lem:get_commutation_by_automorphism} with $\phi$.

    We now first consider $u_{i 1}$ and $u_{k 5}$. Applying Lemma~\ref{lem:choose_q_middle} already yields their commutation:
    \begin{table}[H]
        \begin{tabularx}{0.2\linewidth}[t]{|*{4}{X|}}
            \multicolumn{4}{c}{Lemma~\ref{lem:choose_q_middle}}\\
            \hline
            $j$ &$l$ &  $p$ &  $q$\\
             \hline
            $1$&$5$ & $4$ & $8$\\
               & & $6$ & $9$\\
               & & $9$ & $1$\\
               & & $11$ & $9$\\
               \hline
        \end{tabularx}
    \end{table}
    Using this, we can show commutation for those vertices $l$ that satisfy  $d(1, l) = 2$, i.e.  $l \in \{3, 4, 6\} $.
    We again apply Lemma~\ref{lem:choose_q_middle} to these vertices to reduce the number of terms we have
    to consider:
    \begin{table}[H]
        \begin{tabularx}{0.2\linewidth}[t]{|*{4}{X|}}
            \multicolumn{4}{c}{Lemma~\ref{lem:choose_q_middle}}\\
            \hline
            $j$ &$l$ &  $p$ &  $q$\\
             \hline
            $1$&$3$ & $10$ & $11$\\
               & & $12$ & $6$\\
               &$4$ & $2$ & $7$\\
               & & $7$ & $2$\\
               & &$9$&$7$\\
               &&  $11$& $2$\\
               &$6$& $3$& $4$\\
               && $4$& $2$\\
               && $9$& $2$\\
               \hline
        \end{tabularx}
    \end{table}
    Starting with vertex $3$, we see that 
     \[
         u_{i 1} u_{k 3} = u_{i 1} u_{k 3}\left( u_{i 1} + u_{i 5} + u_{i 6} + u_{i 8} \right) 
    .\] 
    We note, that the automorphism that rotates $1$ onto  $3$ maps  $5$ to  $7$ and we thus get
    $u_{i 3} u_{k 7} = u_{k 7} u_{i 3}$ for any  $i, k \in V(C_{12}(4, 6))$ by 
    Lemma~\ref{lem:get_commutation_by_automorphism}. We can thus apply
    Lemma~\ref{lem:monomial_zero} with $q:= 7$ to  $1, 3$ and  $5$, since  $7 \sim 1$ but  $5 \not \sim 7$
    and get  $u_{i1} u_{k 3} u_{i 5} = 0$.

    Recall, that by Definition~\ref{def:cn}, we denote by $\CN(i, j)$ the set of common neighbours of the vertices  $i$ and  $j$.
    We now get by Lemma~\ref{lem:different_numbers_common_neighbours} that 
    $u_{i 1} u_{k 3} u_{i 6} = 0$ and that  $u_{i 1} u_{k 3} u_{i 8} = 0$, since 
    $\lvert \CN(1, 3)\rvert = 3$, $\lvert \CN(3, 6)\rvert =2$ and  $\lvert \CN(3, 8)\rvert = 4$.
    We thus have commutation of $u_{i 1}$ and  $u_{k 3}$.

    Looking at vertex  $4$, we see that we have
    \[
      u_{i 1} u_{k 4} = u_{i 1} u_{k 4} \left( u_{i 1} + u_{i 6} \right) 
    .\] 
    We can apply Lemma~\ref{lem:different_numbers_common_neighbours} to see that 
    $u_{i 1} u_{k 4} u_{ i 6 } = 0$, as  $\lvert\CN(1, 4)\rvert = 2$ and $\lvert \CN(4, 6)\rvert = 6$.
    By this, we also get commutation of $u_{i 1}$ and  $u_{k 4}$. 

    Lastly, looking at vertex $6$, we see
    \[
    u_{i 1} u_{k 6} = u_{i 1} u_{k 6} \left( u_{i 1} + u_{i 8} u_{i 9} + u_{i 11} \right) 
    .\] 
    We will now show, that $u_{k 6}$ commutes with $u_{i 8}$ and $u_{i 9}$.
    If this is the case, then we get
    \begin{align*}
        u_{i 1} u_{k 6} u_{i 8} &= u_{i 1} u_{i 8} u_{k 6} = 0\\
        u_{i 1} u_{k 6} u_{i 9} &= u_{i 1} u_{i 9} u_{k 6} = 0
    .\end{align*} 

    Consider the automorphism
    \[
    \sigma = (1\; 6)(2\;5)(3\;4)(7\;12)(8\;11)(9\;10)
    .\] 
    Since $\sigma$ maps  $1$ to  $6$, we can translate the question of whether  $u_{k 6}$ commutes with 
    other generators to the question whether $u_{v 1}$ commutes with other generators, for any vertex  $v$ 
    by using Lemma~\ref{lem:get_commutation_by_automorphism}.
    If we thus want to know, whether $u _{k 6}$ and $u _{i 8}$ commute, we can instead ask, whether
    $u _{v 1}$ and $u _{w 11}$ commute.
    Since we already know that $u_{v 1}$ and $u_{w 3}$ commute for any vertices $v, w$, and since
    the mirroring automorphism  $\phi$ from above maps  $1$ to  $1$ and  $3$ to  $11$, we get that
    $u_{v 1}$ and $u_{w 11}$ commute for any $v, w$, again by Lemma~\ref{lem:get_commutation_by_automorphism}.
    Therefore, also $u_{k 6}$ and $u_{i 8}$ commute.

    Next, to see that $u_{k 6}$ and $u_{i 9}$ commute, we need to show that $u_{v 1}$ and $u_{w 10}$ commute.
    But we already know that $u_{v 1}$ and $u_{w 4}$ commute and since $\phi(4) = 10$, we get commutation
    of  $u_{v 1}$ and $u_{w 10}$ and thus also of $u_{k 6}$ and $u_{i 9}$.

    Lastly, we can apply Lemma~\ref{lem:monomial_zero} with $q:= 2$ to see that
     \[
        u_{i 1} u_{k 6} u_{i 11} = 0
    .\] 
    To apply this lemma, we need that $u_{k 6}$ commutes with $u_{r 2}$ for any vertex $r$. 
    This can easily be seen by using Lemma~\ref{lem:get_commutation_by_automorphism} with the 
    automorphism $\tau'(v) := v + 1 \mod 12$, as this maps $1$ to  $2$ and  $5$ to  $6$. 

    We thus have that $u_{i 1}$ and $u_{k 6}$ commutes and thus have commutation of all $u_{i j}$ and 
    $u_{kl}$ where $d(i, k) = d(j, l) = 2$.

    We still need to show commutation of $u_{i 1}$ with $u_{k 2}$ and with $u_{k 7}$.
    Applying Lemma~\ref{lem:choose_q_right} with  $q:= 4$ to  $1$ and  $2$ yields
     \[
         u_{i 1} u_{k 2} = u_{i 1} u_{k 2} \left( u_{i 1} + u_{i 6} \right) 
    .\] 
    Since we have already seen above that $u_{k 2}$ and $u_{i 6}$ commute we get
    \[
        u_{i 1} u_{k 2} u_{i 6} = u_{i 1} u_{i 6} u_{k 2} = 0
    \] 
    and thus get commutation of $u_{i 1}$ and $u_{k 2}$.

    Applying Lemma~\ref{lem:choose_q_right} with $q:= 4$ to  $1$ and  $7$ yields
     \[
          u_{i 1} u_{k 7} = u_{i 1} u_{k 7} \left( u_{i 1} + u_{i 6} + u_{i 11} \right) 
    .\] 
    Since there is an automorphism that maps $1$ to  $6$ and  $2$ to  $7$, namely 
    $v \mapsto v + 5 \mod 12$, we can use Lemma~\ref{lem:get_commutation_by_automorphism} to 
    get that $u_{k 7}$ and $u_{i 6}$ commute and thus
    \[
        u_{i 1} u_{k 7} u_{i 6} = 0
    .\] 
    Similarly, the automorphism $v \mapsto v - 2 \mod 12$ maps $1$ to $11$ and  $7$ to  $5$, and
    we thus get with Lemma~\ref{lem:get_commutation_by_automorphism} that  $u_{k 7}$ and $u_{i 11}$ 
    commute and thus
    \[
        u_{i 1} u_{k 7} u_{i 11} = 0
    .\] 
    All in all we get that $u_{i 1}$ and $u_{k 7}$ commute and with that, all of the generators commute
    and we get  $\QBan(C_{12}(4, 6)) = G_{aut}(C_{12}(4, 6))$.
\end{proof}
\begin{prop}
    The graph $C_{12}(2, 5^+)$ does not have quantum symmetries.
\end{prop}
\begin{proof}
    We begin by showing that $u_{i 1}$ commutes with those generators $u_{k l}$ where $d(l, 1) = 1$, i.e.  $l \in \{2, 3, 6, 11, 12\} $.
    As a first step, we apply Lemma~\ref{lem:choose_q_middle} to these vertices:
    \begin{table}[H]
        \begin{tabularx}{0.2\linewidth}[t]{|*{4}{X|}}
            \multicolumn{4}{c}{Lemma~\ref{lem:choose_q_middle}}\\
            \hline
            $j$ &$l$ &  $p$ &  $q$\\
             \hline
            $1$&$2$ & $3$ & $1$\\
               & & $9$ & $3$\\
               && $12$&  $3$\\
               &$3$ & $2$ & $4$\\
               & & $5$ & $2$\\
               & &$8$&$2$\\
               \hline
        \end{tabularx}
        \quad \quad
        \begin{tabularx}{0.2\linewidth}[t]{|*{4}{X|}}
            \multicolumn{4}{c}{Lemma~\ref{lem:choose_q_middle}}\\
            \hline
            $j$ &$l$ &  $p$ &  $q$\\
             \hline
               $1$&$6$& $5$& $2$\\
               && $7$& $2$\\
               && $8$& $2$\\
               &$11$& $9$& $3$\\
               && $10$& $2$\\
               && $12$& $3$\\
               \hline
        \end{tabularx}
        \quad \quad
        \begin{tabularx}{0.2\linewidth}[t]{|*{4}{X|}}
            \multicolumn{4}{c}{Lemma~\ref{lem:choose_q_middle}}\\
            \hline
            $j$ &$l$ &  $p$ &  $q$\\
             \hline
               $1$&$12$& $2$& $4$\\
               && $7$& $2$\\
               && $10$& $2$\\
               && $11$& $1$\\
               \hline
        \end{tabularx}
    \end{table}
    From the above, we already see that $u_{i 1}$ and $u_{k 12}$ commute. By Lemma~\ref{lem:get_commutation_by_automorphism}
    we get commutation of all $u_{i j'}$ and $u_{k l'}$ where $j'$ is uneven and  $l' = j' - 1 \mod 12$, using the 
    automorphism  $v \mapsto v + 2 \mod 12$ repeatedly for the lemma.

    Next, we note that for the commutation of  $u_{i 1}$ and $u_{k 2}$, the only thing missing is that
    \[
        u_{i 1} u_{k 2} u_{i 4} = 0
    .\] 
    This can be seen using Lemma~\ref{lem:monomial_zero} with $q:=3$ and the fact that $u_{k 2}$ and $u_{r 3}$ commute by the above.

    Similarly, we only need
    \[
        u_{i 1} u_{k 3} u_{i 4} = 0
    \] 
    to see that $u_{i 1}$ and $u_{k 3}$ commute and can see that again with Lemma~\ref{lem:monomial_zero} with $q := 2$.

    In order to see that  $u_{i 1}$ and $u_{k 6}$ commute, we need that
    \[
        u_{i 1} u_{k 6} u_{i 4} = 0
    .\] 
    By Lemma~\ref{lem:get_commutation_by_automorphism} with the automorphism
    \[
        \phi = (1 \; 6)(2\;5)(3\;4)(7\;12)(8\;11)(9\;10)
    \] 
    we get that $u_{k 6}$ and $u_{i 4}$ commute, since $\phi(1) = 6$ and  $\phi(3) = 4$ and we have already shown commutation
    of  $u_{k 1}$ and $u_{i 3}$.
    We thus get 
    \[
        u_{i 1} u_{k 6} u_{i 4} = u_{i 1}u_{i 4} u_{k 6} = 0
    \] 
    as desired.

    For the commutation of $u_{i 1}$ and $u_{k 11}$ we need
    \[
        u_{i 1} u_{k 11} u_{i 4} = 0
    .\] 
    Since we already know that $u_{k 11}$ and $u_{r 10}$ commute, we can apply Lemma~\ref{lem:monomial_zero} with $q := 10$
    to achieve this.
    We thus have commutation of all generators  $u_{ij}$ and $u_{kl}$ where $d(j, l) = 1$.

    For some vertices $l$ with  $d(l, 1) = 2$ it is enough to apply Lemma~\ref{lem:choose_q_right} first and then 
    Lemma~\ref{lem:choose_q_middle} to get that the missing monomial is $0$:
    \begin{table}[H]
        \begin{tabularx}{0.26\linewidth}[t]{|*{3}{X|}l|}
            \multicolumn{4}{c}{Lemma~\ref{lem:choose_q_right}}\\
            \hline
            $j$ &$l$ &  $q$ &  $p$\\
             \hline
            $1$&$4$ & $3$ & $\{1, 8\} $\\
               &$5$& $6$ & $\{1, 8\} $\\
               &$10$& $5$ & $\{1, 2\} $ \\
               \hline
        \end{tabularx}
        \quad \quad
        \begin{tabularx}{0.2\linewidth}[t]{|*{4}{X|}}
            \multicolumn{4}{c}{Lemma~\ref{lem:choose_q_middle}}\\
            \hline
            $j$ &$l$ &  $p$ &  $q$\\
             \hline
            $1$&$4$ & $8$ & $2$\\
               &$5$ & $8$ & $2$\\
               &$10$& $2$ &$6$\\
               \hline
        \end{tabularx}
    \end{table}
    For the other three vertices, we also first apply Lemma~\ref{lem:choose_q_right}:
    \begin{table}[H]
        \begin{tabularx}{0.26\linewidth}[t]{|*{3}{X|}l|}
            \multicolumn{4}{c}{Lemma~\ref{lem:choose_q_right}}\\
            \hline
            $j$ &$l$ &  $q$ &  $p$\\
             \hline
            $1$&$7$ & $6$ & $\{1, 4\} $\\
               &$8$& $6$ & $\{1, 4, 5\} $\\
               &$9$& $7$ & $\{1, 3, 4\} $ \\
               \hline
        \end{tabularx}
    \end{table}
    To get that $u_{i 1} u_{k 7} u_{i 4} = 0$  we apply Lemma~\ref{lem:monomial_zero} with $q := 5$.

    Next, we consider the automorphism
    \[
        \sigma = (1\;12)(2\;11)(3\;10)(4\;9)(5\;8)(6\;7)
    .\] 
    We see that $\sigma(1) = 12$ and $\sigma(5) = 8$ and thus with Lemma~\ref{lem:get_commutation_by_automorphism}
    we get that  $u_{k 8}$ and $u_{r 12}$ commute. To see now that both  $u_{i 1} u_{k 8} u_{i 4} = 0$ and
    $u_{i 1} u_{k 8} u_{i 5} = 0$ hold, we can apply Lemma~\ref{lem:monomial_zero} with $q := 12$ to each of those monomials.
    With the same automorphism $\sigma$, Lemma~\ref{lem:get_commutation_by_automorphism} shows that  $u_{k 9}$ and 
    $u_{r 12}$ commute, and again, we see that  $u_{i 1} u_{k 9} u_{i 3} = 0$ and $u_{i 1} u_{k 9} u_{i 4} = 0$ hold
    by applying Lemma~\ref{lem:monomial_zero} with $q:= 12$ to both monomials.

    All in all we see that all generators commute with $u_{i 1}$ and thus with Corollary~\ref{cor:commutation_with_one_implies_all}
    we get $\QBan(C_{12}(2, 5^+)) = G_{aut}(C_{12}(2, 5^+))$.
\end{proof}
\begin{prop}
    The graph $C_{12}(4, 5^+)$ does not have quantum symmetries.
\end{prop}
\begin{proof}
    For the vertices  $l \in \{3, 10, 11, 12\} $, it is sufficient to apply Lemma~\ref{lem:choose_q_middle} to show that
    $u_{i 1}$ and $u_{k l}$ commute:
    \begin{table}[H]
        \begin{tabularx}{0.2\linewidth}[t]{|*{4}{X|}}
            \multicolumn{4}{c}{Lemma~\ref{lem:choose_q_middle}}\\
            \hline
            $j$ &$l$ &  $p$ &  $q$\\
             \hline
            $1$&$3$ & $5$ & $2$\\
               & & $6$ & $12$\\
               && $9$&  $12$\\
               &&$10$& $12$\\
               && $12$& $4$\\
               &$10$ & $3$ & $1$\\
               & & $4$ & $2$\\
               & &$7$&$2$\\
               &&$8$& $2$\\
               && $12$& $4$\\
               \hline
        \end{tabularx}
        \quad \quad
        \begin{tabularx}{0.2\linewidth}[t]{|*{4}{X|}}
            \multicolumn{4}{c}{Lemma~\ref{lem:choose_q_middle}}\\
            \hline
            $j$ &$l$ &  $p$ &  $q$\\
             \hline
               $1$&$11$& $2$& $10$\\
               && $5$& $10$\\
               && $6$& $10$\\
               &&$8$& $2$\\
               && $9$& $8$\\
               &$12$& $4$& $6$\\
               && $7$& $2$\\
               && $8$& $2$\\
               &&$11$& $1$\\
               \hline
        \end{tabularx}
    \end{table}
    We will now take a look at the rest of the vertices in distance $2$ to the vertex  $1$, i.e.  $l \in \{4, 7, 8\} $.
    To simplify, we will again first apply Lemma~\ref{lem:choose_q_middle}:
    \begin{table}[H]
        \begin{tabularx}{0.2\linewidth}[t]{|*{4}{X|}}
            \multicolumn{4}{c}{Lemma~\ref{lem:choose_q_middle}}\\
            \hline
            $j$ &$l$ &  $p$ &  $q$\\
             \hline
            $1$&$4$ & $2$ & $5$\\
               & & $6$ & $7$\\
               && $7$&  $2$\\
               &$7$ & $4$ & $2$\\
               & & $5$ & $4$\\
               & &$9$&$6$\\
               &$8$&$6$& $9$\\
               && $10$& $11$\\
               &&$11$& $1$\\
               \hline
        \end{tabularx}
    \end{table}
    Next, we want to show that $u_{i 1} u_{k 4} u_{i 9} = 0$ and $u_{i 1} u_{k 4} u_{i 10} = 0$ hold.
    For this, we first observe that the permutation 
    \[
        \phi = (1\;4)(2\;3)(5\;12)(6\;11)(7\;10)(8\;9)
    \] 
    is an automorphism of $C_{12}(4, 5^+)$.
    By Lemma~\ref{lem:get_commutation_by_automorphism} with $\phi$ we thus get that  $u_{k 4}$ and $u_{r 2}$ commute.
    We can thus apply Lemma~\ref{lem:monomial_zero} to both monomials with $q:=2$ to get the desired result
    and see that  $u_{i 1}$ and $u_{k 4}$ commute.

    For $l = 7$, we want to show that  $u_{i 1} u_{k 7} u_{i 2} = 0 = u_{i 1} u_{k 7} u_{i 10}$.
    Similarly to above, we note that 
    \[
        \sigma = (1\;7)(2\;8)(3\;9)(4\;10)(5\;11)(6\;12)
    \] 
    is an automorphism of $C_{12}(4, 5^+)$ and thus by Lemma~\ref{lem:get_commutation_by_automorphism} we get
    commutation of $u_{k 7}$ and $u_{r 9}$. Then, by Lemma~\ref{lem:monomial_zero} with $q := 9$, we get the 
    desired result and thus  $u_{i 1}$ and $u_{ k 7 }$ commute.

    For $l = 8$, we need to show that  $u_{i 1} u_{k 8} u_{i 2} = 0 = u_{i 1} u_{k 8} u_{i 5}$.
    In this case, we use the automorphism of $C_{12}(4, 5^+)$
    \[
        \tau = (1\;8)(2\;7)(3\;6)(4\;5)(9\;12)(10\;11)
    .\] 
    Using this with Lemma~\ref{lem:get_commutation_by_automorphism}, we see that $u_{k 8}$ and $u_{i 6}$ commute
    and can thus apply Lemma~\ref{lem:monomial_zero} with $q := 6$ to get the desired outcome.
    We thus get that $u_{i 1}$ and $u_{k 8}$ commute and with this, all generators 
    $u_{ij}$ and $u_{kl}$ commute where $d(j, l) = 2$.

    For the remaining vertices $l$ with $d(l, 1) = 1$, we now apply Lemma~\ref{lem:choose_q_right} and, 
    where necessary, Lemma~\ref{lem:choose_q_middle} to see that the remaining monomial also vanishes:
    \begin{table}[H]
        \begin{tabularx}{0.26\linewidth}[t]{|*{3}{X|}l|}
            \multicolumn{4}{c}{Lemma~\ref{lem:choose_q_right}}\\
            \hline
            $j$ &$l$ &  $q$ &  $p$\\
             \hline
              $1$ &$5$& $12$ & $\{1, 4\} $\\
               &$6$& $12$ & $\{1, 7\} $ \\
               &$9$& $12$ &$\{1, 8\} $\\
             &$2$ & $12$ & $\{1\} $\\
               \hline
        \end{tabularx}
        \quad \quad
        \begin{tabularx}{0.2\linewidth}[t]{|*{4}{X|}}
            \multicolumn{4}{c}{Lemma~\ref{lem:choose_q_middle}}\\
            \hline
            $j$ &$l$ &  $p$ &  $q$\\
             \hline
            $1$&$5$ & $4$ & $6$\\
               & $6$& $7$ & $2$\\
               &$9$& $8$&  $2$\\
               \hline
        \end{tabularx}
    \end{table}
    We thus see that all generators commute and we get $\QBan(C_{12}(4, 5^+)) = G_{aut}(C_{12}(4, 5^+))$.
\end{proof}

\newpage
\section{Computing Non-Commutative Groebner bases in OSCAR}

\label{appendix:computing_non_commutative_groebner_bases_in_oscar}
\smallskip
\begin{center} by \textsc{Julien Schanz and Daniel Schultz} \end{center}
In this section, we give some details of our implementation of the non-commutative Buchberger algorithm 
in the OSCAR\cite{oscar} package in the Julia~\cite{julia} programming language and how to use it
when computing quantum symmetries of graphs.

We used the basic non-commutative Buchberger algorithm that was described in the PhD thesis of Xingqiang Xiu~\cite{xiu_2012}
together with some simple cases of the optimization strategies described there connected to the removal of redundant
obstructions.

We moreover used the selection strategy described in the PhD thesis of Keller~\cite{keller_1997} to select the next
obstruction that will be considered in the Buchberger algorithm.

When using the Groebner bases for quantum symmetries, we can not include the involution. Given a graph $\Gamma$ on $n$ vertices, 
we therefore compute the Groebner basis of the following ideal:
\[
    R(\Gamma) := \langle u_{ij}^2 = u_{ij}, \sum_{k} u_{ik} = \sum_{k} = u_{ki}, Au = uA \rangle
.\] 
Using this ideal, we can define the $\C$-algebra
 \[
     A^+(\Gamma) = \langle u_{ij} | 1 \le i, j \le n \rangle / R(\Gamma)
.\] 
We then have the following implication:
\begin{lem}
    Let $\Gamma$ be a graph on $n$ vertices and let  $a$ and  $b$ be polynomials in the letters $u_{ij}$ for $1 \le i, j \le n$.
    If $a = b$ holds in  $A^+(\Gamma)$, then  $a = b$ also holds in  $\QBan(\Gamma)$.

    We moreover have that if $A^+(\Gamma)$ is commutative then so is  $\QBan(\Gamma)$.
\end{lem}
\begin{proof}
    Since we have an algebra homomorphism given by
    \begin{align*}
        \phi: A^+(\Gamma) &\to \QBan(\Gamma)\\
        u_{ij} &\mapsto u_{ij}
    \end{align*}
    whose image lies dense in $\QBan(\Gamma)$, the statements follow.
\end{proof}
Note however, that the reverse implication does not hold, i.e. we can not conclude that $\QBan(\Gamma)$ is non-commutative
if  $A^+(\Gamma)$ is non-commutative.

\end{document}